\documentclass[12pt]{article}
\usepackage{amsmath,amsthm,amssymb}
\usepackage{graphicx}

\newcommand{\Ch}{{\cal C}}

\newcommand{\N}{{\rm l}\!{\rm N}}

\newcommand{\ind}{{\rm ind}\,}

\newtheorem{theorem}{Theorem}[section]
\newtheorem{deflem}[theorem]{Definition and Lemma}
\newtheorem{definition}[theorem]{Definition}
\newtheorem{prop}[theorem]{Proposition}

\newtheorem{cor}[theorem]{Corollary}

\newtheorem{lemma}[theorem]{Lemma}

\renewenvironment{proof}{{\bf Proof:}}{\mbox{}\hfill $\Box$}

\theoremstyle{definition}
\date{}

\title{Local cyclic homology of group Banach algebras I: 
 Hyperbolic groups}

\author{Michael Puschnigg}
\date{}

\begin{document}

\maketitle

\section{Introduction}
 
Cyclic homology of Banach algebras has been invented by Alain Connes in the early eighties \cite{Co},\cite{Co1}. 
It is supposed to provide an explicitely calculabe approximation of topological $K$-theory $K_*^{top}$. This is of interest because the topological $K$-functor is known on the one hand to provide a lot of information about Banach- and in particular $C^*$-algebras. On the other hand it is, despite its simple and elementary definition, notoriously difficult to calculate.\\
\\
Cyclic homology theories $HC_*$ appear usually as derived functors and are thus accesible to calculation by classical homological algebra. They are the target of a natural transformation $ch:K_*^{top}\to HC_*$, called the Chern-Connes character. One would like the Chern-Connes character to be as close to a natural equivalence as possible. The cyclic theory which comes closest to this goal is local cyclic cohomology $HC^{loc}_*$ \cite{Pu1}. This is a ${\mathbb Z}/2{\mathbb Z}$-graded bifunctor from the category of Banach algebras to the category of complex vector spaces. Its formal properties are quite similar to those of topological $K$-Theory: continuous homotopy invariance, stability and more generally topological Morita-invariance, excision and compatibility with topological direct limits (on the subcategory of Banach algebras with approximation property) \cite{Pu1}. Moreover, there exists a bivariant Chern-Connes character 
$ch_{biv}:KK_*(-,-)\to HC^{loc}_*(-,-)$ on Kasparov's bivariant $K$-theory of $C^*$-algebras, 
which is uniquely characterized by its multiplicativity. Its complexification turns out to be an isomorphism on the bootstrap category of separable $C^*$-algebras $KK$-equivalent to commutative ones.
\\
\\
In this paper we calculate bivariant local cyclic cohomology for the Banach convolution algebras of summable functions on word-hyperbolic groups. This considerably extends our knowledge about local cyclic cohomology as up to now only the case of group Banach algebras of free and polycyclic groups was known. \\
We develop a setup for studying local cyclic cohomology of group Banach algebras, which is inspired by Burghelea's \cite{Bu} and Nistor's \cite{Ni} well known analysis of  periodic cyclic (co)homology of group rings. Our main tool, which allows to carry out the calculation explicitely for hyperbolic groups, is the efficient subdivision algorithm for hyperbolic metric spaces introduced by Bader, Furman and Sauer \cite{BFS}.\\
\\
In order to present our results in more detail we recall some facts 
about cyclic cohomology. The cyclic bicomplex of an associative complex algebra $A$ is the 
${\mathbb Z}/2{\mathbb Z}$-graded chain complex $(\Omega^*A,b+B)$ of algebraic differential forms 
(graded by parity) over $A$. The differential is the sum of the Hochschild differential $b$ and Connes' cyclic differential $B$. The closely related Hochschild complex $(\Omega^* A,b)$ can be made explicit in the case of a group ring. It equals the Bar-complex of the group with coefficients in the adjoint representation, which brings group (co)homology into play. Moreover the adjoint representation decomposes as direct sum of subrepresentations given by the linear span of the conjugacy classes in the group.
$$
(\Omega^*{\mathbb C}\Gamma,b)\,\simeq\, C_*^{Bar}(\Gamma,Ad({\mathbb C}\Gamma))\,\simeq\,
\underset{\langle g\rangle}{\bigoplus}\,C_*^{Bar}(\Gamma,Ad({\mathbb C}\langle g\rangle)).
\eqno(1.1)
$$
This decomposition carries over to the cyclic bicomplex. \\
The cyclic chain complex of a Banach- or abstract algebra is actually acyclic, but the homology of its completion with respect to various natural adic or locally convex topologies gives rise to interesting cyclic homology theories like periodic, entire or analytic cyclic homology. Bivariant local cyclic cohomology of a pair of algebras equals the group of morphisms between their analytic cyclic bicomplexes in a suitable derived category \cite{Co},\cite{Co2},\cite{Pu1},\cite{Me}. The homogeneous decomposition (0.1) carries over to the analytic cyclic bicomplex of 
$\ell^1(\Gamma)$ (this is not the case for other Banach completions of ${\mathbb C}\Gamma$). \\
For simplicity we restrict our attention for a moment to the contribution\\ $(\Omega^*{\mathbb C}\Gamma,b+B)_{\langle e\rangle}$ of the conjugacy class of the unit. For an abstract group its contribution to periodic cyclic homology equals group homology with complex coefficients: 
$$
HP_*({\mathbb C}\Gamma)_{\langle e\rangle}\,\simeq\,H_*(\Gamma,{\mathbb C}).
\eqno(1.2)
$$
The passage from the group ring to the group Banach algebra corresponds then to the passage from ordinary group homology to $\ell^1$-homology (or dually to bounded cohomlogy), which is quite inaccessibe. \\
It is here that it becomes crucial to work with local cyclic cohomology. If $\ell_S(-)$ denotes a word length function on $\Gamma$, then the group Banach algebra appears as the topological (not the algebraic) direct limit of the convolution algebras of summable functions of 
exponential decay (w.r.t. the word length). Local cyclic (co)homology is compatible with topological direct limits, so that
$$
\underset{\lambda\to 1^+}{\lim}\,\,HC_*^{loc}(\ell^1_\lambda(\Gamma,S))\,\overset{\simeq}{\longrightarrow}\,HC_*^{loc}(\ell^1(\Gamma)).
\eqno(1.3)
$$
where 
$$
\ell^1_\lambda(\Gamma,S)\,=\,\{\,\underset{g\in\Gamma}{\sum}\,a_g u_g,\,
\underset{g\in\Gamma}{\sum}\,\vert a_g\vert\,\lambda^{\ell_S(g)}<\infty\,\}.
\eqno(1.4)
$$
The individual groups on the left hand side of (0.3) are still inaccessible at this point but the transition maps are amenable to study. For a hyperbolic group the Rips complex
spanned by Bar-simplices of uniformly bounded (large) diameter is a deformation retract of the full Bar complex. A particularly nice retraction is provided by the subdivision technique of Bader, Furman and Sauer \cite{BFS}.

\newpage

The Rips complex is finite dimensional, and thus complete in any norm, so that we may deduce
$$
HC_*^{loc}(\ell^1(\Gamma))_{\langle e\rangle}\,\simeq\,\underset{\lambda\to 1^+}{\lim}\,\,HC_*^{loc}(\ell^1_\lambda(\Gamma,S))_{\langle e\rangle}\,\simeq\,H_*(\Gamma,{\mathbb C}).
\eqno(1.5)
$$
Getting back to the cyclic bicomplex as whole, we deduce from a similar analysis of the twisted Bar-complex that for a hyperbolic group 
$$
Fil_N^{Hodge}HC_*^{loc}(\ell^1(\Gamma))=0,\,\,\,N>>0
\eqno(1.6)
$$ 
where the Hodge filtration of the cyclic complexes is given by the subcomplexes generated by differential forms of degree bigger than $N$.\\
\\
From now on one has to consider the contributions of the various conjugacy classes $\langle g\rangle$  of $\Gamma$ separately from each other. Because $\langle g\rangle\,\simeq\,\Gamma/Z(g)$ as $\Gamma$-spaces, the structure of the centralizer $Z(g)$ plays a crucial role in the calculation of cyclic cohomology \cite{Bu},\cite{Ni}. For hyperbolic groups this structure is well understood. If $g\in\Gamma$ is an element of infinite order, then $Z(g)$ is virtually cyclic, whereas the centralizer of a torsion element is itself a hyperbolic group, whose word metric is quasiisometric to the one inherited from $\Gamma$. Moreover, there are only finitely many conjugacy classes of torsion elements. A quite detailed knowledge of the metric structure of $\Gamma$ is then required to deduce
$$
HC_*^{loc}(\ell^1(\Gamma))_{\langle g\rangle}\,\simeq\,
\begin{cases}
H_*(Z(g),{\mathbb C}) & \vert g\vert<\infty, \\
0 & \vert g\vert=\infty. \\
\end{cases}
\eqno(1.7)
$$
This coincides with the results of Burghelea and Nistor for the periodic cyclic homology of group rings. As all our norm estimates are uniform, i.e. independent of the individual choice of a conjugacy class, we obtain finally 
\begin{theorem}
Let $\Gamma$ and $\Gamma'$ be word hyperbolic groups. Then the bivariant local cyclic cohomology of their group Banach algebras equals 
$$
HC_*^{loc}(\ell^1(\Gamma),\ell^1(\Gamma'))\,\simeq\,Hom_{\mathbb C}(H_*(\Gamma,Ad({\mathbb C}\Gamma_{tors})),H_*(\Gamma',Ad({\mathbb C}\Gamma_{tors}')))
\eqno(1.8)
$$
where ${\mathbb C}\Gamma_{tors}$ denotes the linear span of the set of torsion elements in $\Gamma$, acted on by the adjoint action of $\Gamma$.
\end{theorem}
Taking $\Gamma=1$ and $\Gamma'=1$, respectively, yields the local cyclic homology groups  and the local cyclic cohomology groups of $\ell^1(\Gamma))$.
A slightly more general version of this as well as some partial results about periodic and entire cyclic cohomology of group Banach algebras can be found in the final section 7 of this paper. It should be noted that the isomorphism (0.7) holds for any unconditional Banach algebra over ${\mathbb C}\Gamma$, where a homogeneous decomposition does not hold on the level of complexes but only in the derived category \cite{Pu5}.

\newpage

Let us finally compare our calculation of local cyclic cohomology with the known results about the $K$-theory of group Banach algebras. In his thesis, 
Vincent Lafforgue showed among other things that hyperbolic groups satisfy the Bost-conjecture, i.e. the $\ell^1$-assembly map
$$
K_*^\Gamma(\underline{E\Gamma})\,\,\,\overset{\simeq}{\longrightarrow}\,\,\,K_*(\ell^1(\Gamma))
\eqno(1.9)
$$
is an isomorphism of abelian groups. Here $\underline{E\Gamma}$ denotes a universal proper $\Gamma$-space, which in the case of a hyperbolic group may be taken to be the geometric realization of a Rips-complex of $\Gamma$.
Applying the Chern-Connes character with values in local cyclic cohomology yields then the commutative diagram
$$
\begin{array}{ccc}
K_*^\Gamma(\underline{E\Gamma}) & \longrightarrow & K_*(\ell^1(\Gamma))\\
& & \\
ch\downarrow & & \downarrow ch \\
 & & \\
H_*(\Gamma,Ad({\mathbb C}\Gamma_{tors})) & \longrightarrow & HC_*^{loc}(\ell^1(\Gamma)) \\
\end{array}
\eqno(1.10)
$$
Our main result implies that the lower horizontal arrow is an isomorphism. 
This checks with Lafforgue's result and shows that the Chern character map
$$
ch:\,K_*(\ell^1(\Gamma))\otimes_{\mathbb Z}{\mathbb C}\,\overset{\simeq}{\longrightarrow}\,HC_*^{loc}(\ell^1(\Gamma))
\eqno(1.11)
$$
is indeed an isomorphism. Note however that Lafforgue's result does not imply ours because there is no Chern-Connes character on Lafforgue's bivariant Banach $K$-theory. Nor does our result imply Lafforgue's.\\
\\ 
The content of the paper is as follows. In section 2 we recall the definition of various cyclic chain complexes and review Nistor's analysis of the cyclic complexes of group rings. Section 3 collects the facts we need about hyperbolic metric spaces and hyperbolic groups. They are used in section 4 to study the Bar-complex and in section 5 to bound the norm of various linear operators on it. In section 6 we extend Nistor's approach to the analytic cyclic bicomplex of group Banach algebras. Our main results are presented in section 7.\\
\\
Section 5 of the present paper generalizes the calculations of our paper \cite{Pu3}, where a partial analysis of the local cyclic homology of the group Banach algebra was the key step in  
verifying the Kadison-Kaplansky idempotent conjecture for word-hyperbolic groups. The proof of Proposition 2.9 of \cite{Pu3} contains however a mistake (the inequality on page 165, line 2 is wrong), which spoils the crucial Corollary 2.10. This corollary is however an immediate consequence of Corollary 5.3 of the present paper, so that all the results of \cite{Pu3} except Proposition 2.9 hold as stated.\\
\\
An essential part of this work was done during stays at the Mathematisches Institut der Universit\"at Heidelberg, the Schr\"odinger Institut f\"ur Mathematik ind Physik in Wien and the Banach Center for Mathematics at the Polish Academy of Sciences in Warsaw. I thank all these institutions for their hospitality and their support.

\newpage

\section{Cyclic homology of group algebras}

In this section we establish our notation and recall some basic facts about cyclic homology.

\subsection{Hochschild homology \cite{Co}, \cite{Co1}}

The universal differential graded algebra 
$\Omega A$ over an associative complex algebra 
$A$ is given by the space
of algebraic differential forms
$$
\Omega A\,=\,\underset{n\geq 0}{\bigoplus}\,\Omega^nA
$$
$$
\begin{array}{lcr}
\Omega^nA & := & \widetilde{A}\otimes A^{\otimes n},\\
 & & \\
a^0da^1\ldots da^n & \leftrightarrow & a^0\otimes a^1\otimes\ldots\otimes a^n,
\end{array}
\eqno(2.1)
$$
where $\widetilde{A}$ is the algebra obtained from $A$ by adjoining a unit. The differential equals
$$
d:\,\Omega^nA\,\to\,\Omega^{n+1}A,\,\,\,a^0da^1\ldots da^n\mapsto
 da^0da^1\ldots da^n.
\eqno(2.2)
$$
There are two basic linear operators on the space of algebraic differential forms:\\ the {\bf Hochschild boundary operator} 
$$
\begin{array}{cccc}
b: & \Omega^*A & \longrightarrow & \Omega^{*-1}A, \\
 & & & \\
 & a^0da^1\ldots da^{n-1}da^n & \mapsto & 
 (-1)^{n-1}[a^0da^1\ldots da^{n-1},\,a^n],
\end{array}
\eqno(2.3)
$$
and {\bf Connes' operator} 
$$
\begin{array}{cccc}
B: & \Omega^*A & \longrightarrow & \Omega^{*+1}A \\
 & & & \\
 & a^0da^1\ldots da^n & \mapsto &
 \underset{i=0}{\overset{n}{\sum}}\,(-1)^{in}da^i\ldots da^nda^0\ldots da^{i-1}.
\end{array}
\eqno(2.4)
$$
They satisfy the relations
$$\begin{array}{cccc}
b^2\,=\,0, & B^2\,=\,0, & \text{and} & bB\,+\,Bb\,=\,0.
\end{array}
\eqno(2.5)
$$
The {\bf Hochschild complex} of $A$ is 
$$
C_*(A)\,:=\,(\Omega^*A,\,b)
\eqno(2.6)
$$
The canonical identification
$\Omega^nA\simeq\widetilde{A}\otimes A^{\otimes n}\,\simeq\,
A^{\otimes (n+1)}\oplus A^{\otimes n}$ 
allows to view the Hochschild-complex as the total complex associated to a bicomplex with $C_{pq}(A)\,=\,A^{\otimes(q+1)}$ if $p=0,1,q\geq 0$ and $C_{pq}(A)=0$ otherwise. Its second column is contractible if $A$ is unital, so that the Hochschild complex $C_*(A)$ becomes chain homotopy equivalent to the first column of the bicomplex, given by
$$
\widetilde{C}_*(A)\,=\,(A^{\otimes(*+1)},b),
\eqno(2.7)
$$
$$
b(a^0\otimes\ldots\otimes a^n)\,=\,
\underset{i=0}{\overset{n-1}{\sum}}\,(-1)^i
a^0\otimes\ldots\otimes a^ia^{i+1}\otimes\ldots\otimes a^n
\,+\,(-1)^na^na^0\otimes\ldots\otimes a^{n-1}.
$$

The homology $HH_*(A,A):=H_*(C_*(A))$ of the Hochschild complex is called the {\bf Hochschild homology} and the cohomology of the dual complex $HH^*(A,A^*):=H^*(Hom_{{\mathbb C}-lin}(C_*(A),{\mathbb C}))$ is called the {\bf Hochschild cohomology} of $A$ (with values in the dual 
bimodule $A^*=Hom_{{\mathbb C}-lin}(A,{\mathbb C})$).
\\
Consider the bifunctors
$$
\otimes:\,_A Mod_A\times_A Mod_A\,\overset{\simeq}{\longrightarrow}\,
Mod_{A\otimes A^{op}}\times_{A\otimes A^{op}}Mod\,
\overset{\otimes_{A\otimes A^{op}}}{\longrightarrow}\,{\mathbb C}-Mod
$$
and
$$
Hom:\,_A Mod_A\times_A Mod_A\,\overset{\simeq}{\longrightarrow}\,
_{A\otimes A^{op}}Mod\times_{A\otimes A^{op}}Mod\,
\overset{Hom_{A\otimes A^{op}-Mod}}{\longrightarrow}\,{\mathbb C}-Mod
$$
on the abelian category of $A$-bimodules. (Here $A^{op}$ denotes the opposite algebra of $A$ and the left hand functors identify the category of $A$-bimodules with the category of $A\otimes A^{op}$-left-(or right)-modules.) There are canonical isomorphisms
$$
\begin{array}{cc}
HH_*(A,A)\,\simeq\,Tor_*^{\widetilde{A}\otimes\widetilde{A}^{opp}}(A,A),
&
HH_*(A,A^*)\,\simeq\,Ext^*_{\widetilde{A}\otimes\widetilde{A}^{opp}}(A,A^*),
\end{array}
\eqno(2.8)
$$ 
which allow to interprete Hochschild (co)homology as derived functors. In particular, these groups can be computed 
via projective resolutions of $A$, viewed as bimodule over itself.

\subsection{Cyclic homology \cite{Co}, \cite{Co1}}

The {\bf cyclic complex} $C^\lambda_*(A)$ of an associative complex algebra $A$ is
the ${\mathbb Z}_+$-graded chain complex $(C_*^\lambda(A),\partial)$ given by 
$$
C_n^\lambda(A)\,=\,\underset{k\geq 0}{\bigoplus}\,\Omega^{n-2k}A.
\eqno(2.9)
$$
Its differential $\partial:C^\lambda_n(A)\to C^\lambda_{n-1}(A)$ in degree $n$ is given by $\partial=b+B$ on 
$\underset{k>0}{\bigoplus}\,\Omega^{n-2k}A$ and by 
$\partial=b$ on $\Omega^n A$.
The homology of the cyclic complex is the {\bf cyclic homology} $HC_*(A)$ of $A$. The cohomology of the dual complex is the 
{\bf cyclic cohomology} $HC^*(A)$ of $A$.
There is a canonical natural extension of chain complexes
$$
\begin{array}{ccccccccc}
0 & \to & C_*(A) & \overset{I}{\longrightarrow} & 
C^\lambda_*(A) & \overset{S}{\longrightarrow} & 
C^\lambda_*(A)[-2] & \to & 0 \\
\end{array}
\eqno(2.10)
$$
The associated long exact homology sequence
$$
\begin{array}{ccccccccc}
\longrightarrow & HH_*(A) & \overset{I}{\longrightarrow} & HC_*(A) & 
\overset{S}{\longrightarrow} & HC_{*-2}(A) &  \overset{B}{\longrightarrow} & HH_{*-1}(A) & \longrightarrow
\end{array}
\eqno(2.11)
$$
is known as {\bf Connes' ISB-sequence}. Its connecting map is given by Connes' $B$-operator. It is often used as a tool to 
compute cyclic homology from the Hochschild homology groups.\\
\\
The canonical identification
$\Omega^nA\simeq\widetilde{A}\otimes A^{\otimes n}\,\simeq\,
A^{\otimes (n+1)}\oplus A^{\otimes n}$ 
allows to view also the cyclic complex as the total complex associated to a bicomplex with\\ $C^\lambda_{pq}(A)\,=\,A^{\otimes(q+1)},\,p,q\geq 0$. Its lines are contractible in positive degrees and the zero'th homology of the $n$-th line equals the space of coinvariants of $A^{\otimes(n+1)}$ under the {\bf cyclic operator} 
$$
\begin{array}{cccc}
T_*: & A^{\otimes(*+1)} & \longrightarrow & A^{\otimes(*+1)}\\
 & & & \\
 & a^0\otimes a^1\otimes\ldots\otimes a^n & \mapsto &
 (-1)^na^n\otimes a^0\otimes\ldots\otimes a^{n-1}.
\end{array}
\eqno(2.12)
$$
Note that $T_n^{n+1}=Id$ on $A^{\otimes(n+1)}$, so that $A^{\otimes(n+1)}$ becomes a ${\mathbb Z}/(n+1){\mathbb Z}$-module under the action of $T_n$ for all $n\geq 0$.\\ 
The cyclic complex $C^\lambda_*(A)$ is therefore chain homotopy equivalent
to the complex
$$
\widetilde{C}^\lambda_*(A)\,=\,(A^{\otimes(*+1)}/(Id-T_*)A^{\otimes(*+1)},\,b).
\eqno(2.13)
$$
For unital $A$ the canonical chain map $I:\widetilde{C}_*(A)\to \widetilde{C}^\lambda_*(A)$ from the Hochschild- to the cyclic complex appears in this picture as the passage
$$
I:\,A^{\otimes(*+1)}\,\longrightarrow\,
A^{\otimes(*+1)}/(Id-T_*)A^{\otimes(*+1)}
\eqno(2.14)
$$
from the tensor algebra over $A$ to its space of coinvariants under the action of the cyclic operator.

\subsection{Periodic cyclic homology \cite{Co}, \cite{Co1}}

The inverse limit under the $S$-operation of countably many copies of the cyclic complex yields the {\bf periodic cyclic bicomplex} 
$$
\widehat{CC}_*(A)\,=\,(\underset{n\geq 0}{\prod}\,\Omega^{n}A,\,b+B).
\eqno(2.15)
$$
It is not graded by the integers but only ${\mathbb Z}/2{\mathbb Z}$ graded (by the parity of forms). 
Its homology $HP_*(A)$ is the {\bf periodic cyclic homology}, 
the cohomology $HP^*(A)$ of the dual chain complex is the {\bf periodic cyclic cohomology} of $A$. It defines a ${\mathbb Z}/2{\mathbb Z}$-graded covariant(contravariant) 
functor on the category of complex algebras, which is Morita-invariant, homotopy-invariant w.r.t. polynomial homotopies, and 
satisfies excision, i.e. an extension of algebras gives rise to a six-term exact sequence of periodic cyclic (co)homology groups.\\
The periodic cyclic bicomplex appears as the adic completion of the naturally contractible 
${\mathbb Z}/2{\mathbb Z}$-graded cyclic bicomplex
$$
CC_*(A)\,=\,(\Omega^{ev}A\oplus\Omega^{odd}A,\,b+B),
\eqno(2.16)
$$
with respect to the {\bf Hodge filtration} defined by the subcomplexes
$$
Fil^n_{Hodge}CC_*(A)\,=\,(b\Omega^nA\oplus \Omega^{\geq n}A,
\,b+B)
\eqno(2.17)
$$
generated by the differential forms of degree at least $n$.
There are various cyclic homology theories attached to topological algebras. They are usually defined in terms of more sophisticated completions of the cyclic bicomplex $CC_*(A)$ of $A$.

\subsection{Reduced complexes}
Suppose that $A$ is an algebra with unit $1$. The two sided homogeneous differential ideal 
$$
(d1)\,=\,\Omega^*A\cdot d1\cdot\Omega^*A 
\eqno(2.18)
$$ 
of $\Omega ^*A$ generated by the differential of the unit is called the ideal of {\bf degenerate differental forms}. The quotient algebra 
$$
\overline{\Omega}^*A=\Omega^*A/(d1)
\eqno(2.19)
$$  
is called the algebra of {\bf reduced differential forms}. The space of degenerate differential forms is stable under the Hochschild operator $b$ and Connes' operator $B$, so that they descend both to operators on the space of reduced differential forms.
One may therefore introduce the {reduced Hochschild complex}
$$
{\overline C}_*(A)\,=\,(\overline{\Omega}^*A,b),
\eqno(2.20)
$$
the reduced cyclic complex 
$$
\overline{C}_*^\lambda(A)\,=\,(\underset{k\geq 0}{\bigoplus}\,\overline{\Omega}^{n-2k}A,\partial),
\eqno(2.21)
$$
and the reduced periodic cyclic complex
$$
\widehat{\overline{CC}}_*(A)\,=\,(\underset{n\geq 0}{\prod}\,\overline{\Omega}^{n}A,\,b+B).
\eqno(2.22)
$$
The canonical projection $\Omega^*A\to\overline{\Omega}^*A$ gives rise to natural chain homotopy equivalences 
$$
\begin{array}{ccc}
C_*(A)\overset{\sim}{\to}{\overline C}_*(A), & C_*^\lambda(A)\overset{\sim}{\to}\overline{C}_*^\lambda(A), & \widehat{CC}_*(A)\overset{\sim}{\to}\widehat{\overline{CC}}_*(A).
\end{array}
\eqno(2.23)
$$

\subsection{Group homology and Hochschild homology 
of group rings \cite{Ni}}

Let $\Gamma$ be a group and consider the commutative diagram of functors
$$
\begin{array}{rcccl}
 & & Ad & & \\
 & {\mathbb C}\Gamma-Bimod & \to & {\mathbb C}\Gamma-Mod & \\
 & & & & \\
-\otimes_{{\mathbb C}\Gamma\otimes{\mathbb C}\Gamma^{op}}{\mathbb C}\Gamma & \downarrow & & \downarrow & -\otimes_{{\mathbb C}\Gamma}\mathbb C \\
 & & & & \\
 & {\mathbb C}-Mod & = & {\mathbb C}-Mod & \\
\end{array}
$$
where the upper horizontal arrow associates to a $\Gamma$-bimodule 
 $M$ the same vector space $M$ with the adjoint action and the vertical arrows are given by taking the commutator quotient $[-,{\mathbb C}\Gamma]$ on the left hand side and by taking $\Gamma$-coinvariants on the right hand side, respectively. The functor $Ad$ is exact and sends projective objects to projective objects. This implies that there is a  canonical isomorphism of derived functors
$$
HH_*({\mathbb C}\Gamma,-)\,=\, Tor_*^{{\mathbb C}\Gamma\otimes{\mathbb C}\Gamma^{op}}({\mathbb C}\Gamma,-)\,
\simeq\, H_*(\Gamma,Ad(-))
\eqno(2.24)
$$
which identifies the Hochschild-homology of ${\mathbb C}\Gamma$ with coefficients in a ${\mathbb C}\Gamma$-bimodule $M$ with the group homology with coefficients in $Ad\,M$. Taking $M={\mathbb C}\Gamma$ one obtains a canonical isomorphism
$$
HH_*({\mathbb C}\Gamma,{\mathbb C}\Gamma)\,\simeq\, H_*(\Gamma, Ad({\mathbb C}\Gamma))
\eqno(2.25)
$$
This isomorphism can be described explicitely in terms of the Hochschild complex and a 
standard complex calculating group (co)homology.

\begin{definition}
The {\bf Bar-complex} $\Delta_\bullet(X)$ of a set $X$ is the simplicial set 
with $n$-simplices $\Delta_n(X)\,=\,X^{n+1}$, face maps
$$
\partial_i([x_0,\ldots,x_n])\,=\,[x_0,\ldots,x_{i-1},x_{i+1},\ldots,x_n],
\eqno(2.26)
$$
and degeneracy maps
$$
s_j([x_0,\ldots,x_n])=[x_0,\ldots,x_j,x_j,\ldots,x_n]\,\,\text{for}\,\,0\leq i,j\leq n\in{\mathbb N}.\eqno(2.27)
$$
The {\bf support} of a Bar-simplex is
$
Supp([x_0,\ldots,x_n])\,=\,\{x_0,\ldots,x_n\}\subset X.
$
\end{definition}

Every map of sets $f:X\to Y$ gives rise to a simplicial map 
$$
f_\bullet:\,\Delta_\bullet(X)\to\Delta_\bullet(Y),\,[x_0,\ldots,x_n]\mapsto [f(x_0),\ldots,f(x_n)].
\eqno(2.28)
$$
In particular, every group action on the set $X$ gives rise to a simplicial action on the Bar-complex $\Delta_\bullet(X)$.

\begin{definition}
The {\bf Bar chain complex} $C_*(X,{\mathbb C})$ of a set $X$ is the chain complex with complex coefficients associated to the Bar-complex. Its differentials are as usual given by the alternating sum 
of the linear operators induced by the face maps. The quotient $\overline{C}_*(X,{\mathbb C})$ of the Bar chain complex by the contractible subcomplex spanned by degenerate Bar-simplices is called 
the {\bf reduced Bar chain complex}.
\end{definition}

The augmentation map $C_0(X,{\mathbb C})\to{\mathbb C}$ of the Bar-complex sends any zero simplex to 1. The augmented Bar-complex is contractible (but there is no natural contraction). If $x\in X$ is a base point, then 
$$
\begin{array}{cccc}
s_x: & C_*(X,{\mathbb C}) & \to & C_{*+1}(X,{\mathbb C}) \\
& & & \\
& [x_0,\ldots,x_n] & \mapsto & [x,x_0,\ldots,x_n] \\
\end{array}
\eqno(2.29)
$$
is a contracting homotopy of the augmented Bar-complex:
$$
Id\,=\,\partial\circ s_x\,+\,s_x\circ\partial.
$$
In particular, the homology of the Bar-complex is one dimensional and concentrated in degree zero. 
The augmentation map identifies it canonically with $\mathbb C$.

\begin{lemma}
Let $\varphi_*,\psi_*:C_*(X,{\mathbb C})\to C_*(Y,{\mathbb C})$ be chain maps of Bar complexes
which induce the identity in homology. Then the linear operator
$$
\begin{array}{cccc}
h(\varphi,\psi): & C_*(X,{\mathbb C}) & \to & C_{*+1}(Y,{\mathbb C}) \\
& [x_0,\ldots, x_n] & \mapsto & \underset{i=0}{\overset{n}{\sum}}\,(-1)^i\,[\varphi_i(x_0,\ldots,x_i),
\psi_{n-i}(x_i,\ldots,x_n)] \\
\end{array}
\eqno(2.30)
$$
defines a natural chain homotopy between $\varphi$ and $\psi$:
$$
\psi_*-\varphi_*\,=\,\partial\circ h(\varphi,\psi)\,+\,h(\varphi,\psi)\circ\partial.
\eqno(2.31)
$$
If $\varphi_*$ and $\psi_*$ preserve the degenerate sub complexes, then so does $h(\varphi,\psi)$.
Moreover, if $G$ is a group acting on $X$ and $Y$, and if $\varphi_*$ and $\psi_*$ are $G$-equivariant, then $h(\varphi,\psi)$ is $G$-equivariant as well.
\end{lemma}

\begin{proof}
Let $[x_0,\ldots,x_n]\in\Delta_n(X)$ be a Bar-simplex. One finds 
$$
\partial\circ h(\varphi,\psi)([x_0,\ldots,x_n])\,=\,\partial(\underset{i=0}{\overset{n}{\sum}}\,(-1)^i\,[\varphi(x_0,\ldots,x_i),
\psi(x_i,\ldots,x_n)])
$$
$$
=\underset{i=0}{\overset{n}{\sum}}\,(-1)^i\,[(\partial\circ\varphi)(x_0,\ldots,x_i),
\psi(x_i,\ldots,x_n)]\,-\,\underset{i=0}{\overset{n}{\sum}}\,[\varphi(x_0,\ldots,x_i),
(\partial\circ\psi)(x_i,\ldots,x_n)]
$$
$$
=\underset{i=0}{\overset{n}{\sum}}\,(-1)^i\,[(\varphi\circ\partial)(x_0,\ldots,x_i),
\psi(x_i,\ldots,x_n)]\,-\,\underset{i=0}{\overset{n}{\sum}}\,[\varphi(x_0,\ldots,x_i),
(\psi\circ\partial)(x_i,\ldots,x_n)]
$$
$$
=\underset{i=0}{\overset{n}{\sum}}\underset{j=0}{\overset{i-1}{\sum}}\,(-1)^{i+j}
[\varphi(x_0,\ldots,\check{x}_j,\ldots,x_i),\psi(x_i,\ldots,x_n)]
$$
$$
+\underset{i=0}{\overset{n}{\sum}}\,
[\varphi(x_0,\ldots,x_{i-1}),\psi(x_i,\ldots,x_n)]\,-\,
\underset{i=0}{\overset{n}{\sum}}\,
[\varphi(x_0,\ldots,x_{i}),\psi(x_{i+1},\ldots,x_n)]
$$
$$
+\underset{i=0}{\overset{n}{\sum}}\underset{j=i+1}{\overset{n}{\sum}}\,(-1)^{i+j+1}
[\varphi(x_0,\ldots,x_i),\psi(x_i,\ldots,\check{x}_j,\ldots,x_n)]
$$
$$
=\,-\,\underset{j=0}{\overset{n}{\sum}}\underset{i=j+1}{\overset{n}{\sum}}\,(-1)^{(i-1)+j}
[\varphi(x_0,\ldots,\check{x}_j,\ldots,x_i),\psi(x_i,\ldots,x_n)]\,+\,(\psi_n\,-\,\varphi_n)
$$
$$
-\underset{j=0}{\overset{n}{\sum}}\underset{i=0}{\overset{j-1}{\sum}}\,(-1)^{i+j}
[\varphi(x_0,\ldots,x_i),\psi(x_i,\ldots,\check{x}_j,\ldots,x_n)]
$$
$$
=\,(\psi_n\,-\,\varphi_n\,-\,h(\varphi,\psi)\circ\partial)([x_0,\ldots,x_n]).
$$
\end{proof}

Let the group $\Gamma$ act on $X=\Gamma$ by left translations. The Bar complex $C_*(\Gamma,{\mathbb C})$ defines then a free resolution of the trivial $\Gamma$-module. Consequently the complex
$C_*(\Gamma,M)\,=\,C_*(\Gamma,{\mathbb C})\otimes_{{\mathbb C}}M$ defines a free resolution of any given $\Gamma$-module $M$ and the group homology $H_*(\Gamma,M)$ of $\Gamma$ with coefficients in $M$ can be calculated as the homology of the complex of coinvariants $H_*(C_*(\Gamma,M)_\Gamma,\partial)$ of the Bar-resolution twisted by $M$.
The explicit formulas for the passage between group homology and Hochschild homology in terms of the Bar resolution are given in 

\begin{lemma}
There exists a canonial $\Gamma$-equivariant chain map
$$
\begin{array}{cccc}
p_*:  & C_*(\Gamma,\,Ad({\mathbb C}\Gamma)) & \to & \widetilde{C}_*({\mathbb C}\Gamma), \\
&&& \\
& [g_0,\ldots,g_n;v] & \mapsto & (g_n^{-1}vg_0)d(g_0^{-1}g_1)\ldots d(g_{n-1}^{-1}g_n).\\
\end{array}
\eqno(2.32)
$$
A $\Gamma$-equivariant linear section of $p$ is given by 
$$
\begin{array}{cccc}
q_*: & \widetilde{C}_*({\mathbb C}\Gamma) & \to & C_*(\Gamma,\,Ad({\mathbb C}\Gamma)) \\
&&& \\
& h_0dh_1\ldots dh_n & \mapsto & [h_0,h_0h_1,\ldots,h_0h_1\ldots h_n;h_0h_1\ldots h_n],\\
\end{array}
\eqno(2.33)
$$
The epimorphism $p$ identifies the canonical deformation retract $\widetilde{C}_*({\mathbb C}\Gamma)$ of the Hochschild complex of ${\mathbb C}\Gamma$ with the complex of 
coinvariants $C_*(\Gamma,\,Ad({\mathbb C}\Gamma))_\Gamma$ of the Bar-complex of $\Gamma$, twisted by the adjoint representation $Ad({\mathbb C}\Gamma)$.
The maps $p_*$ and $q_*$ send degenerate Bar-simplices to degenerate differential forms and vice-versa.
\end{lemma}

Following Nistor, one describes the cyclic permutation operator and Connes' operator on algebraic differential forms in terms of lifted operators on the Bar complex, twisted by the adjoint representation.

\begin{definition}.\\
\begin{itemize}
\item[a)]
Let $\widetilde{T}_*$ be the $\Gamma$-equivariant operator on the twisted Bar-complex given by
$$
\begin{array}{cccc}
\widetilde{T}_n: & C_n(\Gamma,Ad({\mathbb C}\Gamma)) & \to & C_n(\Gamma,Ad({\mathbb C}\Gamma)), \\
& & & \\
& [g_0,\ldots,g_n;v] & \mapsto & (-1)^n\,[v^{-1}g_n,g_0,\ldots,g_{n-1};v].\\
\end{array}
\eqno(2.34)
$$
\item[b)] Let $\widetilde{B}$ be the $\Gamma$-equivariant operator on the twisted Bar-complex given by
$$
\begin{array}{cccc}
\widetilde{B}: & C_n(\Gamma,Ad({\mathbb C}\Gamma)) & \to & C_{n+1}(\Gamma,Ad({\mathbb C}\Gamma)), \\
& & & \\
& [g_0,\ldots,g_n;v] & \mapsto & \underset{i=0}{\overset{n}{\sum}}
(-1)^{in}\,[v^{-1}g_i,\ldots,v^{-1}g_n,g_0,\ldots,g_{i};v].\\
\end{array}
\eqno(2.35)
$$
\end{itemize}
\end{definition}
Note that 
$$
\widetilde{T}_n^{n+1}([g_0,\ldots,g_n;v])\,=\,v^{-1}\cdot[g_0,\ldots,g_n;v],
\eqno(2.36)
$$
in degree $n$.
The operators (2.34), (2.35) lift the corresponding operators (2.12), (2.4) on algebraic differential forms in the sense that
$$
T_*\circ p_*\,=\,p_*\circ\widetilde{T}_*:\,
C_*(\Gamma,Ad({\mathbb C}\Gamma))\to \widetilde{C}_*({\mathbb C}\,\Gamma), 
\eqno(2.37)
$$
and
$$
 B\circ\overline{p}_*\,=\,\overline{p}_*\circ\,\widetilde{B}:\,C_*(\Gamma,Ad({\mathbb C}\Gamma))\to \overline{C}_{*+1}({\mathbb C}\,\Gamma)
\eqno(2.38)
$$
respectively, where $\overline{p}_*$ denotes the composition of $p_*$ and the reduction map to the reduced Hochschild complex.

\subsection{The homogeneous decomposition \cite{Bu},\cite{Ni}}

The group ring ${\mathbb C}\Gamma$, viewed as $\Gamma$-module under the adjoint action, decomposes as the direct sum of irreducible submodules
$$
Ad({\mathbb C}\Gamma)\,=\,\underset{\langle v\rangle}{\bigoplus}\,{\mathbb C}\Gamma_{\langle v\rangle}
\eqno(2.39)
$$
given by the linear span of the group elements in a fixed conjugacy class 
$$
\langle v\rangle\,=\,\{gvg^{-1},\,g\in\Gamma\}.
\eqno(2.40)
$$
The Hochschild homology of the group ring, viewed as bimodule over itself, decomposes therefore as
$$
HH_*({\mathbb C}\Gamma,{\mathbb C}\Gamma)\,\simeq\,
H_*(\Gamma,\,Ad({\mathbb C}\Gamma))\,\simeq\,\underset{\langle v\rangle}{\bigoplus}\,H_*(\Gamma,\,{\mathbb C}\Gamma_{\langle v\rangle})\,\simeq\,\underset{\langle v\rangle}{\bigoplus}\,HH_*({\mathbb C}\Gamma,{\mathbb C}\Gamma)_{\langle v\rangle}.
\eqno(2.41)
$$
The underlying decomposition of the Hochschild complex 
$$
C_*({\mathbb C}\Gamma)\,=\,(\Omega^*({\mathbb C}\Gamma),b)
$$ 
is given by
$$
\Omega^*({\mathbb C}\Gamma)\,=\,\underset{\langle v\rangle}{\bigoplus}\,\Omega^*({\mathbb C}\Gamma)_{\langle v\rangle},\,\,\,
\Omega^*({\mathbb C}\Gamma)_{\langle v\rangle}\,=\,Span\{\,g_0dg_1\ldots dg_n,\,g_0g_1\ldots g_n\in\langle v\rangle\}.
\eqno(2.42)
$$
This decomposition is preserved by the cyclic operator 
$T_*:\Omega^*({\mathbb C}\Gamma)\to\Omega^*({\mathbb C}\Gamma)$ and Connes' operator $B:\Omega^*({\mathbb C}\Gamma)\to
\Omega^{*+1}({\mathbb C}\Gamma)$ on the space of algebraic differential forms. 
The cyclic (bi)complexes attached to a group ring 
decompose therefore as
$$
\begin{array}{cc}
C_*^\lambda({\mathbb C}\Gamma)\,=\,\underset{\langle v\rangle}{\bigoplus}\,
C_*^\lambda({\mathbb C}\Gamma)_{\langle v\rangle}, & 
CC_*({\mathbb C}\Gamma)\,=\,\underset{\langle v\rangle}{\bigoplus}\,
CC_*({\mathbb C}\Gamma)_{\langle v\rangle}.
\end{array}
\eqno(2.43)
$$
The canonical deformation retracts $\widetilde{C}_*(A)$ and $\widetilde{C}^\lambda_*(A)$ decompose similarly.\\
The direct sum of the contributions of conjugacy classes of torsion elements is called the {\bf homogeneous part}, the direct sum of the contriburions of the remaining 
conjugacy classes is called the {\bf inhomogeneous part} of the complex under consideration.\\
\\
For $v\in\Gamma$ let $\langle v\rangle$ be its conjugacy class and denote by $Z(v)$ its centralizer 
$$
Z(v)\,=\,\{h\in\Gamma, hv=vh\}.
\eqno(2.44)
$$
The conjugation action of $\Gamma$ on $\langle v\rangle$ 
gives rise to an isomorphism 
$$
\Gamma/Z(v)\,\overset{\simeq}{\longrightarrow}\,\langle v\rangle,\,
\overline{g}\mapsto gvg^{-1}
\eqno(2.45)
$$
of $\Gamma$-sets. 
The restriction of $p_*$ to the subcomplex
$$
C_*(\Gamma,{\mathbb C}v):=C_*(\Gamma,{\mathbb C})\otimes{\mathbb C}v\subset C_*(\Gamma,\,Ad({\mathbb C}\Gamma))
\eqno(2.46)
$$ 
of $Z(v)$-modules defines an epimorphism of complexes
$$
\begin{array}{cccc}
p_{v*}: & C_*(\Gamma,{\mathbb C}v) & \longrightarrow & \widetilde{C}_*({\mathbb C}\Gamma)_{\langle v\rangle} \\
& & & \\
 & [g_0,\ldots, g_n;v] & \mapsto & g_n^{-1}vg_0\otimes g_0^{-1}g_1\otimes\ldots\otimes g_{n-1}^{-1}g_n\\
\end{array}
\eqno(2.47)
$$
which identifies the right hand side of (2.47) with the $Z(v)$-coinvariants of $C_*(\Gamma,{\mathbb C})\otimes{\mathbb C}v$.
Taking homology, one obtains
$$
HH_*({\mathbb C}\Gamma)_{\langle v\rangle}\,\simeq\,H_*(\Gamma,\,{\mathbb C}\langle v\rangle)\,\simeq\,H_*(\Gamma,\,Ind_{Z(v)}^\Gamma\,{\mathbb C})\,\simeq\,
H_*(Z(v),{\mathbb C}).
\eqno(2.48)
$$

\subsection{The contribution of elliptic conjugacy classes}

\begin{deflem}
\begin{itemize}
\item[a)] The linear antisymmetrization operator
$$
\begin{array}{cccc}
\pi_{as}: & C_*(\Gamma,Ad({\mathbb C}\Gamma)) & \to & C_*(\Gamma,Ad({\mathbb C}\Gamma)) \\
 & & & \\
 & [g_0,\ldots,g_n;v] & \mapsto & \frac{1}{(n+1)!}\,\underset{\sigma\in\Sigma_{n+1}}{\sum}\,(-1)^{\epsilon(\sigma)}\,[g_{\sigma(0)},\ldots,g_{\sigma(n)};v] \\
\end{array}
\eqno(2.49)
$$
is a $\Gamma$-equivariant chain endomorphism of the twisted Bar-chain complex compatible with the augmentation.

\item[b)] Suppose that $v\in\Gamma$ is a torsion element and let $U\subset\Gamma$ be the finite cyclic subgroup generated by $v$. Then the linear averaging operator
$$
\begin{array}{cccc}
\pi_{U}: & C_*(\Gamma,\,{\mathbb C}v) & \to & C_*(\Gamma,\,{\mathbb C}v) \\
 & & & \\
 & [g_0,\ldots,g_n;v] & \mapsto & \frac{1}{\vert U\vert^{n+1}}\,
 \underset{h_0,\ldots,h_n\in U}{\sum}\,[h_0g_0,\ldots,h_ng_n;v] \\
\end{array}
\eqno(2.50)
$$
is a $Z(v)$-equivariant chain endomorphism of the twisted Bar-chain complex compatible with the augmentation.

\item[c)] The composite operator
$$
\mu_{v}\,=\,\pi_{as}\circ\pi_{U}:\,C_*(\Gamma,\,{\mathbb C}v)\,\to\,C_*(\Gamma,\,{\mathbb C}v)
\eqno(2.51)
$$
is a $Z(v)$-equivariant chain endomorphism of the twisted Bar-complex compatible with the augmentation. It vanishes on degenerate Bar simplices.
\end{itemize}
\end{deflem}

The $Z(v)$-equivariant chain map $\mu_v$ is compatible with the augmentations and therefore canonically and $Z(v)$-equivariantly chain homotopic to the identity via the homotopy operator $h(\mu_{v},id):\,C_*(\Gamma,\,{\mathbb C}v)\,\to\,C_{*+1}(\Gamma,\,{\mathbb C}v)$ (2.30). The subcomplex generated by degenerate Bar-simplices is stable under $\mu_v$ and $h(\mu_v,id)$.
Taking $Z(v)$-coinvariants and passing to reduced complexes yields a chain map
$\overline{\mu}_{v}:\overline{C}_*({\mathbb C}\Gamma)_{\langle v\rangle}\,
\to\,\overline{C}_*({\mathbb C}\Gamma)_{\langle v\rangle},$ which is canonically chain homotopic to the identity via $\overline{h}(\mu_{v},id):\,\overline{C}_*({\mathbb C}\Gamma)_{\langle v\rangle}\,\to\,\overline{C}_{*+1}({\mathbb C}\Gamma)_{\langle v\rangle}.$\\
\\
Introduce the ${\mathbb Z}/2{\mathbb Z}$-graded periodic 
Hochschild complex of an algebra $A$ as
$$
\widehat{C}_*(A)\,=\,(\underset{n\in{\mathbb N}}{\prod}\,\Omega^nA,\,b)\,=\,(\widehat{\Omega}^*(A),\,b),
\eqno(2.52)
$$
the grading being given by the parity of forms. The Hodge filtration 
$Fil^N_{Hodge}\widehat{C}_*(A),$
$N\in{\mathbb N},$ is again defined by the subcomplexes generated by forms of degree at least $N$.
With these notations understood we have

\begin{prop} The linear operator
$$
\nu_{v}\,=\,id+\overline{h}(\mu_{v},id)\circ B:\,\overline{\Omega}^*({\mathbb C}\Gamma)_{\langle v\rangle}\,\longrightarrow\,\overline{\Omega}^*({\mathbb C}\Gamma)_{\langle v\rangle}
\eqno(2.53)
$$
induces an isomorphism of chain complexes 
$$
\nu_{v}:\,\widehat{\overline{CC}}_*({\mathbb C}\Gamma)_{\langle v\rangle}\,\overset{\simeq}{\longrightarrow}\,
\widehat{\overline{C}}_*({\mathbb C}\Gamma)_{\langle v\rangle},
\eqno(2.54)
$$
between elliptic components of the reduced periodic Hochschild complex and the reduced periodic cyclic bicomplex. It is strictly compatible with Hodge filtrations.
\end{prop}

\begin{proof}
Note that the operator (2.53) preserves Hodge filtrations so that it extends to an operator on the corresponding completions (2.54).
By construction $\mu_{v}\circ B'=0$ on $C_*(\Gamma,\,{\mathbb C}v),$ so that $\overline{\mu}_{v}\circ B=0$ on $\overline{\Omega}^*({\mathbb C}\Gamma)_{\langle v\rangle}$. The identity
$$
b\circ\nu_{v}\,=\,b\circ(id+\overline{h}(\mu_{v},id)\circ B)\,=\,
b+b\circ \overline{h}(\mu_{v},id)\circ B
$$
$$
=\,b+\left(id-\overline{\mu}_{v}-\overline{h}(\mu_{v},id)\circ b\right)\circ B
\,=\,(b+B)-\overline{\mu}_{v}\circ B+\overline{h}(\mu_{v},id)\circ B\circ b
$$
$$
=(id+\overline{h}(\mu_{v},id)\circ B)\circ(b+B)\,=\,\nu_{v}\circ(b+B)
$$
shows that $\nu_{v}:\,\widehat{CC}_*({\mathbb C}\Gamma)_{\langle v\rangle}\,\overset{\simeq}{\longrightarrow}\,
\widehat{C}_*({\mathbb C}\Gamma)_{\langle v\rangle}$ is a chain map. In fact it is an isomorphism, 
its inverse being given by the convergent geometric series
$$
\nu_{v}^{-1}\,=\,(id+\overline{h}(\mu_{v},id)\circ B)^{-1}\,=\,
\underset{k=0}{\overset{\infty}{\sum}}\,(-\overline{h}(\mu_{v},id)\circ B)^k.
$$
\end{proof}

\begin{cor}
Let $\Gamma$ be a group and let $v\in\Gamma$ be a torsion element. Then
$$
HP_*({\mathbb C}\Gamma)_{\langle v\rangle}\,\simeq\,
\underset{n\in{\mathbb N}}{\prod}\,H_{*+2n}(Z(v),{\mathbb C})\,\simeq\,\underset{n\in{\mathbb N}}{\prod}\,H_{*+2n}(\Gamma,\,{\mathbb C}\langle v\rangle),\,\,\,*\in\{0,1\},
\eqno(2.55)
$$
where ${\mathbb C}\langle v\rangle\subset Ad({\mathbb C}\Gamma)$ denotes the $\Gamma$-sub module spanned by the conjugacy class $\langle v\rangle$.
\end{cor}

\subsection{The contribution of hyperbolic conjugacy classes}

For a group element $v$ of infinite order we denote by $v^{\mathbb Z}$ the infinite cyclic subgroup generated by $v$. Let $\widetilde{Z}_n(v)$ be the subgroup of $GL(C_n(\Gamma,{\mathbb C}v))$ generated by the centralizer $Z(v)$ of $v$ and the lifted cyclic operator $\widetilde{T}_n$. It fits into a commutative diagram of groups 
$$
\begin{array}{ccccccccc}
 & & 0 & & 0 & & & & \\
 & & \uparrow & & \uparrow & & & & \\
 & & {\mathbb Z}/(n+1){\mathbb Z} & = & {\mathbb Z}/(n+1){\mathbb Z} & & & & \\
 & & \uparrow & & \uparrow & & & & \\
 0 & \to & \widetilde{T}_n^{\mathbb Z} & \to & \widetilde{Z}_n(v) & \to & N(v) & \to & 0 \\
 & & \uparrow & & \uparrow & & \parallel & & \\
 0 & \to & v^{\mathbb Z} & \to & Z(v) & \to & N(v) & \to & 0 \\
 & & \uparrow & & \uparrow & & & & \\
 & & 0 & & 0 & & & & \\
\end{array}
\eqno(2.56)
$$
with exact rows and columns.\\
Neither the cyclic permutation operator $T_*$ nor its lift $\widetilde{T}_*$ commute with the differentials in the corresponding complexes $C_*({\mathbb C}\Gamma)_{\langle v\rangle}$ and $C_*(\Gamma,\,{\mathbb C}v)$. It is true, however, that the subspace 
$(Id-\widetilde{T}_*)C_*(\Gamma,{\mathbb C}v)$ is a subcomplex of $C_*(\Gamma,{\mathbb C}v).$ Therefore one can form 
a commutative diagram of complexes of coinvariants
$$
\begin{array}{ccc}
C_*(\Gamma,{\mathbb C}v) & \to & C_*(\Gamma,{\mathbb C}v)_{\widetilde{T}_*^{\mathbb Z}} \\

\downarrow & & \downarrow \\

C_*(\Gamma,{\mathbb C}v)_{Z(v)} & \to & C_*(\Gamma,{\mathbb C}v)_{\widetilde{Z}(v)}
\end{array}
\eqno(2.57)
$$
whose left vertical and bottom horizontal arrows can be canonically identified with
$p_{v*}: C_*(\Gamma,{\mathbb C}v)\,\to\, \widetilde{C}_*({\mathbb C}\Gamma)_{\langle v\rangle}$ and 
$I:  \widetilde{C}_*({\mathbb C}\Gamma)_{\langle v\rangle}\to  \widetilde{C}_*^\lambda({\mathbb C}\Gamma)_{\langle v\rangle}$ respectively. This diagram can thus be identified with the commutative diagram of chain maps
$$
\begin{array}{ccccc}
& C_*(\Gamma,{\mathbb C}v) & \overset{I_v}{\longrightarrow} & C_*(\Gamma,{\mathbb C}v)_{\widetilde{T}_*^{\mathbb Z}}  & \\
 & & & & \\
p_{v*} & \downarrow & & \downarrow & p_{v*}^\lambda \\
 & & & & \\
& \widetilde{C}_*({\mathbb C}\Gamma)_{\langle v\rangle} & \underset{I}{\longrightarrow} &  \widetilde{C}_*^\lambda({\mathbb C}\Gamma)_{\langle v\rangle} &  \\
\end{array}
\eqno(2.58)
$$
If the centralizer of $v$ is virtually cyclic, then Nistor has shown that
$$
I_v:C_*(\Gamma,{\mathbb C}v)\to C_*(\Gamma,{\mathbb C}v)_{\widetilde{T}_*^{\mathbb Z}}
\eqno(2.59)
$$ 
is a quasiisomorphism and the epimorphism 
$$
p_v^\lambda:C_*(\Gamma,{\mathbb C}v)_{\widetilde{T}_*^{\mathbb Z}}\to  \widetilde{C}_*^\lambda({\mathbb C}\Gamma)_{\langle v\rangle}
$$ 
splits as map of complexes. Thus $ \widetilde{C}_*^\lambda({\mathbb C}\Gamma)_{\langle v\rangle}$ is acyclic in positive degrees and
$$
HC_*({\mathbb C}\Gamma)_{\langle v\rangle}\,=\,
HC^*({\mathbb C}\Gamma)_{\langle v\rangle}\,=\,
\begin{cases}
{\mathbb C} & *=0, \\
0 & *>0.
\end{cases}
\eqno(2.60)
$$
We want to generalize the results of Nistor to a topological setting. This is done in two steps. Instead of working with quasiisomorphisms and the derived category one has to use the chain homotopy category of complexes. At first all chain maps and chain homotopies coming up have to be given by explicit formulas in order to be able to check their continuity in a second step.\\
\\
Fix a set-theoretic section $\sigma:\langle v\rangle\to\Gamma$ of the quotient map $\Gamma\to\langle v\rangle,\,g\mapsto gvg^{-1}$.
By definition 
$$
g\,=\,\sigma(g)v\sigma(g)^{-1},\,\,\,\forall g\in\langle v\rangle.
\eqno(2.61)
$$ 
\begin{definition}
\begin{itemize}
\item[a)] The cyclic average map 
$$
\begin{array}{cccc}
N_{cyc}: & \widetilde{C}_*^\lambda({\mathbb C}\Gamma) & \longrightarrow & \widetilde{C}_*({\mathbb C}\Gamma) \\
 & & & \\
  & a^0\otimes a^1\otimes\ldots\otimes a^n & \mapsto & 
  \frac{1}{n+1}\underset{i=0}{\overset{n}{\sum}}\,(-1)^{in}
  a^i\otimes\ldots a^n\otimes a^0\otimes\ldots\otimes a^{i-1} \\
\end{array}
\eqno(2.62)
$$
defines a linear section of the projection $I:\widetilde{C}_*({\mathbb C}\Gamma)\to \widetilde{C}_*^\lambda({\mathbb C}\Gamma)$ (see (2.14)) onto the coinvariants under the cyclic operator. It is compatible with the homogeneous decomposition.

\item[b)] The linear map 
$$
\begin{array}{cccc}
\iota_{v,\sigma}: & \widetilde{C}_*({\mathbb C}\Gamma)_{\langle v\rangle} & \longrightarrow & C_*(\Gamma,{\mathbb C}v)\\
 & & & \\
& g_0\otimes\ldots\otimes g_n & \mapsto & \sigma(g_0g_1\ldots g_n)^{-1}\cdot
[g_0,g_0g_1,\ldots,g_0g_1\ldots g_n;g_0g_1\ldots g_n] \\
\end{array}
\eqno(2.63)
$$
is a linear section of $p_{v*}:\,C_*(\Gamma,{\mathbb C}v)\to
\widetilde{C}_*({\mathbb C}\Gamma)_{\langle v\rangle}$ (see (2.47)).

\item[c)] The composition $\iota_{v,\sigma}\circ N_{cyc}$ defines a linear section of $$I\circ p_{v*}:C_*(\Gamma,{\mathbb C}v)\to \widetilde{C}_*^\lambda({\mathbb C}\Gamma)_{\langle v\rangle}$$
\end{itemize}
\end{definition}

Suppose that $v$ is of infinite order. Consider the central extension 
$$
0  \to  v^{\mathbb Z}  \to  Z(v)  \to  N(v)  \to  0 
\eqno(2.64)
$$
and fix a set theoretic section  $\sigma':N(v)\to Z(v)$ of the quotient map. 

\begin{deflem}
Suppose that the group $N(v)$ is finite. Then
the chain map
$$
\begin{array}{cccc}
N_{\sigma'}: & C_*(\Gamma,{\mathbb C}v) & \to & C_*(\Gamma,{\mathbb C}v)\\
& & & \\
& [g_0,\ldots,g_n;v] & \mapsto & \frac{1}{\vert N(v)\vert}\,
\underset{h\in N(v)}{\sum}\,[\sigma'(h)g_0,\ldots,\sigma'(h)g_n;v]
\end{array}
\eqno(2.65)
$$
satisfies
$$
N_{\sigma'}(Ker(I\circ p_{v*}))\,\subset\,Ker(I_v).
$$
\end{deflem}

\begin{proof}
The averaging operator $N_{\sigma'}$ is obviously a $v^{\mathbb Z}$-equivariant chain map. According to (2.58) $p_{v*}^\lambda\circ I_v\,=\,
I\circ p_{v*}$, so that $I_v(Ker(I\circ p_{v*}))\subset Ker(p_{v*}^\lambda)$. The finite group $N(v)$ acts 
on the complex $C_*(\Gamma,{\mathbb C}v)_{\widetilde{T}_*^{\mathbb Z}}$ by chain maps. The canonical projection $\pi_{N(v)}$ onto its subcomplex of $N(v)$-invariants is a chain map as well. By construction 
$\pi_{N(v)}$ annihilates $Ker(p_{v*}^\lambda),$ so that $0\,=\,\pi_{N(v)}(I_v(Ker(I\circ p_{v*})))$
$=\,I_v(N_{\sigma'}(Ker(I\circ p_{v*}))),$ i.e. $N_{\sigma'}(Ker(I\circ p_{v*}))\subset Ker(I_v)$.
\end{proof}

As a consequence we see that the chain map of degree -1
$$
N_{\sigma'}\circ(\partial\circ \iota_{v,\sigma}\circ N_{cyc}-\iota_{v,\sigma}\circ N_{cyc}\circ\partial):C_*^\lambda({\mathbb C}\Gamma)_{\langle v\rangle}\to C_*(\Gamma,{\mathbb C}v)[-1],
$$
which measures the deviation of 
$$
s'_{v,\sigma,\sigma'}\,=\,N_{\sigma'}\circ\iota_{v,\sigma}\circ N_{cyc}:\, \widetilde{C}_*^\lambda({\mathbb C}\Gamma)_{\langle v\rangle}\to C_*(\Gamma,{\mathbb C}v)
$$ 
from being a chain map, takes its values in the subcomplex 
$Ker(I_v)\subset C_*(\Gamma,{\mathbb C}v)$. Nistor has analyzed this complex.

\begin{lemma} \cite{Ni}
The linear map
$$
Id-\widetilde{T}_*:\, C_*(\Gamma,{\mathbb C}v)[-1]\otimes_{\mathbb C}{\mathbb C}(\Gamma) \, \longrightarrow \, C_*(\Gamma,{\mathbb C}v),
\eqno(2.66)
$$
which is given on simplices by
$$
[g_0,\ldots g_{n-1};v]\otimes g_n\, \mapsto\, [g_0,\ldots g_{n-1},g_n;v]\,-(-1)^{n}[v^{-1}g_n, g_0,\ldots,g_{n-1};v]
\eqno(2.67)
$$
is a morphism of complexes of free $Z(v)$-modules. It is injective if $v$ is an element of infinite order in $\Gamma$.
\end{lemma}

\begin{proof}
It is easily checked that $Id-\widetilde{T}_*$ is a $Z(v)$-equivariant chain map. Its kernel coincides with the set of fixed points of the automorphism $\widetilde{T}_*$. The operator $\widetilde{T}_n^{n+1}$ coincides with left multiplication by $v^{-1}$ on $C_n(\Gamma,{\mathbb C}v)\simeq{\mathbb C}(\Gamma^{n+1})$, which has no fixed points except zero if $v\in\Gamma$ is of infinite order. 
\end{proof}

\begin{lemma}
The complex $C_*^{bar}(\Gamma,\,{\mathbb C}v)[-1]\otimes_{\mathbb C}{\mathbb C}(\Gamma)$ is $\Gamma$-equivariantly contractible. An equivariant contracting homotopy is given by
$$
\chi([g_0,\ldots,g_{n-1};v]\otimes g_n)=[g_n,g_0,\ldots,g_{n-1};v]\otimes g_n,\,n\geq 0.
\eqno(2.68)
$$
\end{lemma}

The chain map $Id-\widetilde{T}_*$ fits into an extension of complexes
$$
0\,\to\, C_*(\Gamma,{\mathbb C}v)[-1]\otimes_{\mathbb C}{\mathbb C}(\Gamma)\,\overset{Id-\widetilde{T}_*}{\longrightarrow}\,C_*(\Gamma,{\mathbb C}v) \,\overset{I_v}{\longrightarrow}\,C_*(\Gamma,{\mathbb C}v)_{\widetilde{T}_*^{\mathbb Z}}\,\to\,0
\eqno(2.69)
$$
The linear operator
$$
\chi_v\,=\,(Id-\widetilde{T}_*)\circ \chi\circ(Id-\widetilde{T}_*)^{-1}:\,Ker(I_v)\,\to\,Ker(I_v)[1]
\eqno(2.70)
$$
is therefore a contracting chain homotopy of this complex in positive degree.

\begin{prop}
The linear map
$$
s_{v,\sigma,\sigma'}\,=\,s'_{v,\sigma,\sigma'}\,- \chi_v\circ (\partial\circ s'_{v,\sigma,\sigma'}\,-\,s'_{v,\sigma,\sigma'}\circ\partial)
:\, \widetilde{C}_*^{\lambda}({\mathbb C}\Gamma)_{\langle x\rangle}\,\to\,C_*(\Gamma,{\mathbb C}v) 
\eqno(2.71)
$$
is a morphism of chain complexes splitting 
$$
I\circ p_v:\,C_*(\Gamma, {\mathbb C}v)\to  \widetilde{C}_*^{\lambda}({\mathbb C}\Gamma)_{\langle v\rangle}:\,\,\,
I\circ p_v\circ s_{v,\sigma,\sigma'}\,=\,Id_{ \widetilde{C}_*^\lambda({\mathbb C}\Gamma)_{\langle v\rangle}}.
$$
\end{prop}

\begin{proof}
By construction the image of 
$(\partial\circ s'_{v,\sigma,\sigma'}\,-\,s'_{v,\sigma,\sigma'}\circ\partial)$ is contained in $Ker(I_v)$ so that 
$s_{v,\sigma,\sigma'}$ is well defined. We calculate the composition $I\circ p_v\circ s_{v,\sigma,\sigma'}$. As the image of $\chi_v\circ(\partial\circ s'_{v,\sigma,\sigma'}\,-\,s'_{v,\sigma,\sigma'}\circ\partial)$ is contained in $Ker(I_v)\subset Ker(I\circ p_v)$ 
$$
I\circ p_v \circ s_{v,\sigma,\sigma'}\,=\,I\circ p_v \circ s'_{v,\sigma,\sigma'}\,=\,Id_{C_*^\lambda({\mathbb C}\Gamma)_{\langle v\rangle}}
$$
as claimed. It is straightforward to check that $s_{v,\sigma,\sigma'}$ is a chain map.
\end{proof}

\begin{cor}
Let $v\in\Gamma$ be an element of infinite order and suppose that 
$v^{\mathbb Z}$ is of finite index in its centralizer. 
Then the inclusion 
$$
Fil^N_{Hodge}\widehat{CC}_*({\mathbb C}\Gamma)_{\langle v\rangle}
\,\hookrightarrow\,\widehat{CC}_*({\mathbb C}\Gamma)_{\langle v\rangle}
$$
is a chain homotopy equivalence for every $N\in{\mathbb N}$.
\end{cor}

\begin{proof}
We may suppose $N>0$. According to 2.13 the composition of chain maps
$$
\widetilde{C}_*^{\lambda}({\mathbb C}\Gamma)_{\langle v\rangle}\,\overset{s_{v,\sigma,\sigma'}}{\longrightarrow}\,C_*(\Gamma,{\mathbb C}v)
\,\overset{I\circ p_v}{\longrightarrow}\,\widetilde{C}_*^{\lambda}({\mathbb C}\Gamma)_{\langle v\rangle}
\eqno(2.72)
$$
equals the identity. Consequently the composition
$$
C_*^{\lambda}({\mathbb C}\Gamma)_{\langle v\rangle}\,\longrightarrow\,
\widetilde{C}_*^{\lambda}({\mathbb C}\Gamma)_{\langle v\rangle}\,\overset{s_{v,\sigma,\sigma'}}{\longrightarrow}\,C_*(\Gamma,{\mathbb C}v)
\,\overset{I\circ p_v}{\longrightarrow}\,\widetilde{C}_*^{\lambda}({\mathbb C}\Gamma)_{\langle v\rangle}\,\longrightarrow\,C_*^{\lambda}({\mathbb C}\Gamma)_{\langle v\rangle}
\eqno(2.73)
$$
is chain homotopic to the identity, where the map on the left hand side is the canonical projection and the map on the right hand side is a chain homotopy inverse of it (see (2.13)).
The contractibility of $C_*(\Gamma,{\mathbb C}v)$ in positive degrees implies the existence of a contracting chain homotopy $\eta_*$ of $C_*^{\lambda}({\mathbb C}\Gamma)_{\langle v\rangle}$ in strictly positive degrees. For an algebra $A$ and $*\in{\mathbb N}$ denote by 
$$
\alpha_*: C_*^{\lambda}(A)
=\underset{k\geq 0}{\bigoplus}\,\Omega^{*-2k}A
\,\longrightarrow\,\underset{l\in{\mathbb Z}}{\prod}\,\Omega^{*+2l}A
=\widehat{CC}_*(A)
\eqno(2.74)
$$
the canonical inclusion and by 
$$
\beta_*:\widehat{CC}_*(A)=
\underset{l\in{\mathbb Z}}{\prod}\,\Omega^{*+2l}A
\,\longrightarrow\,\underset{k\geq 0}{\bigoplus}\,\Omega^{*-2k}A
= C_*^{\lambda}(A)
\eqno(2.75)
$$
the canonical projection. Note that $\beta_*$ is a chain map while $\alpha_*$ is not.  Let
$$
\theta^{(N)}_*\,=\,
\begin{cases}
\alpha_{N+1}\circ\eta_N\circ \beta_N & *\equiv\, N\,(mod\,2), \\
\alpha_{N}\circ\eta_{N-1}\circ \beta_{N-1} & *\equiv\, N-1\,(mod\,2)
\end{cases}
:\widehat{CC}_*({\mathbb C}\Gamma)_{\langle v\rangle}\to \widehat{CC}_{*+1}({\mathbb C}\Gamma)_{\langle v\rangle},
\eqno(2.76)
$$
and put 
$$
\psi^{(N)}_*\,=\,id\,-\,(\theta^{(N)}_*\circ\partial\,+\,\partial\circ\theta^{(N)}_*).
\eqno(2.77)
$$ 
The chain map $\psi^{(N)}_*$ is canonically chain homotopic to the identity 
via a chain homotopy preserving Hodge filtrations and vanishing on $Fil^{N+2}_{Hodge}\widehat{CC}_*({\mathbb C}\Gamma)_{\langle v\rangle}$. Moreover
$$
\psi^{(N)}_*\,=\,id-(\alpha_N\circ\eta_{N-1}\circ \beta_{N-1}\circ\partial
\,+\,\partial\circ \alpha_{N+1}\circ\eta_{N}\circ \beta_{N})
$$
$$
=\,id-(\partial\circ \alpha_{N+1}-\alpha_N\circ\partial)\circ\eta_N\circ \beta_N
-\alpha_N\circ(\eta_{N-1}\circ\partial+\partial\circ\eta_N)\circ \beta_N
$$
on $\widehat{CC}_N({\mathbb C}\Gamma)_{\langle v\rangle}$ so that 
$$
\psi^{(N)}_*(\widehat{CC}_N({\mathbb C}\Gamma)_{\langle v\rangle})
\subset Fil^{N+2}_{Hodge}\widehat{CC}_N({\mathbb C}\Gamma)_{\langle v\rangle}.
$$ 
Thus the chain map 
$$
\zeta^{(N)}_*=\psi^{(N-1)}_*\circ\psi^{(N)}_*:\,\widehat{CC}_*({\mathbb C}\Gamma)_{\langle v\rangle}
\to\,\widehat{CC}_*({\mathbb C}\Gamma)_{\langle v\rangle}
\eqno(2.78)
$$
maps finally $\widehat{CC}_*({\mathbb C}\Gamma)_{\langle v\rangle}$ into $Fil^{N+1}_{Hodge}\widehat{CC}_*({\mathbb C}\Gamma)_{\langle v\rangle}$ and is canonically chain homotopic to the identity via a chain homotopy preserving the Hodge filtration. This proves the assertion of the corollary.
\end{proof}

\section{Hyperbolic spaces and hyperbolic groups}

\subsection{Hyperbolic spaces}

Recall that a metric space $(X,d)$ is {\bf proper} if its closed balls of finite radius are compact and that $(X,d)$ is {\bf geodesic} if any pair of points $x,y\in X$ can be joined by a {\bf geodesic segment}, i.e. if there exists an isometric map
$\gamma:I\to X$ from a bounded closed interval to $X$ such that $\gamma(\partial I)=\{x,\,y\}$.
A geodesic triangle with vertices $x,y,z\in X$ is given by 
three geodesic segments $[x,y],\,[y,z],\,[z,x]$, joining the denoted endpoints. 

\begin{definition} (Gromov) \cite{BH},(III,H)\\
A geodesic metric space $(X,d)$ is {\bf $\delta$-hyperbolic} (for some 
$\delta\geq 0$) if each edge in a geodesic triangle is contained in the tubular $\delta$-neighbourhood of the union of the two other edges.
\end{definition}

A key issue of hyperbolicity is its invariance under quasiisometries. Recall that a map $f:I\to X$ from an interval $I\subset{\mathbb R}$ to a metric space $(X,d)$ is called a
{\bf $(\lambda,C)$-quasigeodesic} for given $\lambda\geq 1$ and $C\geq 0$ if
$$
\lambda^{-1}d_{\mathbb R}(s,t)-C\,\leq\,d_X(f(s),f(t))\,\leq\,
\lambda\,d_{\mathbb R}(s,t)+C,\,\,\forall s,t\in I.
\eqno(3.1)
$$
It is called an {\bf $L$-local $(\lambda,C)$-quasigeodesic} for $L>0$ if its restriction to any subinterval of length at most $L$ satisfies (3.1). Recall that the (possibly infinite) Hausdorff distance between two subsets $Y,Z$ of a metric space $(X,d)$ 
is defined as 
$$
d_{{\mathcal H}}(Y,Z)\,=\,
\inf\,\{
r\in{\overline{\mathbb R}}_+,\,Y\subset B(Z,r),\,Z\subset B(Y,r)\},
\eqno(3.2)
$$
where $B(A,r)\,=\,\{x\in X,\,d(x,A)< r\}$ denotes the open tubular $r$-neighbourhood of $A\subset X$.

\begin{theorem}
\cite{GH},(Section 5, Theorem 25)\\
Let $\delta>0,\lambda\geq 1$ and $C\geq 0$. Then there exist constants 
$L(\delta,\lambda,C)>0$ and $R(\delta,\lambda,C)>0$ such that the following holds. If $(X,d)$ is a proper $\delta$-hyperbolic geodesic metric space, and if 
$f:I\to X$ is an $L(\delta,\lambda,C)$-local, $(\lambda,C)$-quasigeodesic in $X$, then there exists a geodesic
$f':I'\to X$ such that $\widetilde{f}(0)=f(0)$ and 
$$
d_{{\mathcal H}}(f(I),f'(I'))\,\leq\,R(\delta,\lambda,C).
\eqno(3.3)
$$
\end{theorem}

\begin{cor}
Let $(X,d)$ be a proper geodesic $\delta$-hyperbolic metric space and let 
$g:X\to X$ be a hyperbolic isometry, i.e. an isometry such that
$$
\rho_g\,=\,\underset{x\in X}{\inf}\,d(x,gx)\,>0
\eqno(3.4)
$$
is attained and strictly positive. Put
$$
Min(g)\,=\,\{y\in X,\,d(y,gy)=\rho_g\}
\eqno(3.5)
$$
the set of points of minimal displacement under $g$. 
For $x\in Min(g)$ let $f_x:{\mathbb R}\to X$ be a map 
satisfying
$$
\begin{array}{ccc}
i)\,\,\,f_x(0)=x, & ii)\,\,\,f(t+\rho_g)=g\cdot f(t),\,\forall t\in{\mathbb R}, & iii)\,\,\,
f_{x\vert[0,\rho_g]}\,\,\,\text{is a geodesic segment.} \\
\end{array}
$$
Then there exists a constant $C_0(\delta)$ such that
$$
Min(g)\,\subset\,B(f_x({\mathbb R}),\,C_0(\delta)).
\eqno(3.6)
$$
provided that $\rho_g\geq L(\delta,1,0)$. 
\end{cor}

This is a "discrete analog" of the well known fact that 
the set of minimal displacement of a hyperbolic isometry of a 
complete, simply connected Riemannian manifold of negative curvature is an infinite geodesic.

\begin{proof}

Any map $f_x$ as in the assertion is a $\rho_g$-local geodesic segment as the following well known argument shows. Let $s,t\in{\mathbb R},\,0\leq t-s\leq \rho_g$. We may assume without loss of generality that $0< s < \rho_g <t$. Note that 
the restrictions of $f_x$ to $[0,\rho_g]$ and $[\rho_g
,2\rho_g]$ are geodesic segments.
Consequently $d(f_x(s),f_x(t))\leq d(f_x(s),f_x(\rho_g))+d(f_x(\rho_g),f_x(t))\,=\,(\rho_g-s)+(t-¸\rho_g)=t-s$
by the triangle inequality, and
$$
d(f_x(s),f_x(t))\geq d(f_x(s),f_x(s+\rho_g)) - d(f_x(t),f_x(s+\rho_g))
$$
$$
=\,
d(f_x(s),g\cdot f_x(s)) - d(f_x(t),f_x(s+\rho_g))\,\geq\, \rho_g-((s+\rho_g)-t)\,=\,t-s
$$
because $g$ displaces points by at least $\rho_g$. Altogether $d(f_x(s),f_x(t))\,=\,d_{\mathbb R}(s,t)$ as claimed. \\
Let now $y\in Min(g)$ and let $f_y:{\mathbb R}\to X$ be a map satisfying the assumptions of the proposition. According to theorem 3.2 there exist geodesics
$f'_x,f'_y:{\mathbb R}\to X$ such that 
$f'_x(0)=x,\,f'_y(0)=y$ and 
$$
d_{{\mathcal H}}(f_x({\mathbb R}),f'_x({\mathbb R}))
<R(\delta,1,0),\,
d_{{\mathcal H}}(f_y({\mathbb R}),f'_y({\mathbb R}))
<R(\delta,1,0),\,
$$
As 
$$
d(f_x(k\rho_g),f_y(k\rho_g))\,=\,
d(g^kx,g^ky)\,=\,d(x,y),\,\,\,\forall k\in{\mathbb Z},
$$ 
the Hausdorff distance between $f_x({\mathbb R})$ and $f_y({\mathbb R})$ and thus also between
$f'_x({\mathbb R})$ and $f'_y({\mathbb R})$ is finite. But if the Hausdorff distance between two infinite geodesics in a $\delta$-hyperbolic space is finite, then it is actually bounded by $2\delta$. Thus
$$
d(y,f_x({\mathbb R}))\,\leq\,d_{{\mathcal H}}(f_y({\mathbb R}),\,f_x({\mathbb R}))\,\leq\,2R(\delta,1,0)+2\delta=C_0(\delta).
$$
\end{proof}

In the sequel we note for a subset $Y\subset X$ of a metric space $(X,d)$ and for $\lambda\geq 0$ 
$$
geod_\lambda(Y)=\{x\in X,\,\exists y,y'\in Y:\,d(y,x)+d(x,y')\leq d(y,y')+\lambda\},
\eqno(3.7)
$$ 
and
$$
geod(Y)=geod_0(Y)=\{x\in X,\,\exists y,y'\in Y:\,d(y,x)+d(x,y')=d(y,y')\}.
\eqno(3.8)
$$ 

Recall that a (finite) metric tree is a (finite) one-dimensional, simply connected simplicial complex equipped with a length metric such that every closed one-simplex is isometric to a bounded closed interval.

\begin{prop} (Gromov)
Let $F=\{x_0,\ldots,x_n\}$ be a finite metric space with base point $x_0$.  Then there exists a finite metric tree $T_{(F,x_0)}$ and a map $\Phi:X\to T_{(F,x_0)}$ such that the following holds:
\begin{itemize}
\item $\Phi$ is contractive: $d(\Phi(x),\Phi(y))\leq d(x,y),\,\forall x,y\in F$.
\item $\Phi$ preserves the distance from the basepoint: $d(\Phi(x),\Phi(x_0))=d(x,x_0),\, x\in F$.
\item The image of $\Phi$ spans $T_{(F,x_0)}$: $T_{(F,x_0)}\,=\,geod(\Phi(F))$.
\item If $\Phi':F\to T'$ is a map to a metric tree satisfying the previous conditions,  
then there is a (unique) contractive map $f:T_{(F,x_0)}\to  T'$ such that $\Phi'\,=\,f\circ\Phi$.
\end{itemize}
$T_{(F,x_0)}$ is called the {\bf approximating tree} of $(F,x_0)$.
\end{prop} 

\begin{theorem}(Gromov)  \cite{Gr},\cite{GH}
Let $(F,x_0)$ be a finite metric space with base point $x_0$, and let 
$\Phi:F\to T_{(F,x_0)}$ be the canonical map to its approximating tree.
If $F$ is $\delta$-hyperbolic of cardinality $\vert F\vert\leq 2^k+2$, then 
$$
d(x,x')-2k\delta\,\leq\,d(\Phi(x),\Phi(x')),\,\,\,\forall x,x'\in F.\eqno(3.9)
$$
\end{theorem}

\begin{cor} {\bf (Approximation by trees)}\cite{GH},\cite{BFS}
Let $(X,d)$ be a $\delta$-hyperbolic geodesic metric space. Let $F\subset X$ be a finite subset with base point $x_0$, let 
$$
\Phi:F\to T_{(F,x_0)}
$$ 
be the canonical map to its approximating tree and let 
$$
T'\,=\,\{x\in T_{(F,x_0)},\,d(x,\Phi(x_0))\in{\mathbb N}\}.
$$ 
For every $\lambda\geq 0$ there exists a constant $C_5(\delta,\vert F\vert,\lambda)\geq 0$ and maps 
$$
\varphi_\lambda:geod_\lambda(F)\,\longrightarrow\,T',\,\psi:
T'\to geod(F)
\eqno(3.10)
$$
such that
$$
d(\varphi_\lambda(x),\varphi_\lambda(x'))\,\leq\,d(x,x')\,+\, C_5(\delta,\vert F\vert,\lambda),\,\,\,\forall x,x'\in geod_\lambda(F),\eqno(3.11)
$$

$$
d(\psi(y),\psi(y'))\,\leq\,d(y,y')\,+\, C_5(\delta,\vert F\vert,\lambda),\,\,\,\forall y,y'\in T',\eqno(3.12)
$$

$$
d(\psi\circ\varphi_\lambda(x),x)\,\leq\, C_5(\delta,\vert F\vert,\lambda),\,\,\,\forall x\in geod_\lambda(F).\eqno(3.13)
$$
\end{cor}

\begin{proof} We proceed in several steps.\\
{\bf Step 1:} \\
Let $F=\{x_0,\ldots,x_n\}.$ By hyperbolicity any point $x\in geod_\lambda(F)$ lies at distance at most $\frac12\lambda+2\delta=C_1(\delta,\lambda)$ from $geod(F)$ and thus at distance at most $\frac12\lambda+3\delta+1=C_2(\delta,\lambda)$ from a point $x'\in geod(\{x_0,x_j\})\cap\{z\in X,\,d(x_0,z)\in{\mathbb N}\}$ for some $j\in\{0,\ldots,n\}.$ Let $\Phi:F\to T_{(F,x_0)}$ be the   map from $F$ to its approximating tree and define $\varphi(x)\in T'$ as the point at distance $d(x_0,x')$ from the basepoint on the unique geodesic segment 
joining $y_0=\Phi_{(F,x_0)}(x_0)$ and $\Phi_{(F,x_0)}(x_j)$ in $T_{(F,x_0)}$.
\\
\\
{\bf Step 2:}\\
Let $y\in T'$ and choose an index $k\in\{0,\ldots\,n\}$ such that $y$ lies on the unique geodesic segment in $T$ joining $y_0$ and $\Phi_{(F,x_0)}(x_k)$. Choose as 
$\psi(y)\in geod(F)$ a point in $geod(\{x_0,x_k\})$ situated at distance $d(y_0,y)$ from $x_0$.\\
\\
{\bf Step 3:}\\
Let $x,y\in geod_\lambda(F)$ and let $x'\in geod(\{x_0,x_i\}),\,y'\in geod(\{x_0,x_j\})$ be the auxiliary points used to define $\varphi(x)$ and $\varphi(y)$.
Put $F'=F\cup\{x',y'\}$. By the universal property of the approximating tree there is a commutative diagram
$$
\begin{array}{ccccc}
 & F & \hookrightarrow & F' & \\
& & & & \\
\Phi_{(F,x_0)} & \downarrow & & \downarrow & \Phi_{(F',x_0)} \\
& & & & \\
& T_{(F,x_0)} & \to & T_{(F',x_0)} & \\
\end{array}
$$
whose lower horizontal arrow is a bijective isometry. It follows that 
$$
d(\varphi(x),\varphi(y))= d(\Phi_{(F',x_0)}(x'),\Phi_{(F',x_0)}(y'))\leq
 d(x',y')\leq d(x,y)+2C_2(\delta,\lambda).
$$
\\
{\bf Step 4:}\\
Let $y,y'\in T'$ and put $F''=F\cup\{\psi(y),\psi(y')\}$.
By construction
$y\,=\,\Phi_{(F'',x_0)}(\psi(y))$ and $y'\,=\,\Phi_{(F'',x_0)}(\psi(y')),$ so that 
$d(\psi(y),\psi(y'))\leq d(y,y')+2\log_2(\vert F\vert)\delta$ by 3.5.

\newpage

\begin{figure}[htpb]
\begin{center}
\includegraphics[scale=0.15]{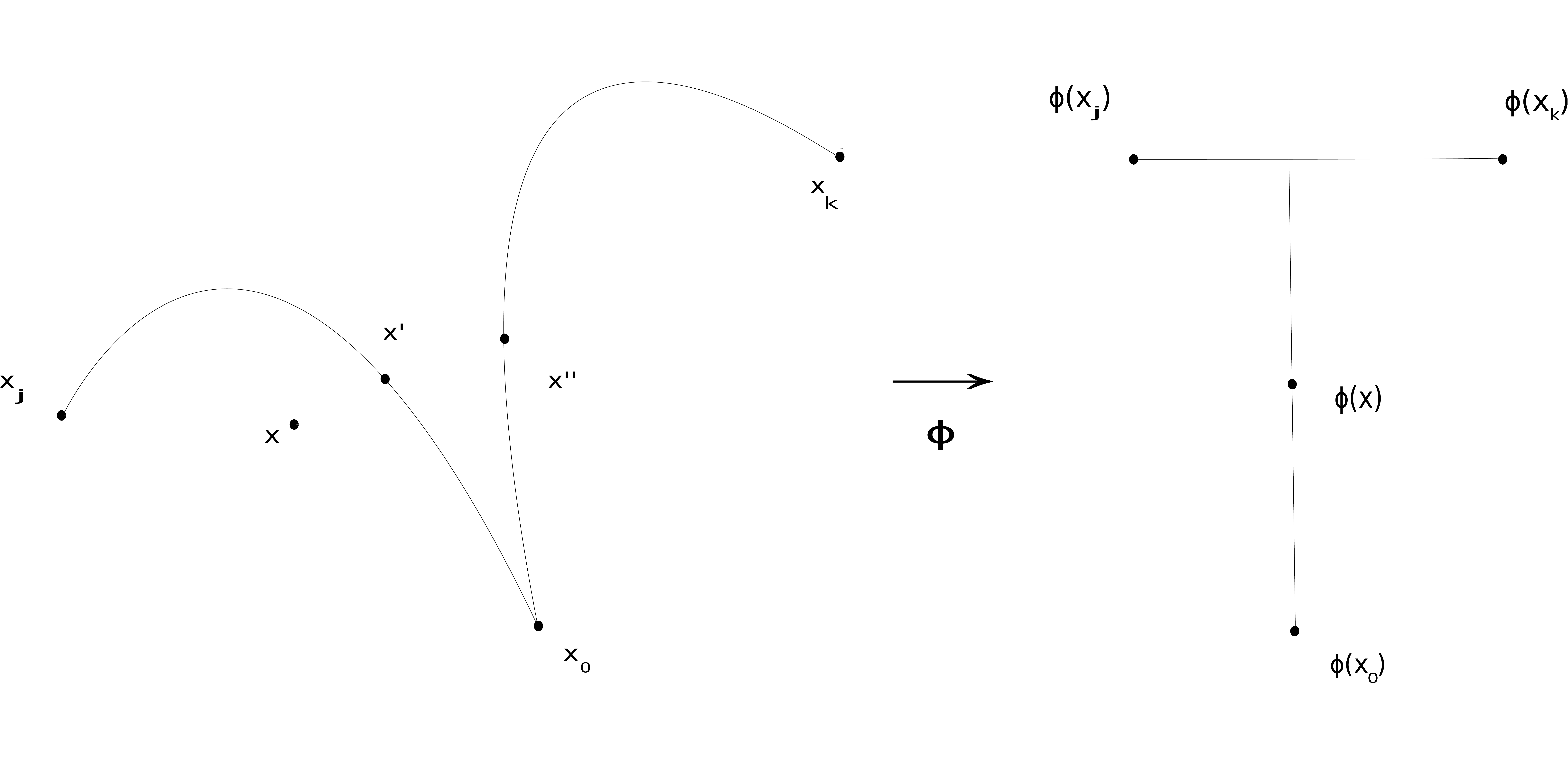}
\caption{Mapping $geod_\lambda(F)$ to the approximating tree of $F$}
\label{default}
\end{center}
\end{figure}

.\\
{\bf Step 5:}
\\
Let $x\in geod_\lambda(F)$ and let $x'\in geod(\{x_0,x_j\})\cap\{z\in X,\,d(x_0,z)\in{\mathbb N}\}$ be the auxiliary point satisfying $\varphi(x)=\Phi_{(F\cup\{x'\},x_0)}(x')$. 
The point $x''=\psi(\varphi(x))\in geod(\{x_0,x_k\})$ satisfies then $y''=\Phi_{(F\cup\{x',x''\},x_0)}(x'')=\Phi_{(F\cup\{x',x''\},x_0)}(x')$. The point $y''$ lies therefore on the geodesic segments $[y_0,\Phi_{(F,x_0)}(x_j)]$ and $[y_0,\Phi_{(F,x_0)}(x_k)]$ of the tree $T_{(F\cup\{x',x''\},x_0)}=T_{(F,x_0)}$. It follows that
$$
d(y_0,y'')\,\leq\,\frac{1}{2}\left(d(\Phi(x_0),\Phi(x_j))+d(\Phi(x_0),\Phi(x_k))-d(\Phi(x_j),\Phi(x_k))\right)
$$
and thus 
$$
d(x_0,x')=d(x_0,x'')\leq\frac12\left(d(x_0,x_j)+d(x_0,x_k)-d(x_j,x_k)\right)+2\log_2(\vert F\vert)\delta
$$ 
because $\Phi$ is contractive and satisfies (3.9). As $x'\in geod(\{x_0,x_j\})$ and $x''\in geod(\{x_0,x_k\})$ 
the previous estimate implies by hyperbolicity that 
$d(x',x'')\leq C_3(\delta)+4\log_2(\vert F\vert)\delta$ so that finally 
$$
d(x,\psi(\varphi(x)))\leq d(x,x')+d(x',x'')\leq C_2(\delta,\lambda)+C_3(\delta)+4\log_2(\vert F\vert)\delta=C_4(\delta,\vert F\vert, \lambda).
$$ 
Thus the assertions holds for $C_5(\delta,\vert F\vert,\lambda)=Max(2C_2(\delta,\lambda),2\log(\vert F\vert)\delta,C_4(\delta,\vert F\vert,\lambda))$.
\end{proof}

The key point of the previous result is that the {\bf distortion of the tree approximation of $geod_\lambda(F)$  does not depend on the diameter of $F$}.

\newpage

\subsection{Hyperbolic groups}

Let $(\Gamma,S)$ be a finitely generated group with associated word length function $\ell_S$ and word metric $d_S$. The corresponding Cayley graph ${\mathcal G}(\Gamma,S)$ is a proper geodesic metric space on which $\Gamma$ acts properly, isometrically and cocompactly by left translation. The group $\Gamma$ is called {\bf hyperbolic}
if its Cayley graph with respect to some (and thus to every) finite, symmetric set of generators is hyperbolic in the sense of 3.1. (The constant $\delta$ depends of course on the choice of $S$.) By abuse of language we call the pair $(\Gamma,S)$ a {\bf $\delta$-hyperbolic group} if ${\mathcal G}(\Gamma,S)$ is a $\delta$-hyperbolic space.
\\
We recall a few facts about hyperbolic groups. These are all due to 
M.~Gromov.

\begin{theorem} {\bf (Centralizers)} (\cite{BH}, III,\,3.7,\,3.9,\,3.10) \\
Let $\Gamma$ be a hyperbolic group. For $g\in\Gamma$ let $Z(g)\,=\,\{h\in\Gamma,\,gh=hg\}$ be its centralizer.\\
a) If $g$ is of infinite order, then the centralizer 
$Z(g)$ is virtually cyclic, i.e. the infinite cyclic central subgroup generated by $g$ is of finite index in $Z(g)$.\\
b) If $g$ is of finite order, then $Z(g)$ is itself a 
hyperbolic group and the restriction of the word metric of $\Gamma$ to $Z(g)$ is quasi-isometric to any word metric on $Z(g)$.
\end{theorem}

\begin{theorem} {\bf (Stable length)} (\cite{BH}, III,\,3.10,\,3.15)\\
Let $(\Gamma,S)$ be a hyperbolic group. Define the stable length of $g\in\Gamma$ by
$$
\ell_S^{stable}(g)\,=\,\underset{n\to\infty}{\lim}\,
\frac{\ell_S(g^n)}{n}.
\eqno(3.14)
$$
There exists $\epsilon_0>0$ such that $\ell_S^{stable}(g)\geq\epsilon_0$ for all $g\in\Gamma$ of infinite order.
\end{theorem}

\begin{theorem} {\bf (Finite subgroups)} (\cite{BH},\,III,3.2) \\
A hyperbolic group contains only finitely many conjugacy classes of finite subgroups.
\end{theorem}

\begin{theorem} {\bf (Conjugacy problem)}\\
Let $(\Gamma,S)$ be a $\delta$-hyperbolic group and let $v\in\Gamma$ be an element of minimal word length in its conjugacy class $\langle v\rangle $. 

\begin{itemize}
\item[a)] There exists a section $\sigma_v:\langle v\rangle\to\Gamma$
of the $\Gamma$-equivariant map $\Gamma\to\langle v\rangle,$\\
$g\mapsto gvg^{-1}$ satisfying
$$
\ell_S(\sigma_v(u))\,\leq\,\frac12\ell_S(u)\,+\, C_{11}(\Gamma,S,\delta)
\eqno(3.15)
$$
for all $u\in\langle v\rangle$ and a suitable constant $C_{11}$ depending only on $\Gamma$, $S$ and $\delta$.
\item[b)] There exist constants $C_{12},\,C_{13},$ depending only on $\Gamma,\,S,\,\delta$, such that any section as in a) satisfies in addition
$$
d_S(\sigma_v(gug^{-1}),\,g\cdot\sigma_v(u))\,\leq\,C_{12}(\Gamma,S,\delta)\cdot\ell_S(g)+\ell_S(\langle v\rangle)+C_{13}(\Gamma,S,\delta)
\eqno(3.16)
$$
for all $g\in\Gamma$ and all $u\in\langle v\rangle.$
\end{itemize}
\end{theorem}

\begin{proof}

As we have not found a precise reference for this result, we outline its proof. For a conjugacy class $\langle v\rangle$ denote by $\ell_S(\langle v\rangle)=\ell_S(v)$ the minimal word length of its elements. 
The notations of 2.2 are understood. We distinguish two cases. \\
\\
{\bf Case 1:} $\ell_S(\langle v\rangle)\geq Max(L(\delta,1,10\delta), L(\delta,1,12\delta))$ and $\vert v\vert=+\infty$.\\
\\
For given $u\in\langle v\rangle$ let $h(u)\in\Gamma$ be an element of minimal word length such that $u'=h(u)^{-1}uh(u)$ is represented by a cyclic permutation of a minimal word representing $v$. 
Pick then $a,b\in\Gamma$ such that $ab=v,\,ba=u',\,\ell_S(a)+\ell_S(b)=\ell_S(\langle v\rangle).$ 
If $\ell_S(b)\leq\frac12\ell_S(\langle v\rangle)$ put $\sigma_v(u)=h(u)b$, otherwise put $\sigma_v(u)=h(u)a^{-1}$.
By construction $\sigma_v(u)v\sigma_v(u)^{-1}=u$. Any concatenation 
$$
\gamma=[e,h(u)]\vee [h(u),uh(u)]\vee [uh(u),u]
$$ 
of geodesic segments joining the denoted endpoints is then an $\ell_S(\langle v\rangle)$-local $(\delta,1,10\delta)$-quasigeodesic segment in ${\mathcal G}(\Gamma,S)$ (see \cite{Pu5},\,4.2). By our assumption $\ell_S(\langle v\rangle)\geq L(\delta,1,10\delta)$, so that $\gamma$ is at uniformly bounded Hausdorff distance $R(\delta,1,10\delta)$ from any true geodesic segment $[e,u]$. Therefore  
$$
\vert 2\ell_S(h(u))\,+\,\ell_S(\langle v\rangle)\,-\,\ell_S(u)\vert\,\leq\,C_6(\delta)
$$
for a universal constant $C_6(\delta)$. It follows that 
$$
\ell_S(\sigma_v(u))\,\leq\,\ell_S(h(u))\,+\,\frac12\ell_S(\langle v\rangle)\,\leq\,\frac12\ell_S(u)\,+\,\frac12C_6(\delta).
$$
Concerning assertion b), we distinguish two subcases. If $$\ell_S(g)< \frac{\ell_S(u)-\ell_S(\langle v\rangle)}{2}-L(\delta,1,12\delta)-\delta,$$ 
then the concatenation 
$$
\gamma_0=[e,gh(u)]\vee [gh(u),guh(u)]\vee [guh(u),gug^{-1}]
$$ 
of any geodesic segments joining the denoted endpoints will form a $L(\delta,1,12\delta)$-local $(\delta,1,12\delta)$-quasigeodesic segment. By 3.2 
it will be at uniformly bounded Hausdorff distance $R(\delta,1,12\delta)$ from any true geodesic segment $[e,gug^{-1}]$ and thus at Hausdorff distance at most $R(\delta,1,12\delta)+R(\delta,1,10\delta)$ from the $L(\delta,1,10\delta)$-local $(\delta,1,10\delta)$-quasigeodesic segment 
$$
\gamma_1=[e,h(gug^{-1})]\vee [h(gug^{-1}),gug^{-1}h(gug^{-1})]\vee [gug^{-1}h(gug^{-1}),gug^{-1}].
$$
It follows that the distance between the midpoints of these quasigeodesic segments is bounded by a constant $C_7(\delta)$ depending only on $\delta$. As the midpoint of 
$\gamma_0$ is at distance at most $\frac12\ell_S(\langle v\rangle)$ from $g\sigma_v(u)$ and that of $\gamma_1$ is at distance at most $\frac12\ell_S(\langle v\rangle)$ from $\sigma_v(gug^{-1})$ the assertion follows. If however $\ell_S(g)\geq \frac{\ell_S(u)-\ell_S(\langle v\rangle)}{2}-L(\delta,1,12\delta)-\delta$, then
$$
d_S(\sigma_v(gug^{-1}),\,g\cdot\sigma_v(u))\leq\ell_S(\sigma_v(gug^{-1}))+\ell_S(g\cdot\sigma_v(u))
$$ 
$$
\leq 2\ell_S(g)+\ell_S(u)+2C_6(\delta)
$$
$$
\leq\,4\ell_S(g)+\ell_S(\langle v\rangle)+2(L(\delta,1,12\delta)+\delta+C_6(\delta))
$$ 
and assertion b) follows as well.

\begin{figure}[htpb]
\begin{center}
\includegraphics[scale=0.17]{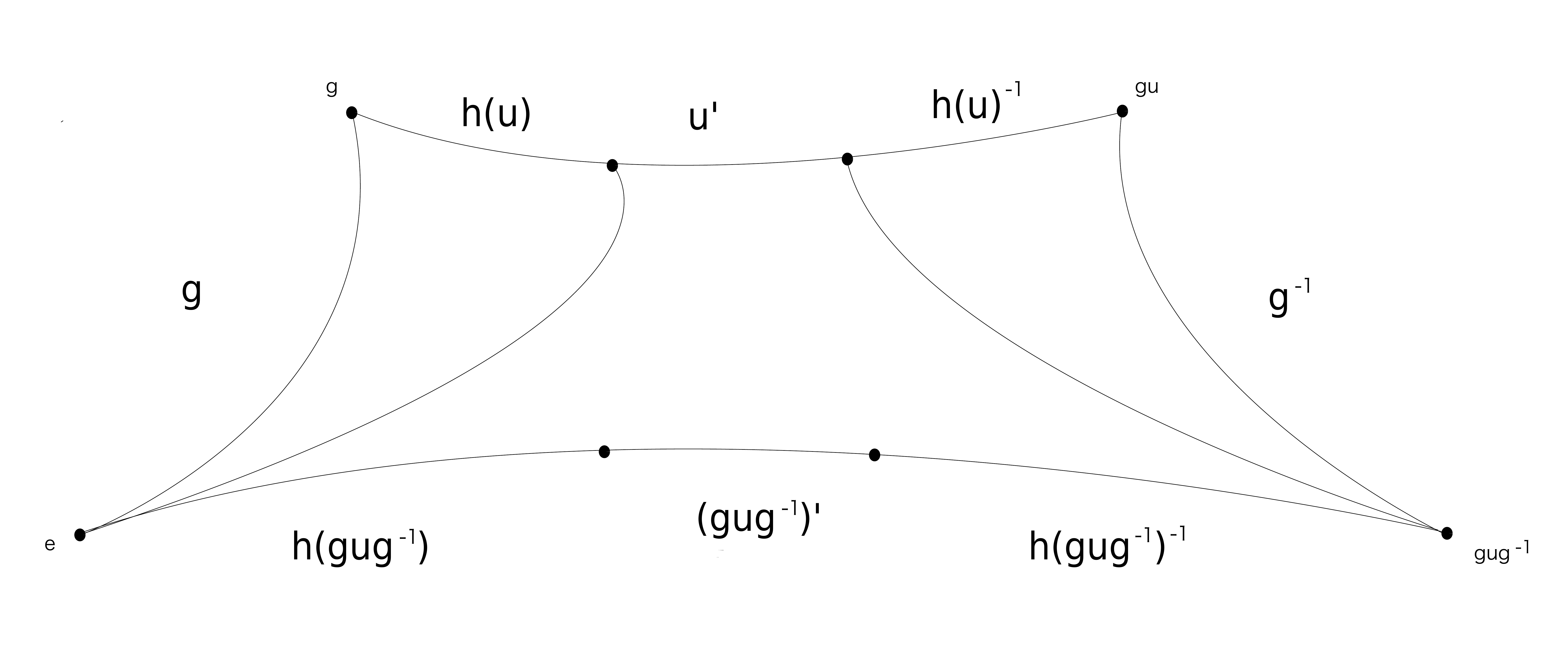}
\caption{Case 1 of 3.10}
\label{default}
\end{center}
\end{figure}

{\bf Case 2:} $\ell_S(\langle v\rangle)< Max(L(\delta,1,10\delta), L(\delta,1,12\delta))$ or $\vert v\vert<\infty$.\\
\\
This case concerns only a finite number of conjugacy classes. 
According to \cite{Pu5},\,4.1., for example, there exists a constant 
$C_{8}(\Gamma,S,\delta)$ such that some cyclic permutation of a minimal word representing 
an element $u$ in any of the finitely many conjugacy classes $\langle v\rangle$ under consideration will be of length less than $C_8(\Gamma,S,\delta)$. It follows that the midpoint of any geodesic segment $[e,u]$ will be situated at distance at most $\frac12C_8$ from an element $a\in\Gamma$ such that $u=ab,\,\ell_S(u)=\ell_S(a)+\ell_S(b)$ and $u'=ba$ satisfies $\ell_S(ba)\leq C_8$. Pick for any such $u'$ some $h(u')\in\Gamma$ of minimal word length satisfying 
$u'=h(u')vh(u')^{-1}$. The element  $\sigma_v(u)=ah(u')$ satisfies $\sigma_v(u)v\sigma_v(u)^{-1}=u$. Moreover it is at distance at most $C_{9}(\Gamma,S,\delta)$ from the midpoint of 
any geodesic segment $[e,u]$ because there are only finitely many choices for $u'$. It follows that 
$\ell_S(\sigma_v(u))\,\leq\,\frac12\ell_S(u)\,+\, C_{9}$ and $d_S(\sigma_v(gug^{-1}),\,g\cdot\sigma_v(u))\,\leq\,C_9\cdot\ell_S(g)+C_{10}$ for a suitable choice of $C_{10}(\Gamma,S,\delta)$. 
The assertion of the theorem holds then for $C_{11}=Max(\frac12 C_6(\delta),C_9(\Gamma,S,\delta)),C_{12}= Max(4,C_9(\Gamma,S,\delta))$ and $C_{13}=Max(C_7(\delta),2(L(\delta,1,12\delta)+\delta+C_6(\delta)),C_{10}(\Gamma,S,\delta))$.
\end{proof}

\begin{lemma}
Let $(\Gamma,S)$ be a $\delta$-hyperbolic group. Then there  exists a constant $C_{14}$ such that the following holds: if $g\in\Gamma$ is an element of infinite order and of minimal length in its conjugacy class, then there exists a set-theoretic section $\sigma':N(g)\to Z(g)$ of the extension
$$
0\,\to\,g^{\mathbb Z}\,\to\,Z(g)\,\to\,N(g)\,\to\,0
$$ 
which satisfies 
$$
\ell_S(\sigma'(h))\,\leq\,\ell_S(\langle g\rangle)\,+\,C_{14}(\Gamma,S,\delta),\,\forall h\in N(g).
\eqno(3.17)
$$
\end{lemma}
\begin{proof}
Let $g\in\Gamma$ be an element of infinite order which is of minimal word length $\ell_S(\langle g\rangle)$ in its conjugacy class $\langle g\rangle$. 
Consider the action of $\Gamma$ on its Cayley graph ${\mathcal G}(\Gamma,S)$. Because $g\in\Gamma$ is of infinite order (it suffices actually to demand that its order is different from two), it does not fix any edge of ${\mathcal G}(\Gamma,S)$ and its set of minimal displacement contains necessarily vertices: $Min(g)\cap {\mathcal G}_0(\Gamma,S)\neq\emptyset$. As $d_S(u,gu)\,=\,\ell_S(u^{-1}gu),\,g,u\in\Gamma,$ one finds
$$
\{u\in\Gamma,\,\ell_S(u^{-1}gu)=\ell_S(g)\}\,\subset\,Min_{{\mathcal G}(\Gamma,S)}(g).
\eqno(3.18)
$$
In particular 
$$
Z(g)\subset Min_{{\mathcal G}(\Gamma,S)}(g).
$$
Suppose that $\ell_S(g)\geq L(\delta,1,0)$ and consider a quasigeodesic $f_e:{\mathbb R}\to{\mathcal G}(\Gamma,S),$ which intertwines the translation by $\ell_S(g)$ on $\mathbb R$ and the 
translation by $g$ on ${\mathcal G}(\Gamma,S),$ and whose restriction to $[0,\ell_S(g)]$ is a geodesic segment joining $e$ and $g$.
Then  
$$
Min_{{\mathcal G}(\Gamma,S)}(g)\,\subset\,B(f_e({\mathbb R}),C_0(\delta)).
\eqno(3.19)
$$
according to 3.3. One finds therefore for every $h\in Z(g)$ an integer $k$ such that $d(h,g^k\cdot f_e([0,\ell_S(g)]))<C_0(\delta)$ or
$\ell_S(g^{-k}h)<\ell_S(g)+C_0(\delta).$
This shows that every $g^{\mathbb Z}$-orbit in $Z(g)$ contains an element of length at most $\ell_S(g)+C_0(\delta)$, which is equivalent to the assertion. It remains to deal with conjugacy classes of length less or equal to $L(\delta,1,0)$. There is only a finite number of them so that 
the claim still holds if the constant $C_{14}$ is chosen large enough.
\end{proof}

\section{Controlled resolutions of hyperbolic groups}

\subsection{Rips complexes}

\begin{definition}
Let $(X,d)$ be a metric space and let $R\geq 0$. The {\bf Rips-complex} $\Delta_\bullet^R(X)$ of $(X,d)$ is the simplicial subset of the Bar-complex $\Delta_\bullet(X)$ given by the Bar-simplices of diameter at most $R:$
$$
\Delta_n^R(X)\,=\,\{(x_0,\ldots,x_n)\in X^{n+1},\,d(x_i,x_j)\leq R,\,0\leq i,j\leq n\}.
\eqno(4.1)
$$
\end{definition}
If $f:X\to Y$ is a contractive map of metric spaces, i.e. 
$d(f(x),f(x'))\leq d(x,x'),$ $\forall x,x'\in X,$ then $f_\bullet$ preserves the Rips-complexes, i.e. $f_\bullet:\Delta_\bullet^R(X)\to \Delta_\bullet^R(Y),\,R\geq 0$. 
More generally, if $g:X\to Y$ is a map of metric spaces satisfying 
$$
d_Y(g(x),g(x'))\leq \varphi(d_X(x,x'))
$$ 
for $x,x'\in X$ and some monotone increasing function $\varphi:{\mathbb R}_+\to{\mathbb R}_+$, then 
the simplicial map $f_\bullet$ preserves Rips-complexes in the sense that 
$$
f_\bullet:\,\Delta_\bullet^R(X)\to \Delta_\bullet^{\varphi(R)}(Y).
\eqno(4.2)
$$
In particular, every isometric group action on a metric space $(X,d)$ gives rise to a simplicial action on the Rips-complexes $\Delta_\bullet^R(X)$ for any $R>0$.

\begin{definition}
The {\bf Rips chain complexes} $C_*^R(X,{\mathbb C})$
of a metric space are the chain complexes associated to the simplicial sets $\Delta_\bullet^R(X)$. They are subcomplexes of the Bar chain complex.
\end{definition}

The {\bf $\ell^1$-norm} $\parallel -\parallel_1$ is the largest norm on $C_*(\Gamma,{\mathbb C})$ such that any base vector corresponding to a simplex has norm one. The {\bf support} of a Bar-chain is the union of the support of the simplices occuring in it with nonzero multiplicity.

\subsection{Contracting metric trees}

A {\bf metric tree} $(T,d)$ is a one-dimensional, simply connected simplicial complex $T$, equipped with a length metric $d$ such that every edge is isometric to a closed bounded interval.
Let $y\in T$ be a base point and let $T'\,=\,\{x\in T,\,d(x,y)\in{\mathbb N}\}$.

\begin{lemma}
There exists a contracting chain homotopy
$$
\sigma_*(T,y):\,C_*(T',{\mathbb C})\,\to\,C_{*+1}(T',{\mathbb C})
\eqno(4.3)
$$
of the augmented Bar complex of $T'$ such that the following holds for $n\geq 1$:

\begin{itemize}
\item $Supp(\sigma_n(T,y)([x_0,\ldots,x_n]))\,\subset \{x\in T',\,d(x,geod(\{x_0,\ldots,x_n\}))<1\},$ 
\item $\parallel \sigma_n(T,y)([x_0,\ldots,x_n])\parallel_1\,\leq\,C_{15}(n,diam(\{x_0,\ldots,x_n\})).$
\end{itemize}

Note that the constant $C_{15}(n,diam(\{x_0,\ldots,x_n\}))$ does not depend on the given metric tree.
\end{lemma}

\begin{proof}
Put $\sigma_{-1}(T,y):\,C_{-1}(T',{\mathbb C})={\mathbb C}\,\to\,C_0(T',{\mathbb C}),\,\,\,1\mapsto [y]$
and 
$$
\sigma_0(T,y):\,C_0(T',{\mathbb C})\,\to\,C_1(T',{\mathbb C}),\,\,\,[x]\mapsto\underset{i=0}{\overset{r-1}{\sum}}[x_i,x_{i+1}],\eqno(4.4)
$$ 
where $r=d(x,y),\,x_0=y, x_r=x$ and $d(x_i,x_{i+1})=1$, i.e. $x_0=y,\ldots,x_r=x$ are the points of $T'$ lying on the unique geodesic segment joining the base point $y$ and the point $x$ in $T$. For a $k$-simplex $[x_0,\ldots,x_k]\in\Delta_k(T'),\,k>0,$ define inductively 
$$
\sigma_{k}(T,y)([x_0,\ldots,x_k])\,=\,
s_{x_0}((Id-\sigma_{k-1,y}\circ\partial)([x_0,\ldots,x_k])).
\eqno(4.5)
$$ 
The linear maps $\sigma_*(T,y),*\geq -1,$ define then a contracting homotopy of 
the Bar complex $C_*(T',{\mathbb C}).$ Let $[x_0,x_1]\in C_1(T',{\mathbb C})$ be a one simplex.
The one chain $\sigma_0(T,y)\circ\partial([x_0,x_1])$ is supported on $\{x\in T',\,d(x,geod(\{x_0,x_1\}))<1\}$ and its $\ell^1$-norm is bounded by $d(x_0,x_1)+2.$ It follows that the assertion of the lemma hold for $n=1$ and $C_{15}(1,R)=R+2$. The assertions in higher degree follow by an easy induction argument.
\end{proof}

\begin{cor}
The restriction of $\sigma_*(T,y)$ to the Rips complexes $C^R_*(T',{\mathbb C}), R\geq 0,$ satisfies
$$
\sigma_n(T,y)(C_n^R(T',{\mathbb C}))\subset C_{n+1}^{R+2\delta+2}(T',{\mathbb C}),\,n\geq 0.\eqno(4.6)
$$
\end{cor}
\begin{proof}
By construction $\sigma_0(T,y)(C_0(T',{\mathbb C}))\subset C_{1}^{1}(T',{\mathbb C})$. Let $[x_0,\ldots,x_n]\in\Delta_n(T')$ be a Bar $n$-simplex. According to 4.3 $\sigma_n(T,y)([x_0,\ldots,x_n])$ is supported in $\{x\in T',\,d(x,geod(\{x_0,\ldots,x_n\}))<1\}$. As this set is of diameter at most $diam(geod(\{x_0,\ldots,x_n\}))+2\leq diam(\{x_0,\ldots,x_n\})+2\delta+2$ the assertion follows from (4.5) by induction.
\end{proof}

\subsection{Comparing Bar- and Rips complexes for hyperbolic metric spaces (after Bader-Furman-Sauer)}

\begin{definition}
Let $(\Gamma,S)$ be a word-hyperbolic group and let $d_S$ be the associated word metric on $\Gamma$. An {\bf equivariant $\lambda$-geodesic bicombing of $(\Gamma,S)$ of weight $\mu>0$} is a $\Gamma$-equivariant linear map 
$$
\Theta_1: C_1(\Gamma,{\mathbb C})\to C_1^1(\Gamma,{\mathbb C})
\eqno(4.7)
$$
from the space of Bar-1-chains in $\Gamma$ to the space of Rips 1-1-chains such that
 the following assertions hold for $g_0,g_1\in\Gamma$:
\begin{itemize}
\item $\Theta_1$ is compatible with the simplicial differential, i.e. 
$$
\partial\Theta_1([g_0,g_1])\,=\,\partial[g_0,g_1]\,=\,[g_1]-[g_0].
\eqno(4.8)
$$
\item The support of the image of a simplex is contained in its $\lambda$-geodesic hull, i.e.
$$supp(\Theta_1([g_0,g_1])\subset geod_\lambda(\{g_0,g_1\}),\eqno(4.9)$$ 
\item $\Theta_1$ is of weight less or equal to $\mu$, i.e. 
$$\parallel\Theta_1([g_0,g_1])\parallel_1\,\leq\,
\mu(d_S(g_0,g_1)+1).\eqno(4.10)$$ 
\end{itemize}

\end{definition}

A bicombing is a simplicial analog of the geodesic segment joining two given points in a non-positively curved space.

\begin{theorem}(Bader-Furman-Sauer)\cite{BFS}
Let $(\Gamma,S)$ be a $\delta$-hyperbolic group. Let $\Theta_1$ be an equivariant $\lambda$-geodesic bicombing of $(\Gamma,S)$ of weight $\mu>0$. Then
the bicombing extends to a $\Gamma$-equivariant chain map $\Theta_*:C_*(\Gamma,{\mathbb C})\to C_*(\Gamma,{\mathbb C})$, which equals the identity in degree zero, and such that for suitable constants $R_n,\,\lambda_n,\,\mu_n,\,n\in{\mathbb N},$ the following holds: 

\begin{itemize}
\item $\Theta_n(C_n(\Gamma,{\mathbb C}))\,\subset\,C_n^{R_n}(\Gamma,{\mathbb C})$
\item $Supp(\Theta_n([g_0,\ldots,g_n]))\,\subset\,geod_{\lambda_n}(\{g_0,\ldots,g_n\}),\,\,\,\forall g_0,\ldots,g_n\in\Gamma.$
\item $\parallel \Theta_n([g_0,\ldots,g_n])\parallel_1\,\leq\,\mu_n(diam(\{g_0,\ldots,g_n\})+1),\,\,\,\forall g_0,\ldots,g_n\in\Gamma.$
\end{itemize}
\end{theorem}

\begin{proof}
We recall the argument of Bader, Furman and Sauer based on induction with respect to the degree. As $\Theta_1$ is a $\Gamma$-equivariant and $\lambda$-geodesic bicombing of $(\Gamma,S)$ of weight $\mu$, the assertion holds for $n=1$ with $R_1=1,\,\lambda_1=\lambda,$ and $\mu_1=\mu$. Suppose that a chain map $\Theta_*$ satisfying the assertions of the theorem has been constructed up to degree $*=n-1$. The chain group $C_n(\Gamma,{\mathbb C})$ is a free $\Gamma$-module with basis given by the simplices $[e,g_1,\ldots,g_n],\,g_1,\ldots,g_n\in\Gamma$. We fix such a simplex $\alpha_n\in\Delta[\Gamma]_n$, let $(Supp(\alpha_n),e)$ be its support, equipped with the base point $e$, and consider its approximating tree $T=T_{(Supp(\alpha_n),e}$. Recall the roughly isometric maps
$$
\begin{array}{cc}
\varphi_{\lambda_{n-1}}:\,geod_{\lambda_{n-1}}(Supp(\alpha_n))\,\to\,T', & 
\psi:\,T'\,\to\,geod(Supp(\alpha_n))
\end{array}
$$
and put, following \cite{BFS},
$$
\Theta_n(\alpha_n)\,=\,\left(\psi\circ\sigma_{n-1}(T',e)\circ\varphi_{\lambda_{n-1}}\,+\,h(\psi\circ\varphi_{\lambda_{n-1}},id)\right)\circ\Theta_{n-1}\circ\partial(\alpha_n)
\eqno(4.11)
$$
Here $\sigma_{*}(T',e)$ is the contracting homotopy (4.3) of the discretized tree $T'$ with base point $e$ and $h(\psi\circ\varphi_{\lambda_{n-1}},id)$ demotes the canonical chain homotopy (2.30) 
between $\psi\circ\varphi_{\lambda_{n-1}}$ and the identity.  The map $\Theta_n$ is well defined because 
$$
\Theta_{n-1}(\partial(\alpha_n))\subset geod_{\lambda_{n-1}}(Supp(\alpha_n))
$$ 
by the induction assumption and extends in a unique way to a $\Gamma$-equivariant linear map of chain groups. Moreover $\Theta_{n-1}(\partial(\alpha_n))$ is a cycle in $C_{n-1}(\Gamma,{\mathbb C})$, and the chain $\Theta_n(\alpha_n)$ is a filling of it. Thus $\Theta_*:C_*(\Gamma,{\mathbb C})\to C_*(\Gamma,{\mathbb C})$ is a chain map for $*\geq 0$.\\
\\
Next we estimate the support of $\Theta_n(\alpha_n)$. The induction assumption, 3.6 and 4.3 imply 
$$
Supp(h(\psi\circ\varphi_{\lambda_{n-1}},id)\circ\Theta_{n-1}\circ\partial(\alpha_n))\,\subset\,geod_{2C_5(\delta,n+1,\lambda_{n-1})+\lambda_{n-1}}(Supp(\alpha_n))
\eqno(4.12)
$$
and
$$
Supp(\psi\circ\sigma_{n-1}(T',e)\circ\varphi_{\lambda_{n-1}}\circ\Theta_{n-1}\circ\partial(\alpha_n))\,\subset\,geod(Supp(\alpha_n))
\eqno(4.13)
$$
so that
$$
Supp(\Theta_n(\alpha_n))\,\subset\,geod_{\lambda_n}(Supp(\alpha_n))
$$
for $\lambda_n(\delta,\lambda,\mu)=2C_5(\delta,n+1,\lambda_{n-1}(\delta,\lambda,\mu))+\lambda_{n-1}(\delta,\lambda,\mu)$.\\
\\
Now we investigate the behavior of $\Theta_*$ on the Rips complexes. By induction assumption
$\Theta_{n-1}(\partial(\alpha_n))\,\subset\,C_{n-1}^{R_{n-1}}(\Gamma,{\mathbb C})$ for a suitable $R_{n-1}(\delta,\lambda,\mu)$. 
A look at (4.11), 3.6 and 4.4 shows then that 
$$
\Theta_n(\alpha_n)\,\subset\,C_n^{R_n}(\Gamma,{\mathbb C})
$$
$$
R_n(\delta,\lambda,\mu)=R_{n-1}(\delta,\lambda,\mu)+2C_5(\delta,n+1,\lambda_{n-1}(\delta,\lambda,\mu))+2\delta+2.
$$

For the $\ell^1$-norm of $\Theta_n$ we find, again by the induction assumption,
$$
\parallel\Theta_n(\alpha_n)\parallel_{\ell^1}\,\leq\,
(C_{15}(R_{n-1}(\delta,\lambda,\mu)+C_5(\delta,n+1,\lambda_{n-1}(\delta,\lambda,\mu)),n-1)+n)
\parallel\Theta_{n-1}(\partial(\alpha_n))\parallel_{\ell^1}
$$
$$
\leq (C_{15}(R_{n-1}(\delta,\lambda,\mu)+C_5(\delta,n+1,\lambda_{n-1}(\delta,\lambda,\mu),n-1)+n)\cdot(n+1)\cdot
\mu_{n-1}\cdot(diam(\{g_0,\ldots,g_n\})+1),
$$
so that the last assertion of the theorem holds with
$$
\mu_n(\delta,\lambda,\mu)= (C_{15}(R_{n-1}(\delta,\lambda,\mu)+C_5(\delta,n+1,\lambda_{n-1}(\delta,\lambda,\mu),n-1)+n)\cdot(n+1)\cdot\mu_{n-1}(\delta,\lambda,\mu).
$$
\end{proof}

The main result of this section is the 

\begin{theorem}
Let $(\Gamma,S)$ be a $\delta$-hyperbolic group and let $R>0$ be so large that the augmented Rips-complex $C^R_*(\Gamma,{\mathbb C})$ is contractible (the choice $R=6\delta +4$ suffices \cite{Gr}). Then there exists a $\Gamma$ equivariant chain map 
$$
\Theta'_*:C_*(\Gamma,{\mathbb C})\to C_*^R(\Gamma,{\mathbb C}),
$$
which equals the identity in degree zero, vanishes in sufficiently high degrees, and such that for suitable constants $C_{18},C_{19}>0$
the following holds:
$$
Supp(\Theta'_*(\alpha))\,\subset\,geod_{C_{18}}(Supp(\alpha)),\,\,\,
\forall\alpha\in\Delta_\bullet(\Gamma),
\eqno(4.14)
$$
$$
\parallel\Theta'_*(\alpha)\parallel_1\,\leq\,C_{19}\cdot(diam(Supp(\alpha))+1),\,\,\,
\forall\alpha\in\Delta_\bullet(\Gamma).
\eqno(4.15)
$$
Moreover one may suppose that $\Theta'_*$ vanishes on degenerate chains.
\end{theorem}

\begin{proof}
By our hypothesis the Rips complex $C^R_*(\Gamma,{\mathbb C})$ defines, as well as the Bar-complex $C_*(\Gamma,{\mathbb C})$ a free resolution of the constant $\Gamma$-module $\mathbb C$. Consequently, there exists a $\Gamma$-equivariant chain map $\eta_*:C_*(\Gamma,{\mathbb C})\to C^R_*(\Gamma,{\mathbb C})$, which equals the identity in nonpositive degree. The antisymmetrization operator
$$
\begin{array}{cccc}
\pi_{as}: & C^R_*(\Gamma,{\mathbb C}) & \to & C^R_*(\Gamma,{\mathbb C}) \\
 & & & \\
 & [g_0,\ldots,g_n] & \mapsto & \frac{1}{(n+1)!}\,\underset{\sigma\in\Sigma_{n+1}}{\sum}\,
 Sig(\sigma)\,[g_{\sigma(0)},\ldots,g_{\sigma(n)}] \\
\end{array}
\eqno(4.16)
$$
is a chain map which equals the identity in degree zero. It vanishes in degrees above 
$d\,=\,\sharp\{ g\in\Gamma,\,\ell_S(g)<R\}$, because the vertices of Rips-simplices cannot be pairwise different in this case. Fix now a $\lambda$-geodesic  bicombing $\Theta_1$ of $\Gamma$ of weight $\mu$ and let $\Theta_*: C_*(\Gamma,{\mathbb C})\to C_*(\Gamma,{\mathbb C})$ the corresponding chain map. Define the desired chain map $\Theta'_*$ as the composition
$$
\Theta'_*\,=\,\pi_{as}\circ\eta_*\circ\pi_{as}\circ\Theta_*:\,C_*(\Gamma,{\mathbb C}) \to  C^R_*(\Gamma,{\mathbb C}).
\eqno(4.17)
$$
It vanishes in degrees above or equal to $d$. Let $R'=\underset{n<d}{Max}\,R_n(\delta,\lambda,\mu)$,\\
$\lambda'\,=\,\underset{n<d}{Max}\,\lambda_n(\delta,\lambda,\mu),$
$\mu'\,=\,\underset{n<d}{Max}\,\mu_n(\delta,\lambda,\mu)$ in the notations of 4.6. One has $Im(\Theta_*)\subset C^{R'}_*(\Gamma,{\mathbb C})$ by the previous theorem. The complex $C^{R'}_*(\Gamma,{\mathbb C})$ is given in degree $n$ by the finitely generated, free $\Gamma$-module with canonical basis the Bar-$n$-simplices of diameter at most $R'$ with first vertex $g_0=e$.
The support of the union of the images of these finitely many simplices is finite and therefore of finite diameter $C_{16}$. The $\ell^1$-norms of the images under $\eta_*$ of these finitely many simplices are bounded by some constant $C_{17}$.
It follows that 
$$
Supp(\Theta'_*(\alpha))\subset Supp(\eta_*(\Theta_*(\alpha)))\subset B((Supp(\Theta_*(\alpha))),C_{16})
\eqno(4.18)
$$
$$
\subset B(geod_{\lambda'}(Supp(\alpha)),C_{16})\,\subset\,
geod_{\lambda'+2C_{16}}(Supp(\alpha)).
$$
For the $\ell^1$-norm we note that
$$
\parallel\Theta'_*(\alpha)\parallel_1\,\leq\,
\parallel\eta_*(\Theta_*(\alpha)\parallel_1\,\leq\,
C_{17}\cdot\parallel\Theta'_*(\alpha)\parallel_1\,\leq\,
C_{17}\cdot\mu'\cdot(diam(Supp(\alpha))+1),
$$
which establishes the desired estimates. A further look at formula (4.17) shows finally that the subcomplex of degenerate chains is stable under $\Theta_*$. As the operator $\pi_{as}\circ\Theta_*$
 and therefore also $\Theta'_*$ vanishes on degenerate chains the last claim follows.
\end{proof}

The key point of the previous theorem is the fact that the {\bf $\ell^1$-norm of the image of a simplex is bounded by a linear function of its diameter}. 

\section{Norm estimates in Bar complexes}

Now the continuity properties of the operators constructed so far will be studied.
For this we introduce a family of norms on the bar resolution of a finitely
generated group.

\begin{definition}
Let $(\Gamma,S)$ be a finitely generated group and let $Ad(\Gamma)$ be its underlying set, equipped with the adjoint action. 
\begin{itemize}
\item[a)] The {\bf weight} of a simplex $\alpha=[g_0,\ldots,g_n;v]\in\Delta(\Gamma)_n\times
Ad(\Gamma)$ equals
$$
\vert\alpha\vert:=
d_S(g_0,g_1)+\ldots+d_S(g_{n-1},g_n)+d_S(g_n,vg_0)
\eqno(5.1)
$$ 
The weight of a chain is the maximum of the weights of its simplices.
\item[b)] For $\lambda\geq 1$ let $\parallel-\parallel_\lambda:\,C_*(\Gamma)\otimes Ad(\Gamma)\to{\mathbb R}_+$ be the weighted $\ell^1$-norm
$$
\parallel\underset{\alpha\in\Delta(\Gamma)\times Ad(\Gamma)}{\sum}\,a_\alpha\cdot u_\alpha\parallel_\lambda\,=\,
\underset{\alpha\in\Delta(\Gamma)\times Ad(\Gamma)}{\sum}\,\vert a_\alpha\vert\cdot\lambda^{\vert\alpha\vert}
\eqno(5.2)
$$
\end{itemize}
\end{definition} 

The action of $\Gamma$ on $(C_*(\Gamma)\otimes Ad(\Gamma),
\parallel\,\,\parallel_{\lambda})$ is
isometric.

\begin{prop}
Let $(\Gamma,S)$ be a word-hyperbolic group. Let $\Theta'_*:\,\overline{C}_*(\Gamma,{\mathbb C})\to \overline{C}_*^R(\Gamma,{\mathbb C})$ be a chain map as in 4.7. Then the linear operator 
$$
\widetilde{\nabla}_*\,=\,h(\Theta'_*,\,Id)\otimes Id_{Ad({\mathbb C}\Gamma)}:\,\, \overline{C}_*(\Gamma,{\mathbb C})\otimes Ad({\mathbb C}\Gamma)  \to  \overline{C}_{*+1}(\Gamma,{\mathbb C})\otimes Ad({\mathbb C}\Gamma) 
\eqno(5.3)
$$
satisfies
$$
\partial\circ\widetilde{\nabla}_*+\widetilde{\nabla}_*\circ\partial\,=\,Id_*
$$
in sufficiently large degrees $*\geq d>>0$ and there are constants $C_{20},\,C_{21}>0$ such that 
$$
\vert\widetilde{\nabla}_*(\alpha)\vert\,\leq\,\vert\alpha\vert+C_{20},
\eqno(5.4)
$$
and
$$
\parallel\widetilde{\nabla}_*(\alpha)\parallel_1\,\leq\,C_{21}\cdot\left(d_S(g_0,g_1)+\ldots+d_S(g_{k-1},g_k) \right),\,k=min(n,d),
\eqno(5.5)
$$
for all twisted Bar-simplices $\alpha=[g_0,\ldots,g_n;v]\in \Delta_n(\Gamma)\times Ad(\Gamma),\,n\in{\mathbb N}$.
\end{prop}

\begin{proof}
The chain map $\Theta'$ vanishes in in degree $*\geq d,$ so that in fact
$$
\widetilde{\nabla}([g_0,\ldots,g_n;v])=\underset{k=0}{\overset{min(n,d-1)}{\sum}}\,(-1)^i[\Theta'(g_0,\ldots,g_k),g_k,\ldots,g_n;v].
$$
From (4.15) one derives the estimate
$$
\parallel\widetilde{\nabla}([g_0,\ldots,g_n;v])\parallel_1\,\leq\,
\underset{k=0}{\overset{min(n,d-1)}{\sum}}\,C_{19}\cdot(diam(\{g_0,\ldots,g_k\})+1)
$$
for the $\ell^1$-norm, which proves the first assertion for $C_{21}=2dC_{19}$. To estimate the weights let $[g'_0,\ldots,g'_k]$ be a simplex appearing with non-zero multiplicity in the chain 
$\Theta'(g_0,\ldots,g_k)$. We find
$$\vert[g'_0,\ldots,g'_k,g_k,\ldots,g_n;v]\vert\,=\,
\underset{i=0}{\overset{k-1}{\sum}}\,d(g'_i,g'_{i+1})\,
+\,d(g'_k,g_k)\,+\,\underset{j=k}{\overset{n-1}{\sum}}\,d(g_j,g_{j+1})\,
+\,d(g_n,vg'_0)
$$
$$
\leq\,
\underset{i=0}{\overset{k-1}{\sum}}\,d(g'_i,g'_{i+1})\,
+\,d(g'_k,g'_0)+d(g'_0,g_k)+\,\underset{j=k}{\overset{n-1}{\sum}}\,d(g_j,g_{j+1})\,
+\,d(g_n,vg_0)+d(vg_0,vg'_0)
$$
$$
\leq\,d(g_0,g'_0)+d(g'_0,g_k)+\underset{j=k}{\overset{n-1}{\sum}}\,d(g_j,g_{j+1})\,
+\,d(g_n,vg_0)+(k+1)R
$$
because $[g'_0,\ldots,g'_k]\in\Delta_k^R(\Gamma)$. According to theorem 4.7 
$$
g'_0\in Supp(\Theta'_k([g_0,\ldots,g_k]))\subset geod_{C_{18}}(\{g_0,\ldots,g_k\}),
$$
so that there exist indices $0\leq l<m\leq k$ such that 
$$
d(g_l,g'_0)+d(g'_0,g_m)\leq d(g_l,g_m)+C_{18}.
$$
Therefore
$$
d(g_0,g'_0)+d(g'_0,g_k)+\underset{j=k}{\overset{n-1}{\sum}}\,d(g_j,g_{j+1})\,
+\,d(g_n,vg_0)+(k+1)R
$$
$$
\leq\,d(g_0,g_l)+d(g_l,g'_0)+d(g'_0,g_m)+d(g_m,g_k)+\underset{j=k}{\overset{n-1}{\sum}}\,d(g_j,g_{j+1})\,
+\,d(g_n,vg_0)+(k+1)R
$$
$$
\leq\,d(g_0,g_l)+d(g_l,g_m)+d(g_m,g_k)+\underset{j=k}{\overset{n-1}{\sum}}\,d(g_j,g_{j+1})\,
+\,d(g_n,vg_0)+(k+1)R+C_{18}
$$
$$
\leq\,\underset{j=0}{\overset{n-1}{\sum}}\,d(g_j,g_{j+1})\,
+\,d(g_n,vg_0)+(d+1)R+C_{18}\,=\,\vert [g_0,\ldots,g_n;v]\vert+(d+1)R+C_{18}
$$
as was to be shown.
\end{proof}

\begin{cor}
Let $(\Gamma,S)$ be a word-hyperbolic group and let $\widetilde{\nabla},\,\widetilde{B}$ be the
operators on $C_*(\Gamma)\otimes Ad(\Gamma)$ introduced in
(5.3) and (2.35). Let $\lambda_0>\lambda_1>1$ be a pair of real numbers and let
$\parallel-\parallel_{\lambda_0},\parallel-\parallel_{\lambda_1}$ be the
corresponding seminorms on $C_*(\Gamma)\otimes Ad(\Gamma)$. There exist constants $C_{22},C_{25}$ such that   
$$
\parallel\widetilde{\nabla}(\alpha)\parallel_{\lambda_1}\,
\leq\,C_{22}\cdot\parallel\alpha\parallel_{\lambda_0}
\eqno(5.6)
$$
and
$$
\frac{1}{k!}\parallel(\widetilde{\nabla}\circ
\widetilde{B})^k\,(\alpha)\parallel_{\lambda_1}\,\leq\,
C_{25}^k\cdot\parallel\alpha\parallel_{\lambda_0}
\eqno(5.7)
$$
for all $\alpha\in C_*(\Gamma)\otimes Ad(\Gamma)$ and
$k\in\N$. 
\end{cor}
\begin{proof}

The norms $\parallel-\parallel_\lambda$ are weighted $\ell^1$-norms so that it suffices to verify the estimates on simplices
$\alpha=[g_0,\ldots,g_n;v]\in\Delta_n(\Gamma)\times Ad(\Gamma)$. Note that 
$$
diam(Supp(\alpha))=diam(\{g_0,\ldots,g_n\})\leq\vert\alpha\vert.
$$ 
Therefore, if 
$$
\widetilde{\nabla}(\alpha)=\underset{i\in I}{\sum}\,\mu_i\alpha_i,\,\mu_i\in{\mathbb C},\,\alpha_i\in\Delta_{n+1}(\Gamma),
$$ 
then 
$$
\parallel\widetilde{\nabla}(\alpha)\parallel_{\lambda_0}\,\leq\,
\lambda_0^{\vert\alpha\vert+C_{20}}\parallel\widetilde{\nabla}(\alpha)\parallel_{\ell^1}\,\leq\,\lambda_0^{\vert\alpha\vert+C_{20}}\cdot C_{21}\cdot diam(Supp(\alpha))
$$
$$
\leq C_{21}\cdot\lambda_0^{C_{20}}\cdot\vert\alpha\vert\cdot\lambda_0^{\vert\alpha\vert}\,
\leq\, C_{21}\cdot\lambda_0^{C_{20}}\cdot(\log\lambda_1-\log\lambda_0)^{-1}\cdot\lambda_1^{\vert\alpha\vert}\,=\,C_{22}\parallel\alpha\parallel_{\lambda_1}
$$
according to 5.2. This proves the first assertion. For the operator $\widetilde{\nabla}\circ \widetilde{B}$ we find
$$
\widetilde{\nabla}\circ\widetilde{B}(\alpha)\,=\,\underset{i=0}{\overset{n}{\sum}}\,
(-1)^{in}\,\widetilde{\nabla}([v^{-1}g_i,\ldots,v^{-1}g_n,g_0,\ldots,g_{i};v]).
$$
Proposition 5.4 leads therefore to the estimate
$$
\vert\widetilde{\nabla}\circ \widetilde{B}(\alpha)\vert\,\leq\,\vert\alpha\vert+C_{20}
$$
because the operator $\widetilde{B}$ preserves weights. To simplify notations put
$$
[h^{(i)}_0,\ldots,h^{(i)}_{n+1};v]=[v^{-1}g_i,\ldots,v^{-1}g_n,g_0,\ldots,g_i;v].
$$ 
For the $\ell^1$-norm of $\widetilde{\nabla}\circ\widetilde{B}$ we  
we find then 
$$
\parallel\widetilde{\nabla}\circ \widetilde{B}(\alpha)\parallel_{\ell^1}\,\leq\,
\underset{i=0}{\overset{n}{\sum}}\parallel\widetilde{\nabla}([h^{(i)}_0,\ldots,h^{(i)}_{n+1};v])\parallel_{\ell^1}
$$
$$
\leq\,C_{21}\underset{i=0}{\overset{n}{\sum}}\underset{j=0}{\overset{min(n,d)}{\sum}}\,d_S(h_j^{(i)},h_{j+1}^{(i)})\,\leq\,
C_{21}\cdot(d+1)\cdot\vert\alpha\vert.
$$
The previous calculations yield inductively the estimates
$$
\vert(\widetilde{\nabla}\circ \widetilde{B})^k(\alpha)\vert\,\leq\,\vert\alpha\vert+kC_{20}
$$
and
$$
\parallel(\widetilde{\nabla}\circ \widetilde{B})^k(\alpha)\parallel_{\ell^1}\,\leq\,
\underset{l=0}{\overset{k-1}{\prod}}\,
C_{21}(d+1)(\vert\alpha\vert+lC_{20})\,\leq\,
\left(C_{21}(d+1)(\vert\alpha\vert+kC_{20})\right)^k
$$
for all $k\in{\mathbb N}$. Put now $C_{23}=C_{21}\cdot(d+1)\cdot\lambda_1^{C_{20}}$, let $C_{24}>0$ 
satisfy $\exp(\frac{C_{23}}{C_{24}})\cdot\lambda_1\leq\lambda_0,$ and 
choose $C_{25}\geq C_{24}\cdot\exp\left(\frac{C_{23}\cdot C_{20}}
{C_{24}}\right)$. Then we find
$$
\frac{1}{k!}\parallel(\widetilde{\nabla}\circ \widetilde{B})^k(\alpha)\parallel_{\lambda_1}\,\leq\,\frac{1}{k!}
\lambda_1^{(\vert\alpha\vert+kC_{20})}\cdot\left(C_{21}(d+1)(\vert\alpha\vert+kC_{20})\right)^k
$$
$$
=\,\frac{\left(C_{23}(\vert\alpha\vert+kC_{20})\right)^k}{k!}\cdot\lambda_1^{\vert\alpha\vert}\,=\,C_{24}^k\cdot
\frac{\left(\frac{C_{23}}{C_{24}}(\vert\alpha\vert+kC_{20})
\right)^k}{k!}\cdot\lambda_1^{\vert\alpha\vert}
$$
$$
\leq\,C_{24}^k\cdot\exp
\left(\frac{C_{23}}{C_{24}}(\vert\alpha\vert+kC_{20})
\right)\cdot\lambda_1^{\vert\alpha\vert}\,=\,
\left(C_{24}\cdot\exp(\frac{C_{23}C_{20}}{C_{24}})\right)^k\cdot\left(
\exp(\frac{C_{23}}{C_{24}})\cdot\lambda_1\right)^{\vert\alpha\vert}
$$
$$
\leq\,C_{25}^k\cdot\lambda_0^{\vert\alpha\vert}\,=\,C_{25}^k\cdot\parallel\alpha\parallel_{\lambda^0}
$$
as claimed.
\end{proof}

\begin{prop}
Let $(\Gamma,S)$ be a word-hyperbolic group and let $v\in\Gamma$ be an element of infinite order and minimal word length in its conjugacy class 
$\langle v\rangle$. Let $\sigma:\langle v\rangle\to\Gamma$ be a section 
of the quotient map $\Gamma\to\langle v\rangle,\,g\mapsto gvg^{-1}$ satisfying (3.15) and (3.16), and let\\ $\sigma':N(v)\to Z(v)$ be a section of the epimorphism $Z(v)\to N(v)=
Z(v)/v^{\mathbb Z}$, which satisfies (3.17). Let 
$$
p_v: C_*(\Gamma,{\mathbb C}v)\to \widetilde{C}_*({\mathbb C}\Gamma)_{\langle v\rangle},\,I:\widetilde{C}_*({\mathbb C}\Gamma)_{\langle v\rangle}\to \widetilde{C}^\lambda_*({\mathbb C}\Gamma)_{\langle v\rangle}, \,\,\text{and}
$$
$$
s_{v,\sigma,\sigma'}:\,\widetilde{C}_*^\lambda({\mathbb C}\Gamma)_{\langle v\rangle}\to C_*(\Gamma,{\mathbb C}v)
$$ 
be the operators introduced in (2.47), (2.14) and (2.71). Fix finally real numbers $\lambda_0>\lambda_1>1$ such that $\lambda_1^3<\lambda_0$.
Then there exists a constant $C_{26}(\Gamma,S,\delta)$ independent of $\langle v\rangle$ such that
$$
\parallel s_{v,\sigma,\sigma'}\circ I\circ p_v(\omega)\parallel_{\lambda_1}\,
\leq\, C_{26}\cdot\parallel\omega\parallel_{\lambda_0}
\eqno(5.8)
$$
for all $\omega\in C_*(\Gamma,{\mathbb C}v)$.
\end{prop}

\newpage

\begin{proof}
We proceed in several steps:\\
\\
{\bf Step 1:}\\
\\
At first we study the effect of the map $s_{v,\sigma,\sigma'_u}\circ I\circ p_v$ on the weight of chains. The weight of a Bar-simplex does not change under the action of the centralizer $Z(v)$ or the lifted cyclic operator $\widetilde{T}_*$. Moreover, it does not increase under the boundary map. These observations suffice to deduce that the operators $s'_{v,\sigma,\sigma'_u}\circ I\circ p_v$ and $(\partial\circ s'_{v,\sigma,\sigma'_u}-s'_{v,\sigma,\sigma'_u}\circ\partial)\circ I\circ p_v$ do not increase weights. The operator $Id-\widetilde{T}_*$ preserves weights and the same holds for its inverse on the subcomplex $Ker(I_v)_*$. For the homotopy operator $\chi$ (see 2.68) and a simplex $\beta=[h_0,\ldots,h_n;v]\in $ we find
$$
\vert \chi(\beta)\vert\,=\,\vert[h_n,h_0,\ldots,h_n;v]\vert\,=
$$
$$
=\,d_S(h_n,vh_n)+d_S(h_n,h_0)-d_S(h_n,vh_0)+\vert\beta\vert
\,\leq 2d_S(h_0,h_n)+\vert\beta\vert\,\leq\,3\vert\beta\vert,
$$
so that finally
$$
\vert 
s_{v,\sigma,\sigma'}\circ I\circ p_v(\alpha)\vert\,\leq\,3\vert\alpha\vert
\eqno(5.9)
$$
for all $\alpha\in \Delta_n(\Gamma)$. \\
\\
{\bf Step 2:}\\
\\
Estimating the $\ell^1$-norm of the operator in question is more involved. \\
Let $\alpha=[g_0,\ldots,g_n]\in\Delta_n(\Gamma)$ be a Bar-simplex. One has
$$
p_v(\alpha)\,=\,g_n^{-1}vg_0\otimes g_0^{-1}g_1\otimes\ldots g_{n-1}^{-1}g_n
$$
and $N_{cyc}\circ I\circ p_v(\alpha)$ is an average of tensors of the form
$$
g_i^{-1}g_{i+1}\otimes\ldots\otimes g_{n-1}^{-1}g_n\otimes g_n^{-1}vg_0\otimes g_0^{-1}g_1\otimes\ldots\otimes g_{i-1}^{-1}g_i,\,\,\,0\leq i\leq n,
$$
so that $s'_{v,\sigma,\sigma'}\circ I\circ p_v(\alpha)$ becomes an average of simplices of the type
$$
\sigma'(h)\sigma(g_i^{-1}vg_i)g_i^{-1}
[g_{i+1},\ldots,g_n,vg_0,\ldots, vg_{i}],\,\,\,0\leq i\leq n,\,\,h\in N(v).
$$
The support of the chain 
$$
(\partial\circ s'_{v,\sigma,\sigma'}-s'_{v,\sigma,\sigma'}\circ\partial)\circ I\circ p_v(\alpha)
$$ 
is therefore contained in the set
$$
U\,=\,\{\sigma'(h)\sigma(g_i^{-1}vg_i)g_i^{-1}g_j,\,\,\sigma'(h)\sigma(g_i^{-1}vg_i)g_i^{-1}vg_k,0\leq k\leq i<j\leq n,\,h\in N(v)\}.
$$
Its diameter is bounded by 
$$
diam(U)\leq 2\,\underset{g\in U}{\sup}\,\,\ell_S(g)\,\leq\,
\,2\underset{h\in N(v)}{\sup}\,\,\ell_S(\sigma'(h))\,+
\,2\underset{0\leq i\leq n}{\sup}\,\,\ell_S(\sigma(g_i^{-1}vg_i))\,+
$$
$$
+\,2Max\left( 
\underset{\,0\leq i<j\leq n}{\sup}\,\ell_S(g_i^{-1}g_j)\,,\,
\underset{\,0\leq i\leq k\leq n}{\sup}\,\ell_S(g_k^{-1}vg_i)
\right)
$$
$$
\leq\,
2(\ell_s(v)+C_{14})\,+\,2\underset{0\leq i\leq n}{\sup}\,(\ell_S(g_i^{-1}vg_i)+C_{11})\,+
\,2\vert\alpha\vert\,\leq\,6\vert\alpha\vert\,+\,2(C_{11}+C_{14})\,\leq\,C_{26}\vert\alpha\vert
$$
by (3.17), (3.15), and the minimality of the word length of $v$ in its conjugacy class.\\ 
\\
{\bf Step 3:}\\
\\
Every chain $c\in Ker(I_v)_n$ can be written as 
$$
c=\underset{i\in I}{\sum}\, \mu_i\,(Id-\widetilde{T}_n^{k_i})(\alpha_i),
\,\mu_i\in {\mathbb C},\,\alpha_i\in\Delta_n(\Gamma),
$$
with 
$$
\parallel c\parallel_{\ell^1}\,=\,\underset{i\in I}
{\sum}\,\vert \mu_i\vert,\, Supp(\alpha_i)\subset Supp(c),\,Supp(\widetilde{T}_n^{k_i}\alpha_i)\subset Supp(c).
$$
Consequently
$$
(1-\widetilde{T}_n)^{-1}(c)\,=\,
\underset{i\in I}{\sum}\,\mu_i\,\left(\underset{j=0}{\overset{k_i-1}{\sum}}\,
\widetilde{T}_n^j\right)(\alpha_i),
$$
and 
$$
\parallel (1-\widetilde{T}_n)^{-1}(c)\parallel_{\ell^1}\,\leq\,
\underset{i\in I}{\sum}\,k_i\vert \mu_i\vert\,\leq\,
\left(\underset{i\in I}{\sup}\,k_i\right)\,\underset{j\in I}{\sum}\,\vert \mu_i\vert\,=\,\left(\underset{i\in I}{\sup}\,k_i\right)\parallel c\parallel_{\ell^1}.
$$
\\
{\bf Step 4:}\\
\\
We want to bound $\left(\underset{i\in I}{\sup}\,k_i\right)$ in terms of the support of the chain $c$. One finds
$$
\left(\underset{i\in I}{\sup}\,k_i\right)\leq Max\{ k\in{\mathbb N},\,\exists \alpha',\alpha''\in \Delta_n(\Gamma),\,Supp(\alpha')\cup Supp(\alpha'')\subset Supp(c),\,\widetilde{T}_n^k(\alpha')=\alpha''\}.
$$
For a triple $(\alpha',\alpha'',k)$ let 
$k=j(n+1)+r,\,0\leq r\leq n$. Then the identity $\widetilde{T}_n^k(\alpha')=\alpha''$
implies $v^{-j}\cdot \widetilde{T}_n^r(\alpha')=\alpha''$ because $\widetilde{T}_n^{n+1}=v^{-1}$ on $\Delta_n(\Gamma)\times Ad(\Gamma)$, and therefore 
$$
v^{-j}\cdot Supp(\alpha')\cap Supp(\alpha'')\neq\emptyset
$$
as $Supp(\alpha')\cap Supp(\widetilde{T}_n^r(\alpha'))\neq\emptyset.$ 
Let $h\in Supp(\alpha')\subset Supp(c)$ and $h'\in Supp(\alpha'')\subset Supp(c)$ be such that $v^{-j}h=h'$. According to Gromov (see 3.8) the stable length of $v$ is bounded from below: 
$$
0<\epsilon_0\,\leq\,\frac{\ell_S(v^{-j})}{j}\,=\,\frac{d_S(h,h')}{j}\,\leq\,\frac{diam(Supp ( c ))}{j}
$$
and
$$
k\,\leq\,(n+2)\cdot j\,\leq\,\frac{(n+2)\cdot diam(Supp( c ))}{\epsilon_0}
$$
and therefore
$$
\left(\underset{i\in I}{\sup}\,k_i\right)\,\leq\,
\epsilon_0(\Gamma,S)^{-1}\cdot (n+2)\cdot diam(Supp(c)).
$$

{\bf Step 5:}\\
\\
Altogether we have established the estimate

$$
\parallel \chi_v\circ(\partial\circ s'_{v,\sigma,\sigma'}-s'_{v,\sigma,\sigma'}\circ\partial)\circ I\circ p_v(\alpha)\parallel_{\ell^1}
$$

$$
=\parallel (Id-\widetilde{T}_n)\circ\chi\circ(1-\widetilde{T}_{n-1})^{-1}\circ(\partial\circ s'_{v,\sigma,\sigma'}-s'_{v,\sigma,\sigma'}\circ\partial)\circ I\circ p_v(\alpha)\parallel_{\ell^1}
$$

$$
\leq 2\parallel \chi\circ (1-\widetilde{T}_{n-1})^{-1}\circ(\partial\circ s'_{v,\sigma,\sigma'}-s'_{v,\sigma,\sigma'}\circ\partial)\circ I\circ p_v(\alpha)\parallel_{\ell^1}
$$

$$
\leq 2\,\epsilon_0(\Gamma,S)^{-1}\cdot (n+2)\cdot diam(U)\cdot
\parallel (\partial\circ s'_{v,\sigma,\sigma'}-s'_{v,\sigma,\sigma'}\circ\partial)\circ I\circ p_v(\alpha)\parallel_{\ell^1}
$$
by step 4
$$
\leq 2\,\epsilon_0(\Gamma,S)^{-1}\cdot (n+2)\cdot C_{26}\cdot\vert\alpha\vert\cdot \parallel (\partial\circ s'_{v,\sigma,\sigma'}-s'_{v,\sigma,\sigma'}\circ\partial)\circ I\circ p_v(\alpha)\parallel_{\ell^1}
$$
by step 3
$$
\leq 4(n+1)(n+2)\,\epsilon_0(\Gamma,S)^{-1}\cdot C_{26}\cdot\vert\alpha\vert\cdot\parallel s'_{v,\sigma,\sigma'}(\alpha)\parallel_1\,\leq\,
4(n+1)(n+2)\,\epsilon_0(\Gamma,S)^{-1}\cdot C_{26}\cdot\vert\alpha\vert
$$
so that finally
$$
\parallel s_{v,\sigma,\sigma'}\circ I\circ p_v(\alpha)\parallel_{\ell^1}\,\leq\,
4(n+1)(n+2)\,\epsilon_0(\Gamma,S)^{-1}\cdot C_{26}\cdot\vert\alpha\vert\,+\,1
\eqno(5.10)
$$
for all $\alpha\in\Delta_n(\Gamma)$.
The conclusion follows then from (5.9), (5.10) and lemma 5.2.
\end{proof}

\section{Local cyclic homology of group Banach\\ algebras}

\subsection{Local cyclic cohomology \cite{Pu1}}

Local cyclic cohomology is a cyclic cohomology theory for topological algebras, for example Banach algebras. It is the target of a Chern character from
topological K-theory and possesses properties similar to the K-functor,
in particular continuous homotopy invariance, excision property and stability
\cite{Pu1}, \cite{Pu2}. It commutes with topological direct limits
and on the class of Banach algebras with approximation property it is stable under passage to smooth dense subalgebras. \cite{Pu1}.\\

The local cyclic homology of a Banach algebra is given by an inductive limit
of cyclic type homology groups of dense subalgebras, which makes it to some
extent computable. We recall now its definition.\\
\\
Let $A$ be a Banach algebra and let $U$ be its open unit ball. For a compact subset $K\subset U$ let ${\mathbb C}[K]\subset A$ be the abstract complex subalgebra of $A$ generated by $K$. There exists a homomorphism $\iota_K:\,{\mathbb C}[K]\to A_K$ to a Banach algebra, which has the following universal property: if $f:{\mathbb C}[K]\to B$ is a homomorphism into a Banach algebra such that $\parallel f(K)\parallel_B\leq 1$, then there is a unique contractive (i.e. of norm less or equal to one) homomorphism of Banach algebras $\widehat{f}:A_K\to B$ such that 
$\widehat{f}\circ\iota_K=f$. The Banach algebra $A_K$ is unique up to isometry and obtained from ${\mathbb C}[K]$ by completion with respect to the largest submultiplicative seminorm on ${\mathbb C}[K]$ for which the norm of $K$ is bounded by one. If $K\subset K'\subset U$ is another compact subset of the unit ball of $A$, then the inclusion ${\mathbb C}[K]\to{\mathbb C}[K']$ extends to a contractive homomorphism $A_K\to A_{K'}$ of Banach algebras. By definition $\underset{K\subset U}{\lim}\,A_K\,=\,A$ as Banach algebras. \\
\\
For an auxiliary Banach algebra
$B$ let $\parallel-\parallel_{\varrho,m},\,\varrho\geq 1,\,m\in\N,$ be
the largest seminorm on $\Omega B$ satisfying   
$$
\parallel b^0db^1\ldots db^n\parallel_{\varrho,m}\,\leq \,\frac{1}{c(n)!}(2+2c(n))^m
\varrho^{-c(n)} \parallel b^0\parallel_{B}\cdot\ldots\cdot\parallel
b^n\parallel_{B} 
\eqno(6.1)
$$
where $c(2n)=c(2n+1)=n$. Note that
$$
\parallel-\parallel_{(\varrho,0)}\,\leq\,\parallel-\parallel_{(\varrho,m)}\,\leq\,C(\varrho,\varrho',m)\cdot
\parallel-\parallel_{(\varrho',0)}
\eqno(6.2)
$$ 
for $\varrho'>\varrho$. Let $\Omega B_{(\varrho)}$ be the Fr\'echet space obtained by completion of $\Omega B$ with respect to the seminorms
$\parallel-\parallel_{\varrho,m},\,m\in\N$. The space $\Omega^nB_{(\varrho)}$ is then
just a projective tensor power of $B$ and $\Omega B_{(\varrho)}$ becomes a weighted topological direct sum of the subspaces $\Omega^nB_{(\varrho)}$. The cyclic differentials $b$ and $B$ extend to bounded operators on $\Omega B_{(\varrho)}$. (This would not hold after completion with respect to only one of the seminorms introduced above.) Denote by 
$CC_*^{(\varrho)}(B):=(\Omega B_{(\varrho)},\,b+B)$ the corresponding cyclic bicomplex. If $1\leq\varrho\leq\varrho'$, then the identity on $\Omega_B$ extends to a bounded linear map $\Omega B_{(\varrho)}\to\Omega B_{(\varrho')}$. It gives rise to a bounded chain map 
$CC_*^{(\varrho)}(B)\to CC_*^{(\varrho')}(B).$\\
\\
In the sequel the notion of a formal inductive limit 
or ind-object will be used. Recall that for a category $\Ch $ 
the category $ind\,{\mathcal C}$ of ind-objects or formal inductive limits over
$\Ch$ is defined as follows. The objects of $ind\,\Ch$ are small directed diagrams over $\Ch$:
\[
\begin{array}{rcl}
\mbox{ob} \;\; ind\,\Ch & = &\{
\mbox{``}\lim\limits_{\stackrel{\longrightarrow}{I}}\! \mbox{''} A_i \; \Big|
\; I \; \mbox{a partially ordered directed set}\; \}\\[0.3cm]
& =& \{ A_i, f_{ij} : A_i \to A_j, \, i \le j \in I \; \Big| \;f_{jk} \circ f_{ij} =
f_{ik} \}
\end{array}
\]

The morphisms between two ind-objects are given by
\[
\mbox{mor}_{ind\,\Ch} (\mbox{``}\lim\limits_{\stackrel{\longrightarrow}{I}}\!
\mbox{''} A_i,\mbox{ ``}\lim\limits_{\stackrel{\longrightarrow}{J}}\!
\mbox{''} B_j )\; := \; \lim\limits_{\stackrel{\longleftarrow}{I}}
\lim\limits_{\stackrel{\longrightarrow}{J}} \mbox{mor}_{\Ch} (A_i,B_j) \]
where the limits on the right hand side are taken in the
category of sets. There is a canonical fully faithful functor 
$\Ch\to ind\,\Ch$, which sends every object to the corresponding constant ind-object.
\\
\\
For a Banach algebra $A$ with open unit ball $U$ we define the associated 
ind-Banach algebra by 
$$
{\mathcal A}\,=\,"\underset{K\subset U}{\lim}"\,A_K,
\eqno(6.3)
$$
where the formal inductive limit is taken over the set of compact subsets of $U$, ordered by inclusion.\\
\\
For a Banach algebra $B$ we define a cyclic ind-complex
$$
{\mathcal C}{\mathcal C}(B)\,=\,"\underset{\varrho\to\infty}{\lim}"\,
CC_*^{(\varrho)}(B)
\eqno(6.4)
$$
as object in the homotopy category $Ho(ind\,{\mathcal C})$ of ${\mathbb Z}/2{\mathbb Z}$-graded Ind-complexes of Fr\'echet spaces.
The {\bf analytic cyclic bicomplex} of a Banach algebra $A$ is finally defined as
$$
{\mathcal C}{\mathcal C}^\omega_*(A)\,=\,{\mathcal C}{\mathcal C}_*({\mathcal A})\,=\,"\underset{K\subset U}{\lim}"\,{\mathcal C}{\mathcal C}_*(A_K)
\,=\,"\underset{\underset{\varrho\to\infty}{K\subset U}}{\lim}"\,CC_*^{(\varrho)}(A_K)
\eqno(6.5)
$$
where $\mathcal A$ denotes the ind-Banach algebra associated to $A$.\\
\\
An object of $Ho(ind\,{\mathcal C})$ is called {\bf weakly contractible} if every chain map from a constant ind-complex to it is chain homotopic to zero. The {derived ind-category} ${\mathcal D}(ind\,{\mathcal C})$ is the triangulated category obtained from $Ho(ind\,{\mathcal C})$ by inverting all morphisms whose mapping cone is weakly contractible \cite{Pu1}. We are now ready to give the 

\begin{definition}
The {\bf bivariant local cyclic cohomology} $HC^{loc}_*(A,B)$ of a pair of Banach algebras $(A,B)$ is the group of morphisms of the associated analytic cyclic bicomplexes in the derived ind-category
$$
HC^{loc}_*(A,B)\,=\,Mor_*^{{\mathcal D}(ind\,{\mathcal C})}({\mathcal C}{\mathcal C}^\omega(A),\,{\mathcal C}{\mathcal C}^\omega(B)).
\eqno(6.6)
$$
\end{definition}

This cohomology theory has particularly nice properties. It is a 
bifunctor (contravariant in the first and covariant in the second 
variable) on the category of (Ind)-Banach algebras. It is a continuous homotopy bifunctor, which satisfies excision in both variables 
for linearly contractible extensions of (Ind)-Banach algebras.
Moreover it is compatible with topological direct limits for algebras with the 
approximation property \cite{Pu1}. \\
The composition of morphisms in ${\mathcal D}(ind\,{\mathcal C})$  defines an associative product
$$
HC^{loc}_*(A,B)\,\widehat{\otimes}\,HC^{loc}_*(B,C)\,\to\,
HC^{loc}_*(A,C)
\eqno(6.7)
$$
on bivariant local cyclic cohomology. (The tensor product is the graded one.) There is a multiplicative {\bf bivariant Chern-Connes character}
$$
ch_{biv}:\,KK_*(A,B)\,\longrightarrow\,HC^{loc}_*(A,B)
\eqno(6.8)
$$
from Kasparov's bivariant $K$-theory for (separable) $C^*$-algebras
to bivariant local cyclic cohomology. It is uniquely characterized by 
its naturality and multiplicativity. It gives rise to an isomorphism
$$
ch_{biv}:\,KK_*(-,-)\otimes{\mathbb Q}\,\overset{\simeq}{\longrightarrow}\,HC_*^{loc}(-,-)
$$
on the bootstrap category of separable $C^*$-algebras which are $KK$-equivalent to commutative ones.

\subsection{Local cyclic homology of group Banach algebras}

Let $\Gamma$ be a countable discrete group and let $\ell^1(\Gamma)$ be the convolution algebra of summable functions on $\Gamma$. The operators of multiplication with characteristic functions of a finite subsets of $\Gamma$ form a net of contractive finite rank operators converging pointwise to the identity. In other words, group Banach algebras of discrete groups have the metric approximation property. For such Banach algebras the local cyclic cohomology groups can be calculated by 
a countable ind-complex.

\begin{lemma}\cite{Pu3} (4.2)(3.5)(3.7)
Let $(\Gamma,S)$ be a finitely generated discrete group. For $\lambda>1$
let $\ell^1_\lambda(\Gamma)$ be the Banach algebra obtained by completion of the complex group ring of $\Gamma$ with respect to the norm
$$
\parallel \underset{g}{\sum}\,a_gu_g\parallel_\lambda\,=\,
\underset{g}{\sum}\,\vert a_g\vert\cdot\lambda^{\ell_S(g)}.
\eqno(6.9)
$$
Then the canonical morphism
$$
"\underset{\lambda\to 1^+}{\lim}"\,\ell^1_\lambda(\Gamma)\,\to\,"\underset{K\subset U}{\lim}"\,\ell^1(\Gamma)_K
$$
of ind-Banach algebras is a stable diffeotopy equivalence 
and the induced morphism of cyclic bicomplexes
$$
"\underset{\lambda\to 1^+}{\lim}"\,{\mathcal C}{\mathcal C}_*(\ell^1_\lambda(\Gamma))\,\to\,
{\mathcal C}{\mathcal C}^\omega_*(\ell^1(\Gamma))
\eqno(6.10)
$$
is an isomorphism in the derived ind-category.
\end{lemma}

The family $\ell^1_\lambda(\Gamma),\,\lambda\geq 1,$ is the inductive
system of the Banach algebras of summable functions on $\Gamma$ of exponential decay in the word metric and the direct limit of this 
system in the category of Banach algebras equals the group Banach algebra $\ell^1(\Gamma)$. 

Recall the homogeneous decomposition
$$
CC_*({\mathbb C}\Gamma)\,\overset{\simeq}{\to}\,
\underset{\langle g\rangle}{\bigoplus}\,CC_*({\mathbb C}\Gamma)_{\langle g\rangle}
$$
of the cyclic bicomplex of the group ring ${\mathbb C}\Gamma$
into a direct sum of subcomplexes, labeled by the set of conjugacy classes in $\Gamma$. The defining seminorms
$\parallel-\parallel_\lambda$ of $\ell^1_\lambda(\Gamma)$ and 
$\parallel-\parallel_{\varrho,m}$ of ${\mathcal C}{\mathcal C}_*$
are weighted $\ell^1$-norms so that the algebraic homogeneous decomposition gives rise to a decomposition of the bicomplexes ${\mathcal C}{\mathcal C}_*(\ell^1_\lambda(\Gamma)),\,\lambda>1,$ into a topological direct sum of Fr\'echet-complexes, 
labeled by the conjugacy classes of $\Gamma$. It should be noted that there is no reason for the existence of homogeneous decompositions of 
the complexes ${\mathcal C}{\mathcal C}_*(\overline{{\mathbb C}\Gamma})$ for completions $\overline{{\mathbb C}\Gamma}$ of ${\mathbb C}\Gamma$ with respect to other norms than weighted $\ell^1$-norms. In the sequel we will only use decompositions of the analytic cyclic bicomplex of a group Banach algebra into finite direct sums corresponding to finite partitions of the set of all conjugacy classes in the given group. 

\begin{lemma}
Let $(\Gamma,S)$ be a finitely generated group which contains only a finite number of conjugacy classes of torsion elements. Then the homogeneous decomposition
$$
CC_*({\mathbb C}\Gamma)\,\overset{\simeq}{\to}\,
\underset{\langle g'\rangle,\,\vert g'\vert<\infty}{\bigoplus}\,CC_*({\mathbb C}\Gamma)_{\langle g'\rangle}\,\oplus\,\underset{\langle g''\rangle,\,\vert g''\vert=\infty}{\bigoplus}\,CC_*({\mathbb C}\Gamma)_{\langle g''\rangle}
$$
gives rise after completion to a finite direct sum decomposition of complexes
$$
CC_*^{(\varrho)}(\ell^1_\lambda(\Gamma))\,\overset{\simeq}{\longrightarrow}\,\underset{\langle g'\rangle,\,\vert g'\vert<\infty}{\bigoplus}\,CC_*^{(\varrho)}(\ell^1_\lambda(\Gamma))_{\langle g'\rangle}
\,\oplus\,CC_*^{(\varrho)}(\ell^1_\lambda(\Gamma))_{inhom}
$$
Consequently, there is a decomposition 
$$
"\underset{\lambda\to 1^+}{\lim}"\,{\mathcal C}{\mathcal C}(\ell^1_\lambda(\Gamma))\,\overset{\simeq}{\longrightarrow}\,\underset{\langle g'\rangle,\,\vert g'\vert<\infty}{\bigoplus}\,"\underset{\lambda\to 1^+}{\lim}"\,{\mathcal C}{\mathcal C}(\ell^1_\lambda(\Gamma))_{\langle g'\rangle}\,\oplus\,"\underset{\lambda\to 1^+}{\lim}"\,{\mathcal C}{\mathcal C}(\ell^1_\lambda(\Gamma))_{inhom}
$$
of ind-complexes and, in view of (5.6) a corresponding decomposition of (bivariant) local cyclic cohomology groups. The contribution of the conjugacy classes of torsion elements is called the homogeneous part, the contribution of the remaining summand is called the inhomogeneous part, respectively.
\end{lemma}

Recall the canonical epimorphism $p:C_*(\Gamma,\,Ad({\mathbb C}\Gamma)) \to \widetilde{C}_*({\mathbb C}\Gamma)$ (see (2.32)) of chain complexes and its linear section $\iota\,=\,\underset{\langle v\rangle}{\bigoplus}\,\iota_{v,\sigma}:C_*({\mathbb C}\Gamma)\to \widetilde{C}_*(\Gamma,\,Ad({\mathbb C}\Gamma))$ (see (2.63)). It depends on the choice of an element of minimal word length $v\in\langle v\rangle$ and on the choice of a section of the map $\Gamma\to\langle v\rangle,\,u\mapsto uvu^{-1}$ for each conjugacy class $\langle v\rangle$ in $\Gamma$. We compare the various norms introduced 
on these complexes.

\begin{lemma}
Fix $\lambda,\varrho>1$ and $n\in{\mathbb N}$ and let $\alpha_n\in C_n(\Gamma,\,Ad({\mathbb C}\Gamma)),\,\beta_n\in \widetilde{C}_*({\mathbb C}\Gamma)$. Denote by $\parallel-\parallel_{(\lambda,\varrho,m)}$ the norm (6.1) on the cyclic complex of the Banach algebra $CC_*(\ell^1_\lambda(\Gamma))$. Then
$$
\parallel p(\alpha_n)\parallel_{(\lambda,\varrho,m)}\,\leq\,\frac{1}{c(n)!}\cdot(2+2c(n))^m\cdot\varrho^{-c(n)}\cdot\parallel\alpha\parallel_\lambda \eqno(6.11)
$$
and
$$ 
\parallel\iota(\beta_n)\parallel_\lambda\,\leq\,c(n)!\cdot
(2+2c(n))^{-m}\cdot\varrho^{c(n)}\cdot\parallel\beta_n\parallel_{(\lambda,\varrho,m)}
\eqno(6.12)
$$
in the notations of 5.1.
\end{lemma}

\begin{proof}
In the sequel we write $C(\varrho,m,n)$ for $\frac{1}{c(n)!}\cdot(2+2c(n))^m\cdot\varrho^{-c(n)}$. Because the norms in question are weighted $\ell^1$-norms it suffices to verify the claim for Bar-simplices $\alpha_n=[g_0,\ldots,g_n;v]$ and elementary differential forms\\
$\beta_n=g_0dg_1\ldots dg_n$. One finds
$$
\parallel p(\alpha_n)\parallel_{(\varrho,m,0)}\,=\,
\parallel (g_n^{-1}vg_0)d(g_0^{-1}g_1)\ldots d(g_{n-1}^{-1}g_n)\parallel_{(\varrho,m,0)}
$$
$$
=\,C(\varrho,m,n)\cdot\parallel
g_n^{-1}vg_0\parallel_{\ell^1_\lambda(\Gamma)}\cdot\underset{i=1}{\overset{n}{\prod}}\,
\parallel
g_{i-1}^{-1}g_i\parallel_{\ell^1_\lambda(\Gamma)}
$$
$$
=\,C(\varrho,m,n)\cdot
\lambda^{\ell_S(g_n^{-1}vg_0)
+\ell_S(g_0^{-1}g_1)+\ldots+\ell_S(g_{n-1}^{-1}g_n)}\,=\,
$$
$$
=\,C(\varrho,m,n)\cdot\lambda^{d_S(g_n,vg_0)+d_S(g_0,g_1)+\ldots+d_S(g_{n-1},g_n)}
$$
$$
=\,C(\varrho,m,n)\cdot\lambda^{\vert[g_0,\ldots,g_n;v]\vert}\,=\,
C(\varrho,m,n)\cdot\parallel\alpha_n\parallel_\lambda
$$
To study the other estimate let $v$ be the distinguished element of the conjugacy class of
$g_0g_1\ldots g_n$ and let $\sigma:\langle v\rangle\to\Gamma$ be a fixed section of $\Gamma\to\langle v\rangle,\,g\mapsto gvg^{-1}$ (see (2.61)).
We calculate
$$
\parallel\iota(\beta_n)\parallel_\lambda\,=\,\parallel\iota_{v,\sigma}(g_0dg_1\ldots dg_n)\parallel_\lambda\,=\,
\parallel [hg_0,hg_0g_1,\ldots,hg_0g_1\ldots g_n;v]\parallel_\lambda,
$$
where $h=\sigma(g_0g_1\ldots g_n)^{-1}$. One finds $\parallel\iota(\beta_n)\parallel_\lambda\,=\,
\lambda^{\vert [hg_0,hg_0g_1,\ldots,hg_0g_1\ldots g_n;v]\vert},$ where 
$$
\vert [hg_0,hg_0g_1,\ldots,hg_0g_1\ldots g_n;v]\vert\,
=\,
d_S(hg_0,hg_0g_1)+ d_S(hg_0g_1,hg_0g_1g_2)+\ldots
$$
$$
\ldots+
d_S(hg_0g_1\ldots g_{n-1},hg_0g_1g_2\ldots g_{n-1}g_n)+
d_S(hg_0g_1g_2\ldots g_n, vhg_0)
$$
$$
=\ell_S(g_1)+\ell_S(g_2)+\ldots+\ell_S(g_n)+\ell_S((hg_0\ldots g_n)^{-1}vhg_0)
$$
and
$$
(hg_0\ldots g_n)^{-1}vhg_0\,=\,(g_0\ldots g_n)^{-1}(h^{-1}vh)g_0
$$
$$
=\,(g_0\ldots g_n)^{-1}\left(\sigma(g_0\ldots g_n)v\sigma(g_0\ldots g_n)^{-1}\right)g_0
\,=\,(g_0\ldots g_n)^{-1}(g_0\ldots g_n)g_0\,=\,g_0,
$$
so that 
$$
\parallel\iota(\beta_n)\parallel_\lambda\,=\,\lambda^{\ell_S(g_0)+\ell_S(g_1)+\ldots+\ell_S(g_n)}\,
=\,\underset{i=0}{\overset{n}{\prod}}\,\parallel g_i\parallel_{\ell^1_\lambda(\Gamma)}\,=\,
C(\varrho,m,n)^{-1}\cdot\parallel\beta_n\parallel_{(\varrho,m,0)}.
$$
\end{proof}

In this chapter we want to finish our calculation of the bivariant local cyclic cohomology groups of various Banach algebras. In view of the definition of these groups as groups of morphisms in the derived ind-category, I would like to say a few words what "calculation" means in this context. A general object of ${\mathcal D}(ind\,{\mathcal C})$ is a formal inductive limit of ${\mathbb Z}/2{\mathbb Z}$-graded complexes of Fr\'echet spaces, for example the analytic cyclic bicomplex of a Banach algebra. The most basic objects of the derived ind-category are constant inductive systems of finite dimensional complexes with differential zero. "Calculating" the local cyclic cohomology of a Banach algebra will mean for us to exhibit an isomorphism in ${\mathcal D}(ind\,{\mathcal C})$ between its analytic cyclic bicomplex and such a constant, finite dimensional ind-complex. The full subcategory generated by these is in fact equivalent to the category of finite dimensional, ${\mathbb Z}/2{\mathbb Z}$-graded vector spaces, so that bivariant local cyclic cohomology groups in question are easily determined once such an identification has been achieved. \\
\\
The calculation of the local cyclic cohomology of
$\ell^1(\Gamma)$ for a word hyperbolic group will now be done in three
steps. In the first it will be shown that the ind-complexes
$"\underset{\lambda\to 1^+}{\lim}"\,Fil^n_{Hodge}{\mathcal C}{\mathcal C}(\ell^1_\lambda(\Gamma))$ are contractible for large $n$. In the second 
we calculate the contribution of a given elliptic conjugacy class to 
the analytic cyclic bicomplex. In the last we show that the remaining inhomogeneous part is contractible. In view of 6.2 this identifies 
the analytic cyclic bicomplex of $\ell^1(\Gamma)$ in the derived ind-category with an explicitely given constant, finite dimensional ind-complex.\\
\\
For a Banach algebra $B$ and an integer $n$ we let
$$
Fil^n_{Hodge}{\mathcal C}{\mathcal C}_*(B)\,=\,
"\underset{\varrho\to \infty}{\lim}"\,Fil^n_{Hodge}\,CC_*^{(\varrho)}(B)
\eqno(6.13)
$$
and define $Fil^n_{Hodge}\,CC_*^{(\varrho)}(B)$ as the closure of 
$Fil^n_{Hodge}\,CC_*(B)$ inside $CC_*^{(\varrho)}(B)$.

\begin{prop}

Let $(\Gamma,S)$ be a word hyperbolic group and let $N\in{\mathbb N}$ be sufficiently large. Then the following holds. 
\begin{itemize}
\item[a)]
The ind-Fr\'echet complex
$$
"\underset{\lambda\to 1^+}{\lim}"\,Fil^N_{Hodge}\,{\mathcal C}{\mathcal C}_*(\ell^1_\lambda(\Gamma))
\eqno(6.14)
$$
is contractible.
\item[b)]
The canonical
map of $\ind$-complexes 
$$
"\underset{\lambda\to 1^+}{\lim}"\,
{\mathcal C}{\mathcal C}_*(\ell^1_{\lambda} (\Gamma))\,\longrightarrow\, 
"\underset{\lambda\to 1^+}{\lim}"\,{\mathcal C}{\mathcal C}_*(\ell^1_{\lambda} (\Gamma))/Fil^N_{Hodge}\,{\mathcal C}{\mathcal C}_*(\ell^1_\lambda(\Gamma))
\eqno(6.15)
$$
is an isomorphism in the homotopy category of formal inductive limits of ${\mathbb Z}/2{\mathbb Z}$-graded complexes 
of Fr\'echet spaces. 

\item[c)] The isomorphism of b) is compatible with the harmonic
decomposition 5.3.

\end{itemize}
\end{prop}

\begin{proof}

We proceed in several steps.\\
\\
{\bf Step 1:}\\
Let $A$ be a complex algebra and let $\nabla:C_*(A)\to C_{*+1}(A)$ 
be an endomorphism of the Hochschild complex of degree one. 
If the operator $\nabla\circ b+b\circ\nabla$ equals the identity in degrees $*\geq N$, then $\nabla$ defines a contracting chain homotopy 
of $Fil^N_{Hodge}\,C_*(A)$ and the operator
$$
\nabla_{cyc}\,=\,\underset{k=0}{\overset{\infty}{\sum}}\,
(-\nabla\circ B)^k\circ\nabla:\,Fil^N_{Hodge}\,\widehat{CC}_*(A)\,
\to\,Fil^N_{Hodge}\,\widehat{CC}_{*+1}(A)
\eqno(6.16)
$$
will define a contracting chain homotopy of the $N$-th step of the Hodge filtration of the periodic cyclic bicomplex $\widehat{CC}_*(A)$, 
as a straightforward calculation shows.\\
\\
{\bf Step 2:}\\
Consider the linear operator 
$$
\widetilde{\nabla}:\,C_*(\Gamma,Ad({\mathbb C}\Gamma))
\to C_{*+1}(\Gamma, Ad({\mathbb C}\Gamma))
$$ 
defined in (5.3). It is 
$\Gamma$-equivariant and descends therefore to an operator 
$$
\nabla:C_*({\mathbb C}\Gamma)\,\to\,C_{*+1}({\mathbb C}\Gamma)
\eqno(6.17)
$$
on the Hochschild complex, which satisfies
$$
\nabla\circ b+b\circ\nabla\,=\,Id 
$$
in sufficiently high degrees. 
We want to show that $\nabla$ extends to a bounded operator 
on the ind-complex $"\underset{\lambda\to 1^+}{\lim}"\,{\mathcal C}{\mathcal C}_*(\ell^1_\lambda(\Gamma))$. 
So fix $\lambda_0,\varrho_0>1$ and let $\lambda_1,\varrho_1$ satisfy \\
$\lambda_0>\lambda_1>1$ and $\varrho_1\geq\varrho_0$. Recall that 
$$
p\circ\iota=Id:C_*({\mathbb C}\Gamma)\to C_*({\mathbb C}\Gamma)
$$
and thus 
$$
\nabla\,=\,\nabla\circ p\circ\iota\,=\,p\circ\widetilde{\nabla}\circ\iota
$$
in the notations of 2.4 and 6.3. We find then for a chain $c_n\in CC_n({\mathbb C}\Gamma)$ the estimate
$$
\parallel\nabla(c_n)\parallel_{(\lambda_1,\varrho_1,0)}\,=\,
\parallel p(\widetilde{\nabla}(\iota(c_n)))\parallel_{(\lambda_1,\varrho_1,0)}
$$
$$
\leq\,
\frac{1}{c(n+1)!}\cdot\varrho_1^{-c(n+1)}
\cdot\parallel\widetilde{\nabla}(\iota(c_n))\parallel_{\lambda_1} \,\,\,\text{by (6.10)}
$$
$$
\leq\,
\frac{1}{c(n+1)!}\cdot\varrho_1^{-c(n+1)}
\cdot C_{22}\cdot\parallel\iota(c_n)\parallel_{\lambda_0} \,\,\,\text{by (5.6)}
$$
$$
\leq\,
c(n)!\cdot\varrho_0^{c(n)}\cdot
\frac{1}{c(n+1)!}\cdot\varrho_1^{-c(n+1)}
\cdot C_{22}\cdot\parallel c_n\parallel_{(\lambda_0,\varrho_0,0)} \,\,\,\text{by (6.11)}
$$
$$
\leq\,C_{22}\cdot\parallel c_n\parallel_{(\lambda_0,\varrho_0,0)}.
$$
The previous inequalities imply similar estimates for the norms\\ 
$\parallel\nabla(c_n)\parallel_{(\lambda_0,\varrho_0,m)},\,m>0$. 
This shows that $\nabla$ extends to a bounded operator
$$
\nabla:\,CC_*^{(\varrho_0)}(\ell^1_{\lambda_0}(\Gamma))\,\to\,
CC_{*+1}^{(\varrho_1)}(\ell^1_{\lambda_1}(\Gamma)).
\eqno(6.18)
$$
It therefore gives rise to a linear morphism of ind-Fr\'echet spaces
$$
\nabla:\,"\underset{\lambda\to 1^+}{\lim}"\,{\mathcal C}{\mathcal C}_*(\ell^1_\lambda(\Gamma))\,\to\,
"\underset{\lambda\to 1^+}{\lim}"\,{\mathcal C}{\mathcal C}_{*+1}(\ell^1_\lambda(\Gamma))
\eqno(6.19)
$$
which still satisfies the identity $\nabla\circ b+b\circ\nabla\,=\,Id$  in sufficiently high degrees.\\
\\
{\bf Step 3:}\\
We want to show now that the operator $\nabla_{cyc}$ of (6.16) extends to a 
linear morphism of the ind-complex $"\underset{\lambda\to 1^+}{\lim}"\,{\mathcal C}{\mathcal C}(\ell^1_\lambda(\Gamma))$
as well. To this end let $k\in{\mathbb N}$ and fix again $\lambda_0,\varrho_0>1$. Choose now $\lambda_1,\varrho_1$ such that
$\lambda_0>\lambda_1>1$ and $\varrho_1\,>\,Max(2\cdot C_{25},\,\varrho_0)$. One finds for a Hochschild 
chain $c_n\in CC_n({\mathbb C}\Gamma)$ 
$$
\parallel(\nabla\circ B)^k(c_n)\parallel_{(\lambda_1,\varrho_1,0)}\,=\,
\parallel p\circ(\widetilde{\nabla}\circ\widetilde{B})^k\circ\iota(c_n)\parallel_{(\lambda_1,\varrho_1,0)}
$$
$$
\frac{1}{c(n+2k)!}\cdot\varrho_1^{-c(n+2k)}\cdot\parallel 
(\widetilde{\nabla}\circ\widetilde{B})^k\circ\iota(c_n)\parallel_{\lambda_1} \,\,\,\text{by (6.11)}
$$
$$
\leq\,\frac{1}{(c(n)+k)!}\cdot\varrho_1^{-(c(n)+k)}\cdot k!\cdot C_{25}^k\cdot\parallel\iota(c_n)\parallel_{\lambda_0} \,\,\,\text{by (5.7)}
$$
$$
\leq\,\frac{1}{(c(n)+k)!}\cdot\varrho_1^{-(c(n)+k)}\cdot k!\cdot C_{25}^k\cdot c(n)!\cdot \varrho_0^{c(n)}\cdot
\parallel c_n\parallel_{(\lambda_0,\varrho_0,0)} \,\,\,\text{by (6.12)}
$$
$$
\leq\,\left(\frac{c(n)!\cdot k!}{(c(n)+k)!}\right)\cdot(\varrho_1^{-1}\cdot\varrho_0)^{c(n)}\cdot(\varrho_1^{-1}\cdot C_{25})^k\cdot
\parallel c_n\parallel_{(\lambda_0,\varrho_0,0)}
$$
$$
\leq\,2^{-k}\cdot\parallel c_n\parallel_{(\lambda_0,\varrho_0,0)}.
$$
Altogether
$$
\parallel\underset{k=0}{\overset{\infty}{\sum}}\,(-\nabla\circ B)^k(c_n)\parallel_{(\lambda_1,\varrho_1,0)}\,\leq\,
\underset{k=0}{\overset{\infty}{\sum}}\,
\parallel(-\nabla\circ B)^k(c_n)\parallel_{(\lambda_1,\varrho_1,0)}
$$
$$
\leq\,\underset{k=0}{\overset{\infty}{\sum}}\,2^{-k}\cdot
\parallel c_n\parallel_{(\lambda_0,\varrho_0,0)}\,=\,
2\parallel c_n\parallel_{(\lambda_0,\varrho_0,0)}
$$
so that the linear map $\underset{k=0}{\overset{\infty}{\sum}}\,(-\nabla\circ B)^k:\,CC_*({\mathbb C}\Gamma)\,\to\,\widehat{C}_{*+1}({\mathbb C} \Gamma)$ 
extends to a bounded operator 
$$
\underset{k=0}{\overset{\infty}{\sum}}\,(-\nabla\circ B)^k:
\,CC_*^{(\varrho_0)}(\ell^1_{\lambda_0}(\Gamma))\,\to\,
CC_{*+1}^{(\varrho_1)}(\ell^1_{\lambda_1}(\Gamma))
\eqno(6.20)
$$
and gives rise to a linear morphism
$$
\underset{k=0}{\overset{\infty}{\sum}}\,(-\nabla\circ B)^k:
"\underset{\lambda\to 1^+}{\lim}"\,{\mathcal C}{\mathcal C}_*(\ell^1_\lambda(\Gamma))\,\to\,
"\underset{\lambda\to 1^+}{\lim}"\,{\mathcal C}{\mathcal C}_{*+1}(\ell^1_\lambda(\Gamma))
\eqno(6.21)
$$
of ind-Fr\'echet spaces. \\
\\
{\bf Step 4:}\\
Step 1, 2 and 3 together prove that the operator
$$
\nabla_{cyc}\,=\,\underset{k=0}{\overset{\infty}{\sum}}\,(-\nabla\circ B)^k\circ\nabla:\,\widehat{CC}_*({\mathbb C}\Gamma)\to 
\widehat{CC}_{*+1}({\mathbb C}\Gamma)
$$ 
extends to a linear morphism
$$
\nabla_{cyc}:\,"\underset{\lambda\to 1^+}{\lim}"\,{\mathcal C}{\mathcal C}_*(\ell^1_\lambda(\Gamma))\,\to\,
"\underset{\lambda\to 1^+}{\lim}"\,{\mathcal C}{\mathcal C}_{*+1}(\ell^1_\lambda(\Gamma))
\eqno(6.22)
$$
of ind-Fr\'echet spaces. Its restriction to a sufficiently high layer 
$"\underset{\lambda\to 1^+}{\lim}"\,Fil^N_{Hodge}\,{\mathcal C}{\mathcal C}_*(\ell^1_\lambda(\Gamma))$ of the Hodge filtration satisfies 
$\nabla_{cyc}\circ(b+B)\,+\,(b+B)\circ\nabla_{cyc}\,=\,Id$, which shows 
that 
$$
"\underset{\lambda\to 1^+}{\lim}"\,Fil^N_{Hodge}\,{\mathcal C}{\mathcal C}_*(\ell^1_\lambda(\Gamma))
$$ 
is indeed a contractible ind-complex for $N>>0$.\\
\\
{\bf Step 5:}\\
The homotopy category of chain complexes is the key example of a triangulated category. This structure is inherited by the homotopy 
category $Ho(ind\,{\mathcal C})$ of ${\mathbb Z}/2{\mathbb Z}$ graded ind-Fr\'echet complexes. To derive assertion b) from a) it suffices therefore to show that the triangle
$$
"\underset{\lambda\to 1^+}{\lim}"\,Fil^N_{Hodge}\,{\mathcal C}{\mathcal C}_*(\ell^1_\lambda(\Gamma)) \overset{i_*}{\to} 
"\underset{\lambda\to 1^+}{\lim}"\,{\mathcal C}{\mathcal C}_*(\ell^1_\lambda(\Gamma))\to "\underset{\lambda\to 1^+}{\lim}"\,{\mathcal C}{\mathcal C}_*(\ell^1_\lambda(\Gamma))/\,Fil^N_{Hodge}\,{\mathcal C}{\mathcal C}_*(\ell^1_\lambda(\Gamma))
$$
is distinguished. This is equivalent to the assertion that the canonical map
$$
Cone(i_*)\,\to\,"\underset{\lambda\to 1^+}{\lim}"\,{\mathcal C}{\mathcal C}_*(\ell^1_\lambda(\Gamma))/\,Fil^N_{Hodge}\,{\mathcal C}{\mathcal C}_*(\ell^1_\lambda(\Gamma))
$$
is a chain homotopy equivalence. To construct a homotopy inverse it suffices to find a linear morphism of ind-Fr\'echet spaces splitting 
the chain map $i_*$. Such a splitting is provided on the dense subcomplex
$CC_*({\mathbb C}\Gamma)/Fil^N_{Hodge}\,CC_*({\mathbb C}\Gamma)$ by the linear map
$$
s_*:\,\Omega^{<N}({\mathbb C}\Gamma)/\,b\Omega^N({\mathbb C}\Gamma)\,\to\,\Omega^*({\mathbb C}\Gamma),\,
s_*\,=\,
\begin{cases}
Id & *<N-1, \\
(Id-b\circ\nabla) & *=N-1. \\
\end{cases}
\eqno(6.23)
$$
It extends by step 2 to the desired linear morphism 
$$
s_*: "\underset{\lambda\to 1^+}{\lim}"\,{\mathcal C}{\mathcal C}_*(\ell^1_\lambda(\Gamma))/\,Fil^N_{Hodge}\,{\mathcal C}{\mathcal C}_*(\ell^1_\lambda(\Gamma))\,\to\,
"\underset{\lambda\to 1^+}{\lim}"\,{\mathcal C}{\mathcal C}_{*}(\ell^1_\lambda(\Gamma))
\eqno(6.24)
$$
of ind-Fr\'echet spaces.\\
\\
{\bf Step 6:}\\
To verify c) we observe that the operator $\widetilde{\nabla}$ is natural with respect to the chosen coefficient module and therefore compatible with the decomposition of $Ad(\Gamma)$ into irreducible subspaces. This shows that $\widetilde{\nabla}$, and therefore also the operators $\nabla$ and $\nabla_{cyc}$ are compatible with the homogeneous decomposition.
\end{proof}

{\vskip 3mm}

The constant, finite dimensional ind-complexes we will have to deal with are attached to finite dimensional, integer graded vector spaces. 
To such a vector space we associate the ${\mathbb Z}/2{\mathbb Z}$ graded complex given by the direct sum of its even, respectively odd components,
equipped with the differential zero.\\
With this being understood we may formulate

\begin{prop}
Let $(\Gamma,S)$ be a word-hyperbolic group and let $v\in\Gamma$ be a torsion element. Let $Z(v)$ be the centralizer of $\Gamma$. Then
there exists a canonical isomorphism
$$
H_*(Z(v),{\mathbb C}) \,\,\overset{\simeq}{\longrightarrow}\,\,
"\underset{\lambda\to 1^+}{\lim}"\,{\mathcal C}{\mathcal C}_*(\ell^1_\lambda(\Gamma))_{\langle v\rangle}
\eqno(6.25)
$$
in the homotopy category of ind-complexes $Ho(ind\,{\mathcal C})$.
\end{prop}

\begin{proof}
Proposition 2.7 indicates how to calculate the contribution of the elliptic conjugacy class $\langle v\rangle$ to the cyclic homology of the group ring ${\mathbb C}\Gamma$. Because the involved operators turn out to be bounded on algebraic differential forms 
of a given degree, but not on the space of all differential forms we will have to invoke our results about the degeneration of the Hodge filtration.\\
We proceed in several steps:\\
\\
{\bf Step 1:}\\
The previous proposition 6.5 shows that the canonical morphisms
$$
"\underset{\lambda\to 1^+}{\lim}"\,\overline{{\mathcal C}}_*(\ell^1_\lambda(\Gamma))_{\langle v\rangle}
\,\to\,
"\underset{\lambda\to 1^+}{\lim}"\,\overline{{\mathcal C}}_*(\ell^1_\lambda(\Gamma))/\,Fil^N_{Hodge}\,\overline{{\mathcal C}}_*(\ell^1_\lambda(\Gamma))_{\langle v\rangle}
\eqno(6.26)
$$
and
$$
"\underset{\lambda\to 1^+}{\lim}"\,\overline{{\mathcal C}{\mathcal C}}_*(\ell^1_\lambda(\Gamma))_{\langle v\rangle}
\,\to\,
"\underset{\lambda\to 1^+}{\lim}"\,\overline{{\mathcal C}{\mathcal C}}_*(\ell^1_\lambda(\Gamma))/\,Fil^N_{Hodge}\,\overline{{\mathcal C}{\mathcal C}}_*(\ell^1_\lambda(\Gamma))_{\langle v\rangle}
\eqno(6.27)
$$
of reduced complexes are isomorphisms in $Ho(ind\,{\mathcal C})$ for $N>>0$. \\
\\
{\bf Step 2:}\\
We show next that the isomorphism
$$
\nu_{v}:\,\widehat{\overline{CC}}_*({\mathbb C}\Gamma)_{\langle v\rangle}\,\to\,
\widehat{\overline{C}}_*({\mathbb C}\Gamma)_{\langle v\rangle}
$$
of chain complexes, constructed in 2.7, gives rise to 
an isomorphism
$$
\nu_{v}:"\underset{\lambda\to 1^+}{\lim}"\,\overline{{\mathcal C}{\mathcal C}}_*(\ell^1_\lambda(\Gamma))/\,Fil^N_{Hodge}\,\overline{{\mathcal C}{\mathcal C}}_*(\ell^1_\lambda(\Gamma))_{\langle v\rangle}\,\to\,
"\underset{\lambda\to 1^+}{\lim}"\,\overline{{\mathcal C}}_*(\ell^1_\lambda(\Gamma))/\,Fil^N_{Hodge}\,\overline{{\mathcal C}}_*(\ell^1_\lambda(\Gamma))_{\langle v\rangle}
\eqno(6.28)
$$
of ind-Fr\'echet complexes in $Ho(ind\,{\mathcal C})$. To this end we have to verify its boundedness. Recall that 
$\nu_{v}=Id+h(\overline{\chi}_v,Id)\circ B:\,\widehat{\overline{\Omega}}^*
({\mathbb C}\Gamma)_{\langle v\rangle}\,\to\,\widehat{\overline{\Omega}}^*
({\mathbb C}\Gamma)_{\langle v\rangle}$ and that the operator $B$ 
extends to a bounded endomorphism of the ind-complexes in question.
We have thus to study the boundedness of the operator $h(\overline{\mu}_v,Id)$ in detail. \\
\\For a Bar-simplex $\alpha_n=[g_0,\ldots,g_n;v]\in
\Delta_n(\Gamma)\times\{v\}$ one has 
$$
h(\mu_v,Id)(\alpha_n)\,=\,
\underset{i=0}{\overset{n}{\sum}}\,(-1)^i[\mu_v(g_0,\ldots,g_i),g_i,\ldots,g_n;v]
$$
where $\mu_v$ is the composition of the antisymmetrization and the averaging operator with respect to the action of the finite cyclic group 
generated by $v$ on the Bar-complex. For the $\ell^1$-norm one finds therefore 
$$
\parallel h(\mu_v,id)(\alpha_n)\parallel_1\,\leq\,(n+1).
\eqno(6.29)
$$
Concerning the weights we observe
$$
\vert h(\mu_v,id)([g_0,\ldots,g_n;v])\vert
$$
$$
\leq\,\underset{0\leq i\leq n}{Max}\,\,
\underset{\sigma\in\Sigma_i}{Max}\,\,
\underset{k_0,\dots,k_i}{Max}\,\,
\vert[v^{k_0}g_{\sigma(0)},v^{k_1}g_{\sigma(1)},\ldots,
v^{k_i}g_{\sigma(i)},g_i,\ldots,g_n;v]\vert
$$
$$
\leq\,(n+2)\,\underset{0\leq s,t\leq n}{Max}\,\,\underset{l}{Max}\,\,\ell_S(g_s^{-1}v^lg_t)
$$
Now
$$
\ell_S(g_s^{-1}v^lg_t)\,=\,\ell_S((g_s^{-1}vg_s)^l(g_s^{-1}g_t))
\,\leq\,(\vert v\vert-1)\cdot\ell_S(g_s^{-1}vg_s)+\ell_S(g_s^{-1}g_t)\,\leq\,\vert v\vert\cdot\vert[g_0,\ldots,g_n;v]\vert,
$$
so that  
$$
\vert h(\mu_v,Id)([g_0,\ldots,g_n;v])\vert\,\leq\,
(n+2)\cdot\vert v\vert\cdot\vert[g_0,\ldots,g_n;v]\vert.
\eqno(6.30)
$$
Let now $\lambda_0,\varrho_0>1$ and choose $\lambda_1,\varrho_1>1$ such that 
$$
\lambda_1^{(N+1)\cdot\vert v\vert}\leq\lambda_0\eqno(6.31)
$$ 
and $\varrho_1\geq\varrho_0$. For elementary differential forms $\omega_n\in C_n({\mathbb C}\Gamma)_{\langle v\rangle},\,n<N$ we find, noting that 
$$
h(\overline{\mu}_v,id)=h(\overline{\mu}_v,id)\circ \overline{p}_v\circ\iota_{v,\sigma}=\overline{p}_v\circ h(\mu_v,Id)\circ\iota_{v,\sigma},
$$
the estimate 
$$
\parallel h(\overline{\mu},id)(\omega_n)\parallel_{(\lambda_1,\varrho_1,0)}\,=\,\parallel \overline{p}_v\circ h(\mu_v,id)\circ\iota_{v,\sigma}(\omega_n)\parallel_{(\lambda_1,\varrho_1,0)}
$$
$$
\leq\,\frac{1}{c(n+1)!}\cdot\varrho_1^{-c(n+1)}\cdot\parallel
h(\mu_v,id)\circ\iota_{v,\sigma}(\omega_n)\parallel_{\lambda_1}
$$
$$
\leq\,(n+1)\cdot\frac{1}{c(n+1)!}\cdot\varrho_1^{-c(n+1)}\cdot
\lambda_1^{\vert h(\mu_v,id)(\iota_{v,\sigma}(\omega_n))\vert}\,
\leq\,(n+1)\cdot\frac{1}{c(n+1)!}\cdot\varrho_1^{-c(n+1)}\cdot
\lambda_1^{(N+1)\cdot\vert v\vert\cdot\vert\iota_{v,\sigma}(\omega_n)\vert}
$$
$$
\leq\,(n+1)\cdot\frac{1}{c(n+1)!}\cdot\varrho_1^{-c(n+1)}\cdot
\lambda_0^{\vert\iota_{v,\sigma}(\omega_n)\vert}\,=\,
(n+1)\cdot\frac{1}{c(n+1)!}\cdot\varrho_1^{-c(n+1)}\cdot
\parallel\iota_{v,\sigma}(\omega_n)\parallel_{\lambda_0}
$$
$$
\leq\,(n+1)\cdot\frac{1}{c(n+1)!}\cdot\varrho_1^{-c(n+1)}\cdot
c(n)!\cdot\varrho_0^{c(n)}\cdot
\parallel\omega_n\parallel_{\lambda_0,\varrho_0,0}
$$
$$
\leq\,N\cdot
\parallel\omega_n\parallel_{(\lambda_0,\varrho_0,0)}.
\eqno(6.32)
$$
The norm in question being a weighted $\ell^1$-norm, this estimate extends to arbitrary elements of $C_n({\mathbb C}\Gamma)_{\langle v\rangle}$. It follows that $\nu_v$ extends to a morphism of ind-complexes as claimed. Its inverse is given by a finite geometric series and is thus bounded as well. Note however that we cannot claim 
that the operator $\nu_v$ extends to a morphism on the whole cyclic 
ind-complex in question because for a given index $\lambda_0>1$ the 
choice of the appropriate $\lambda_1$ depends on the truncation level $N$ and cannot be made uniformly.

\newpage

{\bf Step 3:}\\
Fix a section $\sigma:\,\langle v\rangle\to\Gamma$ of the map
$\Gamma\to\langle v\rangle,\,g\mapsto gvg^{-1}$ satisfying (3.15), (3.16)
and such that $\sigma(v)=e$. Consider the inclusion
$$
\begin{array}{cccc}
j_v: & C_*(Z(v),{\mathbb C}) & \to & C_*(\Gamma,\,{\mathbb C}v) \\
 & & & \\
 & [h_0,\ldots,h_n] & \mapsto & [h_0,\ldots,h_n;v] \\
 \end{array}
 \eqno(6.33)
$$
and the linear map
$$
\begin{array}{cccc}
\kappa_v: & C_*(\Gamma,\,{\mathbb C}v) & \to & C_*(Z(v),\,{\mathbb C}) \\
 & & & \\
 & [g_0,\ldots,g_n;v] & \mapsto & 
 [g_0\cdot\sigma(g_0^{-1}vg_0),\ldots,g_n\cdot\sigma(g_n^{-1}vg_n)].
\end{array}
\eqno(6.34)
$$
These are both $Z(v)$-equivariant chain maps compatible with augmentations and preserving the degenerate sub complexes. One has $\kappa_v\circ j_v=Id_{C_*(Z(v),{\mathbb C})}$
and $j_v\circ\kappa_v$ is chain homotopic to the identity via the 
$Z(v)$-equivariant homotopy operator
$h_*(j_v\circ\kappa_v,id):\,C_*(\Gamma,\,{\mathbb C}v)\to
C_{*+1}(\Gamma,\,{\mathbb C}v).$ 
After passing to $Z(v)$-coinvariants we obtain a corresponding diagram of chain maps
$$
\begin{array}{ccccc}
\overline{C}_*(Z(v),{\mathbb C})_{Z(v)} & \overset{j_v}{\longrightarrow} & \overline{C}_*(\Gamma,\,{\mathbb C}v)_{Z(v)} & \overset{\simeq}{\longrightarrow} & \overline{C}_*(\Gamma,\,{\mathbb C}\langle v\rangle)_{\Gamma} \\
 & & & & \\
 \parallel & & \parallel & & \parallel \\
  & & & & \\
 \overline{C}_*({\mathbb C}Z(v))_{\langle e\rangle} &  \overset{\overline{j}_v}{\longrightarrow} & \overline{C}_*({\mathbb C}\Gamma)_{\langle v\rangle} & \overset{\simeq}{\longrightarrow} & 
 \overline{C}_*({\mathbb C}\Gamma)_{\langle v\rangle}
\end{array}
\eqno(6.35)
$$
where $\overline{j}_v$ is a chain homotopy equivalence with chain homotopy inverse $\overline{\kappa}_v.$\\
We check the analytic properties of these maps. The operator $j_v$ is isometric and $\kappa_v$ is contractive with respect to the $\ell^1$-norms on $C_*(\Gamma,\,{\mathbb C}v)$ and $C_*(Z(v),{\mathbb C})$. Concerning the weights 
we find on the one hand
$$
\vert j_v([h_0,\ldots,h_n])\vert\,\leq\,\vert[h_0,\ldots,h_n]\vert+\ell_S(\langle v\rangle)
\eqno(6.36)
$$
because $h_n$ commutes with $v$. On the other hand
$$
\vert \kappa_v([g_0,\ldots,g_n;v])\vert\,=\,
\vert[g_0\cdot\sigma(g_0^{-1}vg_0),\ldots,g_n\cdot\sigma(g_n^{-1}vg_n)]\vert
$$
$$
=\,\underset{i=0}{\overset{n}{\sum}}\,
d_S(g_i\cdot\sigma(g_i^{-1}vg_i),\,g_{i+1}\cdot\sigma(g_{i+1}^{-1}vg_{i+1})).
$$
For the sum of the first $n$ terms one finds
$$
\underset{i=0}{\overset{n-1}{\sum}}\,
d_S(g_i\cdot\sigma(g_i^{-1}vg_i),\,g_{i+1}\cdot\sigma(g_{i+1}^{-1}vg_{i+1}))
$$
$$
=\underset{i=0}{\overset{n-1}{\sum}}\,
d_S(\sigma((g_i^{-1}g_{i+1})(g_{i+1}^{-1}vg_{i+1})(g_{i}^{-1}g_{i+1})^{-1}),\,(g_{i}^{-1}g_{i+1})\cdot\sigma(g_{i+1}^{-1}vg_{i+1}))
$$
$$
\leq\,\underset{i=0}{\overset{n-1}{\sum}}\,C_{12}(\Gamma,S,\delta)\cdot \ell_S(g_i^{-1}g_{i+1})\,+\,\ell_S(\langle v\rangle)+\,C_{13}(\Gamma,S,\delta)
$$
by (3.16), whereas
$$
d_S(g_n\cdot\sigma(g_n^{-1}vg_n),\,g_{0}\cdot\sigma(g_{0}^{-1}vg_{0}))
$$
$$
\leq d_S(g_n\cdot\sigma(g_n^{-1}vg_n),\,vg_{0}\cdot\sigma(g_{0}^{-1}vg_{0}))
+d_S(vg_0\cdot\sigma(g_0^{-1}vg_0),\,g_{0}\cdot\sigma(g_{0}^{-1}vg_{0}))
$$
$$
=d_S(g_n\cdot\sigma(g_n^{-1}vg_n),\,vg_{0}\cdot\sigma(g_{0}^{-1}vg_{0}))
+\ell_S(\sigma(g_{0}^{-1}vg_{0})^{-1}(g_0^{-1}vg_0)\sigma(g_0^{-1}vg_0))
$$
$$
=d_S(\sigma((g_n^{-1}vg_0)(g_0^{-1}vg_0)(g_n^{-1}vg_0)^{-1}),\,g_n^{-1}vg_{0}\cdot\sigma(g_{0}^{-1}vg_{0}))+\ell_S(v)
$$
$$
\leq C_{12}\cdot \ell_S(g_n^{-1}vg_{0})\,+\,2\ell_S(\langle v\rangle)+\,C_{13} = C_{12}\cdot d_S(g_n,vg_{0})\,+\,2\ell_S(\langle v\rangle)+\,C_{13}
$$
so that altogether
$$
\vert \kappa_v([g_0,\ldots,g_n;v])\vert\,\leq\,
C_{12}\cdot\vert[g_0,\ldots,g_n;v]\vert\,+\,(n+2)\cdot\ell_S(\langle v\rangle)+\,(n+1)\cdot C_{13}
\eqno(6.37)
$$
\\
{\bf Step 4:}\\
\\
For the $Z(v)$-equivariant homotopy operator connecting $j_v\circ\kappa_v$ and the identity one observes that
$$
\parallel  h( j_v\circ\kappa_v,id)\parallel_{\ell^1}\,\leq\,(n+1)
\eqno(6.38)
$$
on $C_n(\Gamma,{\mathbb C}v)$. Concerning the behavior of the homotopy operator with respect to weights  we find for a Bar-$n$-simplex 
$\alpha_n\,=\,[g_0,\ldots,g_n;v]\in\Delta_n(\Gamma)\times\{v\}$ the estimate
$$
\vert h( j_v\circ\kappa_v,id)(\alpha_n)\vert\,=\,\underset{0\leq i\leq n}{Max}\vert
[g_0\sigma(g_0^{-1}vg_0),\ldots,g_i\sigma(g_i^{-1}vg_i),g_i,\ldots,g_n;v]  \vert
$$
and
$$
\vert
[g_0\sigma(g_0^{-1}vg_0),\ldots,g_i\sigma(g_i^{-1}vg_i),g_i,\ldots,g_n;v]  \vert\,=\,
\underset{j=0}{\overset{i-1}{\sum}}\,d_S(g_j\sigma(g_j^{-1}vg_j),g_{j+1}\sigma(g_{j+1}^{-1}vg_{j+1}))
$$
$$
+d_S(g_i\sigma(g_i^{-1}vg_i),g_i)\,+\,\underset{k=i}{\overset{n-1}{\sum}}\,d_S(g_k,g_{k+1})\,
+\,s_S(g_n,vg_0\sigma(g_0^{-1}vg_0))
$$
$$
\leq\underset{j=0}{\overset{i-1}{\sum}}\,\left( C_{12}\cdot d_S(g_j,g_{j+1})\,+\,\ell_S(\langle v\rangle\,+\,C_{13}\right)
$$
$$
+\underset{k=i}{\overset{n-1}{\sum}}\,d_S(g_k,g_{k+1})\,+\,d_S(g_n,vg_0)\,+\,d_S(vg_0,vg_0\sigma(g_0^{-1}vg_0))\,+\,
d_S(g_i\sigma(g_i^{-1}vg_i),g_i)
$$
$$
\leq\,C_{12}\cdot\vert\alpha_n\vert\,+\,n\cdot\ell_S(\langle v\rangle)\,+\,n\cdot C_{13}\,+\,2\underset{0\leq i\leq n}{Max}\,
\ell_S(g_i^{-1}vg_i)
$$
$$
\leq\,(C_{12}+2)\cdot\vert\alpha_n\vert\,+\,n\cdot\ell_S(\langle v\rangle)\,+\,n\cdot C_{13}
\eqno(6.39)
$$
because
$$
\ell_S(g_i^{-1}vg_i)\,=\,d_S(g_i,vg_i)\,\leq
$$
$$
\leq\,\underset{k=i}{\overset{n-1}{\sum}}\,d_S(g_k,g_{k+1})\,+\,d(g_n,vg_0)\,+\,\underset{j=0}{\overset{i-1}{\sum}}\,d_S(vg_k,vg_{k+1})\,=\,\vert\alpha_n\vert.
$$
{\bf Step 5:}\\
\\
According to Gromov, the centralizer $Z(v)$ of a torsion element in a word-hyperbolic group $\Gamma$ is a word-hyperbolic group itself and the restriction of a word metric on $\Gamma$ to $Z(v)$ is quasi-isometric to any word-metric on $Z(v)$. So if $S'$ is any finite symmetric set of generators of $Z(v)$ (any hyperbolic group is finitely generated), then there exist constants $C_{28},C_{29}$ such that
$$
C_{28}^{-1}\cdot\ell_S\vert_{Z(v)}-C_{29}\,\leq\,\ell_{S'}\,\leq\,C_{28}\cdot\ell_S\vert_{Z(v)}+C_{29}
\eqno(6.40)
$$
\\
{\bf Step 6:}\\
Step 3 and Step 4 show in conjunction with Lemma 6.4 that the previously defined chain maps extend to bounded morphisms
$$
\overline{j}_v:\,"\underset{\lambda\to 1^+}{\lim}"\,\overline{{\mathcal C}}_*(\ell^1_\lambda(Z(v)))_{\langle e\rangle}/Fil^N_{Hodge} \,\longrightarrow\,
"\underset{\lambda\to 1^+}{\lim}"\,\overline{{\mathcal C}}_*(\ell^1_\lambda(\Gamma))_{\langle v\rangle}/Fil^N_{Hodge} 
\eqno(6.41)
$$
and
$$
\overline{\kappa}_v:\,"\underset{\lambda\to 1^+}{\lim}"\,\overline{{\mathcal C}}_*(\ell^1_\lambda(\Gamma))_{\langle v\rangle}/Fil^N_{Hodge} \,\longrightarrow\,"\underset{\lambda\to 1^+}{\lim}"\,
\overline{{\mathcal C}}_*(\ell^1_\lambda(Z(v)))_{\langle e\rangle}/Fil^N_{Hodge} 
\eqno(6.42)
$$
of ind-complexes. Moreover $\overline{\kappa}_v\circ\overline{j}_v\,=\,id$ and $\overline{j}_v\circ\overline{\kappa}_v$ is chain homotopic to the identity via the bounded homotopy operator 
$h(\overline{j}_v\circ\overline{\kappa}_v,id)$. So (6.41) is in fact a chain homotopy equivalence of ind-complexes.\\
\\
{\bf Step 7:}
Recall that the inclusion $i_*^R:\,C_*^R(Z(v),{\mathbb C})\,\hookrightarrow C_*(Z(v),{\mathbb C})$ of the Rips-complex into the Bar-complex of the hyperbolic group $Z(v)$ is a chain homotopy equivalence, a chain homotopy inverse being given by the chain map $\Theta'_*$ of 4.7. The composition $\Theta'_*\circ i_*^R$ is a quasi-isomorphism of complexes of finitely generated free $Z(v)$-modules whereas 
$i_*^R\circ\Theta'_*$ is chain homotopic to the identity via the $Z(v)$-equivariant operator $h(\Theta',id)$ of 5.2, this time for the hyperbolic group $Z(v)$. Proposition 5.3 shows that the canonical morphism 
$$
"\underset{\lambda\to 1^+}{\lim}"\,\overline{C}_*^R(Z(v),{\mathbb C})_{Z(v)}\,\overset{i_R}{\longrightarrow} "\underset{\lambda\to 1^+}{\lim}"\,\overline{{\mathcal C}}_*(\ell^1_\lambda(Z(v)))\langle e\rangle/Fil^N_{Hodge} 
\eqno(6.43)
$$
is a chain homotopy equivalence of ind-complexes for $N>>0$. The Rips complex in question being a constant, finite dimensional ind complex calculating the homology of $Z(v)$ with complex coefficients for $R$ sufficiently large, we obtain finally canonical isomorphisms
$$
"\underset{\lambda\to 1^+}{\lim}"\,\overline{C}_*^R(Z(v),{\mathbb C})_{Z(v)}\,\overset{\simeq}{\longrightarrow}
\,\overline{C}_*^R(Z(v),{\mathbb C})_{Z(v)}\,\overset{\simeq}{\longrightarrow}\,H_*(Z(v),{\mathbb C})
\eqno(6.44)
$$
in $Ho(ind\,{\mathcal C})$. Altogether we have established the chain
$$
\begin{array}{c}
H_*(Z(v),{\mathbb C}) \\
\downarrow \\
 C_*^R(Z(v),{\mathbb C})_{Z(v)}  \\
\downarrow \\
"\underset{\lambda\to 1^+}{\lim}"\,{\mathcal C}_*(\ell^1_\lambda(Z(v)))\langle e\rangle/Fil^N_{Hodge}  \\
\downarrow \\
"\underset{\lambda\to 1^+}{\lim}"\,{\mathcal C}_*(\ell^1_\lambda(\Gamma))\langle v\rangle/Fil^N_{Hodge} \\
\downarrow \\
"\underset{\lambda\to 1^+}{\lim}"\,{\mathcal C}{\mathcal C}_*(\ell^1_\lambda(\Gamma))\langle v\rangle/Fil^N_{Hodge}
 \\
\uparrow \\
"\underset{\lambda\to 1^+}{\lim}"\,{\mathcal C}{\mathcal C}_*(\ell^1_\lambda(\Gamma))\langle v\rangle  \\
\end{array}
\eqno(6.45)
$$
 of isomorphisms in the homotopy category of ind-complexes. This is the assertion.
\end{proof}

\begin{prop}

For $\lambda>1$ let
${\mathcal C}{\mathcal C}_*(\ell^1_\lambda(\Gamma))_{inhom}$
be the closure of $CC_*({\mathbb C}\Gamma)_{inhom}$ in
${\mathcal C}{\mathcal C}_*(\ell^1_\lambda(\Gamma))$. 
Then the canonical morphism
$$
"\underset{\lambda\to 1^+}{\lim}"\,Fil_N^{Hodge}{\mathcal C}{\mathcal C}_*(\ell^1_\lambda(\Gamma))_{inhom}\,\longrightarrow\,
"\underset{\lambda\to 1^+}{\lim}"\,{\mathcal C}{\mathcal C}_*(\ell^1_\lambda(\Gamma))_{inhom}
\eqno(6.46)
$$
is a chain homotopy equivalence of ind-complexes for $N>>0$.
\end{prop}

\begin{proof}
The inhomogeneous part of the cyclic complex equals by definition the direct sum
$$
CC_*({\mathbb C}\Gamma)_{inhom}\,=\,\underset{\underset{\vert v\vert=\infty}{\langle v\rangle}}{\bigoplus}
\,CC_*({\mathbb C}\Gamma)_{\langle v\rangle}.
$$
As the Banach spaces $\ell^1_\lambda(\Gamma)$ are weighted $\ell^1$-spaces there is a similar decomposition
$$
{\mathcal C}{\mathcal C}_*(\ell^1_\lambda(\Gamma))_{inhom} \,\simeq\,
\underset{\underset{\vert v\vert=\infty}{\langle v\rangle}}{\bigoplus}
{\mathcal C}{\mathcal C}_*(\ell^1_\lambda(\Gamma))_{\langle v\rangle} 
\eqno(6.47)
$$
of the inhomogeneous part of ${\mathcal C}{\mathcal C}_*(\ell^1_\lambda(\Gamma))$ into a topological direct sum labeled by the conjugacy classes of elements of infinite order in $\Gamma$. 
Recall that for such a conjugacy class the canonical inclusion 
$Fil_{Hodge}^N\widehat{CC}_*({\mathbb C}\Gamma)_{\langle v\rangle}
\hookrightarrow\widehat{CC}_*({\mathbb C}\Gamma)_{\langle v\rangle}$ is a chain homotopy equivalence (Proposition 2.14). A look at the operators (2.78), (2.77) and (2.76) shows that the proof of Proposition 2.14 applies to our claim as well if the natural chain map of cyclic complexes
$$
{\mathcal C}_*^\lambda(\ell^1_{\lambda_0}(\Gamma))_{\langle v\rangle}\,\longrightarrow\,
{\mathcal C}_*^\lambda(\ell^1_{\lambda_1}(\Gamma))_{\langle v\rangle}
\eqno(6.48)
$$
is nullhomotopic in strictly positive degrees for $\lambda_0>\lambda_1>1$ via a homotopy operator 
whose norm is bounded uniformly for all hyperbolic conjugacy classes. Recall the diagram (2.73)
$$
C_*^{\lambda}({\mathbb C}\Gamma)_{\langle v\rangle}\,\longrightarrow\,
\widetilde{C}_*^{\lambda}({\mathbb C}\Gamma)_{\langle v\rangle}\,\overset{s_{v,\sigma,\sigma'}}{\longrightarrow}\,C_*(\Gamma,{\mathbb C}v)
\,\overset{I\circ p_v}{\longrightarrow}\,\widetilde{C}_*^{\lambda}({\mathbb C}\Gamma)_{\langle v\rangle}\,\longrightarrow\,C_*^{\lambda}({\mathbb C}\Gamma)_{\langle v\rangle}
$$
The projection operator $C_*^{\lambda}({\mathbb C}\Gamma)\to
\widetilde{C}_*^{\lambda}({\mathbb C}\Gamma)$ and its chain homotopy inverse, which is given in each degree by a universal (noncommutative) polynomial in the Hochschild operators $b,b'$ and the cyclic operator $T_*$, extend to bounded chain maps ${\mathcal C}_*^{\lambda}(\ell^1_{\lambda}(\Gamma))\to \widetilde{{\mathcal C}}_*^{\lambda}(\ell^1_{\lambda}(\Gamma))$ and
 $\widetilde{{\mathcal C}}_*^{\lambda}(\ell^1_{\lambda}(\Gamma))\to {\mathcal C}_*^{\lambda}(\ell^1_{\lambda}(\Gamma))$ for all $\lambda>1$. Consider now the subcomplex $C_*(\Gamma,{\mathbb C}v)$ of the Bar complex. It is contractible via the homotopy operator (2.29), but this operator is unbounded with respect to the norms we work with. Instead, we will use the operator 5.3, which defines an equivariant contracting homotopy of this complex in sufficiently high degree.
 So let $\lambda_0>\lambda_2>\lambda_1>1$. Then there exist constants $C_{22},\,C_{26}$ depending only on 
 $\Gamma, S,\delta,\lambda_0,\lambda_1,\lambda_2$, but not on $\langle v\rangle$, such that 
 $$
 \parallel s_{v,\sigma,\sigma'}\circ I\circ p_*(\alpha)\parallel_{\lambda_2}\,\leq\,C_{26}\parallel\alpha\parallel_{\lambda_0},\,\forall \alpha\in C_*(\Gamma,{\mathbb C}v)
 $$ 
 by 5.4 and 
 $$
 \parallel 
 \widetilde{\nabla}(\beta)\parallel_{\lambda_1}\,\leq\,C_{22}\parallel\beta\parallel_{\lambda_2},
 \forall\beta\in C_*(\Gamma,{\mathbb C}v).
 $$ 
 by 5.3, so that
$$
\widetilde{C}_*^{\lambda}({\mathbb C}\Gamma)_{\langle v\rangle}\,\overset{s_{v,\sigma,\sigma'}}{\to}\,C_*(\Gamma,{\mathbb C}v)\,\overset{\widetilde{\nabla}}{\to}\,C_{*+1}(\Gamma,{\mathbb C}v)\,\overset{I\circ p_v}{\to}
 \widetilde{C}_{*+1}^{\lambda}({\mathbb C}\Gamma)_{\langle v\rangle}
$$ 
 extends to a bounded operator
 $$
 I\circ p_v\circ\widetilde{\nabla}\circ s_{v,\sigma,\sigma'}:
 \widetilde{{\mathcal C}}_*^{\lambda}(\ell^1_{\lambda_0}(\Gamma))_{\langle v\rangle}\to\widetilde{{\mathcal C}}_{*+1}^{\lambda}(\ell^1_{\lambda_1}(\Gamma))_{\langle v\rangle},
 \eqno(6.49)
 $$
whose norm is bounded uniformly with respect to $\langle v\rangle$. It defines a nullhomotopy of the canonical chain map (6.48). 
\end{proof}

\section{Main results}

\begin{theorem}
Let $\Gamma$ be a word-hyperbolic group. Then there is a canonical isomorphism
$$
H_*(\Gamma,{\mathbb C}\Gamma_{tors}) \, \overset{\simeq}{\longrightarrow}\, "\underset{\lambda\to 1^+}{\lim}"\,{\mathcal C}{\mathcal C}_*(\ell^1_\lambda(\Gamma))
\eqno(7.1)
$$
in the homotopy category $Ho(ind\,{\mathcal C})$ of ${\mathbb Z}/2{\mathbb Z}$-graded ind-complexes of Fr\'echet spaces.
\end{theorem}

\begin{proof}
Let 
$$
"\underset{\lambda\to 1^+}{\lim}"\,{\mathcal C}{\mathcal C}_*(\ell^1_\lambda(\Gamma))\,\overset{\simeq}{\to}\,
\underset{\underset{\vert v\vert<\infty}{\langle v\rangle}}{\bigoplus}"\underset{\lambda\to 1^+}{\lim}"\,{\mathcal C}{\mathcal C}_*(\ell^1_\lambda(\Gamma))_{\langle v\rangle}\,
\oplus\,"\underset{\lambda\to 1^+}{\lim}"\,{\mathcal C}{\mathcal C}_*(\ell^1_\lambda(\Gamma))_{inhom}
$$
be the homogeneous decomposition of the ind-complex into the finite topological direct sum of the contributions of the finitely 
many elliptic conjugacy classes and the inhomogeneous part. Fix an elliptic conjugacy class $\langle v\rangle$. By proposition 6.6 there is a canonical isomorphism
$$
H_*(Z(v),{\mathbb C})\,\overset{\sim}{\to}"\underset{\lambda\to 1^+}{\lim}"\,{\mathcal C}{\mathcal C}_*(\ell^1_\lambda(\Gamma))_{\langle v\rangle}
$$ 
in the homotopy category $Ho(ind\,{\mathcal C})$, where the left hand side, i.e. the homology of the centralizer of $v$ with its grading into even and odd parts, is viewed as a constant, finite dimensional ind-complex with zero differentials. The Shapiro lemma tells us that
$$
H_*(Z(v),{\mathbb C})\,\overset{\simeq}{\to}H_*(\Gamma,\,Ind_{Z(v)}^{\Gamma}{\mathbb C})\,\overset{\simeq}{\to}H_*(\Gamma,{\mathbb C}\langle v\rangle),
\eqno(7.2)
$$
so that we obtain a canonical isomorphism 
$$
H_*(\Gamma,{\mathbb C}\Gamma_{tors})\,\overset{\simeq}{\to}\underset{\underset{\vert v\vert<\infty}{\langle v\rangle}}{\bigoplus}H_*(\Gamma,{\mathbb C}\langle v\rangle)
\overset{\simeq}{\to}\underset{\underset{\vert v\vert<\infty}{\langle v\rangle}}{\bigoplus}"\underset{\lambda\to 1^+}{\lim}"\,{\mathcal C}{\mathcal C}_*(\ell^1_\lambda(\Gamma))_{\langle v\rangle}
\eqno(7.3)
$$
in $Ho(ind\,{\mathcal C})$. For the inhomogeneous part we learn from 6.7 that the identity morphism factors in $Ho(ind\,{\mathcal C})$ as
$$
id:\,"\underset{\lambda\to 1^+}{\lim}"\,{\mathcal C}{\mathcal C}_*(\ell^1_\lambda(\Gamma))_{inhom}\to
\,"\underset{\lambda\to 1^+}{\lim}"\,Fil_{Hodge}^N{\mathcal C}{\mathcal C}_*(\ell^1_\lambda(\Gamma))_{inhom}\to
\,"\underset{\lambda\to 1^+}{\lim}"\,{\mathcal C}{\mathcal C}_*(\ell^1_\lambda(\Gamma))_{inhom}
\eqno(7.4)
$$
for $N>>0$ large enough. As the ind-complex $"\underset{\lambda\to 1^+}{\lim}"\,Fil_{Hodge}^N{\mathcal C}{\mathcal C}_*(\ell^1_\lambda(\Gamma))_{inhom}$ is contractible by 6.5, the same holds for $"\underset{\lambda\to 1^+}{\lim}"\,{\mathcal C}{\mathcal C}_*(\ell^1_\lambda(\Gamma))_{inhom}$ and completes the proof of the theorem.
\end{proof}

\begin{theorem}
Let $\Gamma$ be a word-hyperbolic group. Then there are canonical isomorphisms
$$
H_*(\Gamma,{\mathbb C}\Gamma_{tors})
\overset{\simeq}{\longrightarrow}
\underset{\lambda\to 1^+}{\lim}\,HP^{cont}_*(\ell^1_\lambda(\Gamma))
\eqno(7.5)
$$
and
$$
H^*(\Gamma,{\mathbb C}\Gamma_{tors})
\overset{\simeq}{\longrightarrow}
\underset{1^+\leftarrow\lambda}{\lim}\,HP_{cont}^*(\ell^1_\lambda(\Gamma))
\eqno(7.6)
$$
where $HP^{cont}$ denotes continuous periodic cyclic (co)homology.
\end{theorem}

\begin{proof}
For $N>>0$ and $\lambda>1$ we consider the commutative diagram
$$
\begin{array}{ccc}
{\mathcal C}{\mathcal C}_*(\ell^1_\lambda(\Gamma))   & \longrightarrow & \widehat{CC}^{cont}_*(\ell^1_\lambda(\Gamma)) \\
 & & \\
 \downarrow & & \downarrow \\
 & & \\
 {\mathcal C}{\mathcal C}_*/Fil_{Hodge}^N{\mathcal C}{\mathcal C}_*(\ell^1_\lambda(\Gamma))    & \overset{\simeq}{\longrightarrow} &  \widehat{CC}^{cont}_*/Fil_{Hodge}^N\widehat{CC}^{cont}_*(\ell^1_\lambda(\Gamma)) \\
 \end{array}
\eqno(7.7)
$$
of ind-complexes (the r.h.s. is a constant ind-complex) and note that its bottom horizontal morphism is actually an isomorphism.
Letting $\lambda$ tend to 1 formally one obtains
$$
\begin{array}{ccc}
 "\underset{\lambda\to 1^+}{\lim}"\,{\mathcal C}{\mathcal C}_*(\ell^1_\lambda(\Gamma))  & \longrightarrow &     "\underset{\lambda\to 1^+}{\lim}"\,\widehat{CC}^{cont}_*(\ell^1_\lambda(\Gamma))  \\
 & & \\
 \downarrow & & \downarrow \\
 & & \\
"\underset{\lambda\to 1^+}{\lim}"\,{\mathcal C}{\mathcal C}_*/Fil_{Hodge}^N{\mathcal C}{\mathcal C}_*(\ell^1_\lambda(\Gamma))      & \overset{\simeq}{\longrightarrow} & "\underset{\lambda\to 1^+}{\lim}"\,\widehat{CC}^{cont}_*/Fil_{Hodge}^N\widehat{CC}^{cont}_*(\ell^1_\lambda(\Gamma)) \\
\end{array}
\eqno(7.8)
$$
Its  left vertical arrow is a chain homotopy equivalence according to 6.5. But the proof of this proposition applies verbatim to the continuous periodic cyclic bicomplex so that 
the right vertical arrow in the previous diagram is a chain homotopy equivalence as well. Thus the canonical morphism
$$
"\underset{\lambda\to 1^+}{\lim}"\,{\mathcal C}{\mathcal C}_*(\ell^1_\lambda(\Gamma)) \, \longrightarrow \,"\underset{\lambda\to 1_+}{\lim}"\,\widehat{CC}^{cont}_*(\ell^1_\lambda(\Gamma))  
\eqno(7.9)
$$
is an isomorphism in the homotopy category of ind-complexes. Theorem 7.1 implies that the composite morphism
$$
H_*(\Gamma,{\mathbb C}\Gamma_{tors}) \, \longrightarrow \,"\underset{\lambda\to 1^+}{\lim}"\,\widehat{CC}^{cont}_*(\ell^1_\lambda(\Gamma))  
\eqno(7.10)
$$
is an isomorphism in $Ho(ind\,{\mathcal C})$ as well. Taking homology termwise and passing from formal to true inductive limits in the category of abstract vector spaces yields finally the isomorphism
$$
H_*(\Gamma,{\mathbb C}\Gamma_{tors}) \, \overset{\simeq}{\longrightarrow} \,\underset{\lambda\to 1^+}{\lim}\,HP^{cont}_*(\ell^1_\lambda(\Gamma))  
\eqno(7.11)
$$
of ind-vector spaces. Applying the dual argument to the diagram of topologically dual pro-complexes proves the corresponding cohomological assertion.
\end{proof}

Note that our methods provide no information whatsoever about the individual groups $HP^{cont}(\ell^1_\lambda(\Gamma)),\,\lambda\geq 1$.

\begin{theorem}
Let $\Gamma$ be a word-hyperbolic group. Then there is a canonical isomorphism
$$
H^*(\Gamma,{\mathbb C}\Gamma_{tors})
\overset{\simeq}{\longrightarrow}
\underset{1^+\leftarrow\lambda}{\lim}\,HC_\epsilon^*(\ell^1_\lambda(\Gamma))
\eqno(7.12)
$$
where $HC_\epsilon$ denotes entire cyclic (co)homology \cite{Co2}.
\end{theorem}

\begin{proof}
Fix $\lambda>1$. By definition, the entire cyclic bicomplex $CC_\epsilon^*(\ell^1_\lambda(\Gamma))$ of $\ell^1_\lambda(\Gamma)$ is the cochain complex of those linear functionals on $CC_*({\mathbb C}\Gamma)\,=\,(\Omega({\mathbb C}\Gamma),\,b+B)$, which are with respect to the seminorms $\parallel -\parallel_{(\lambda,\varrho,m)},\,\varrho>1,m\in{\mathbb N}$ (see \cite{Co2}). A look at the proof 
of proposition (6.5) shows then that the canonical chain map
$$
\underset{1^+\leftarrow\lambda}{\lim}\,H^*(CC_\epsilon^*(\ell^1_\lambda(\Gamma))/Fil^N_{Hodge}CC_\epsilon^*(\ell^1_\lambda(\Gamma)))\,\overset{\simeq}{\longrightarrow}\,\underset{1^+\leftarrow\lambda}{\lim}\,HC_\epsilon^*(\ell^1_\lambda(\Gamma))
$$
is an isomorphism for $N>>0$. As the complexes
$$
CC_\epsilon^*(\ell^1_\lambda(\Gamma))/Fil^N_{Hodge}CC_\epsilon^*(\ell^1_\lambda(\Gamma))\,\overset{\simeq}{\longrightarrow}\,
\widehat{CC}_{cont}^*(\ell^1_\lambda(\Gamma))/Fil^N_{Hodge}\widehat{CC}_{cont}^*(\ell^1_\lambda(\Gamma))
$$
coincide, we deduce from the fact that all morphisms in (7.8) are isomorphisms that the canonical morphism
$$
\underset{1^+\leftarrow\lambda}{\lim}\,HP_{cont}^*(\ell^1_\lambda(\Gamma)) \,\overset{\simeq}{\longrightarrow}\  \underset{1^+\leftarrow\lambda}{\lim}\,HC_\epsilon^*(\ell^1_\lambda(\Gamma))
$$
is an isomorphism as well. The assertion follows then from theorem 7.2.
\end{proof}

We turn now to local cyclic (co)homology where much more precise results may be obtained.

\begin{theorem}
Let $\Gamma$ be a word-hyperbolic group. Then there is a canonical isomorphism
$$
H_*(\Gamma,{\mathbb C}\Gamma_{tors})
\,\overset{\simeq}{\longrightarrow}\,{\mathcal C}{\mathcal C}^\omega_*(\ell^1(\Gamma))
\eqno(7.13)
$$
between the homology of $\Gamma$ with coefficients in ${\mathbb C}\Gamma_{tors}$ and the analytic cyclic ind-bicomplex of $\ell^1(\Gamma)$ 
in the derived ind-category ${\mathcal D}(ind\,{\mathcal C})$.
\end{theorem}

\begin{proof}
This follows immediately from 7.1 and 6.2.
\end{proof}

As an immediate consequence one obtains

\begin{theorem}
Let $\Gamma$ be a word-hyperbolic group. Then there are canonical isomorphisms
$$
H_*(\Gamma,{\mathbb C}\Gamma_{tors})
\overset{\simeq}{\longrightarrow}
HC^{loc}_*(\ell^1(\Gamma))
\eqno(7.14)
$$
and
$$
H^*(\Gamma,{\mathbb C}\Gamma_{tors})
\overset{\simeq}{\longrightarrow}
HC_{loc}^*(\ell^1(\Gamma))
\eqno(7.15)
$$
identifying the local cyclic (co)homology groups of $\ell^1(\Gamma)$.
\end{theorem}

More generally one has 

\begin{theorem}
For any word-hyperbolic groups $\Gamma,\,\Gamma'$ and Banach algebras $A,B$ there is a canonical isomorphism
$$
HC^*_{loc}(\ell^1(\Gamma)\otimes_\pi A,\ell^1(\Gamma')\otimes_\pi B)\,\overset{\simeq}{\longrightarrow}\,Hom(H_*(\Gamma,{\mathbb C}\Gamma_{tors}),H_*(\Gamma',{\mathbb C}\Gamma'_{tors}))\otimes HC^*_{loc}(A,B)
\eqno(7.16)
$$
of bivariant local cyclic (co)homology groups, which is compatible with composition products.
\end{theorem}

\begin{proof}
Taking $A=B={\mathbb C}$ and $\Gamma=1$ or $\Gamma'=1$, one obtains Theorem 7.4 as a particular case of 7.5. Recall that the derived ind-category ${\mathcal D}(ind\,{\mathcal C})$ is obtained from the homotopy category $Ho(ind\,{\mathcal C})$ of ind-complexes by inverting all morphisms with weakly contractible mapping cone \cite{Pu1}, \cite{KS}. An ind-complex ${\mathcal C}\,=\,"\underset{\underset{I}{\longrightarrow}}{\lim}"\,C_*^{(i)}$ is weakly contractible if any chain map from a constant ind-complex to $\mathcal C$ is nullhomotomic. The null-system \cite{KS} of weekly contractible ind-complexes is stable under taking projective tensor products with auxiliary ind-complexes. It follows that the class of isomorphisms in the derived ind-category is stable under projective tensor products as well. In particular, the 
canonical morphism
$$
H_*(\Gamma,{\mathbb C}\Gamma_{tors})\otimes{\mathcal C}{\mathcal C}^{\omega}_*(A)\,\overset{\simeq}{\longrightarrow}\,{\mathcal C}{\mathcal C}^{\omega}_*(\ell^1(\Gamma))\otimes_{\pi}{\mathcal C}{\mathcal C}^{\omega}_*(A)
\eqno(7.17)
$$
is an isomorphism in ${\mathcal D}(ind\,{\mathcal C})$  for any hyperbolic group $\Gamma$ and any unital Banach algebra $A$. Recall that there is a natural chain homotopy 
equivalence of ind-complexes (an Eilenberg-Zilber quasi-isomorphism) 
$$
{\mathcal C}{\mathcal C}^{\omega}_*(A)\otimes_{\pi}{\mathcal C}{\mathcal C}^{\omega}_*(B)\,\overset{\simeq}{\longrightarrow}\,{\mathcal C}{\mathcal C}^{\omega}_*(A\otimes_{\pi}B)
\eqno(7.18)
$$
for unital Banach algebras $A$ and $B$ \cite{Pu3}, \cite{Pu1}. Altogether one obtains a natural isomorphism
$$
H_*(\Gamma,{\mathbb C}\Gamma_{tors})\otimes{\mathcal C}{\mathcal C}^{\omega}_*(A)\,\overset{\simeq}{\longrightarrow}\,{\mathcal C}{\mathcal C}^{\omega}_*(\ell^1(\Gamma))\otimes_{\pi}{\mathcal C}{\mathcal C}^{\omega}_*(A)
\,\overset{\simeq}{\longrightarrow}\,{\mathcal C}{\mathcal C}^{\omega}_*(\ell^1(\Gamma)\otimes_{\pi}A)
\eqno(7.19)
$$
in the derived ind-category. We are now ready to calculate the bivariant local cyclic cohomology group in question.
$$
HC_*^{loc}(\ell^1(\Gamma)\otimes_{\pi}A,\ell^1(\Gamma')\otimes_{\pi}B)\,=\,
Mor_{ind\,{\mathcal D}}({\mathcal C}{\mathcal C}^{\omega}_*(\ell^1(\Gamma)\otimes_{\pi}A),{\mathcal C}{\mathcal C}^{\omega}_*(\ell^1(\Gamma')\otimes_{\pi}B))
$$
by the very definition of local cyclic cohomology
$$
\simeq\,Mor_{ind\,{\mathcal D}}(H_*(\Gamma,{\mathbb C}\Gamma_{tors})\otimes{\mathcal C}{\mathcal C}^{\omega}_*(A),H_*(\Gamma',{\mathbb C}\Gamma'_{tors})\otimes{\mathcal C}{\mathcal C}^{\omega}_*(B))
$$
by the previous considerations
$$
\simeq\,Hom(H_*(\Gamma,{\mathbb C}\Gamma_{tors}),H_*(\Gamma',{\mathbb C}\Gamma'_{tors}))\,\otimes\,
Mor_{ind\,{\mathcal D}}({\mathcal C}{\mathcal C}^{\omega}_*(A),{\mathcal C}{\mathcal C}^{\omega}_*(B))
$$
because $H_*(\Gamma,{\mathbb C}\Gamma_{tors})$ and $H_*(\Gamma',{\mathbb C}\Gamma'_{tors})$ are finite dimensional and $ind\,{\mathcal D}$ is additive
$$
\simeq\,Hom(H_*(\Gamma,{\mathbb C}\Gamma_{tors}),H_*(\Gamma',{\mathbb C}\Gamma'_{tors}))\,\otimes\,HC_*^{loc}(A,B).
\eqno(7.20)
$$
This shows the assertion if $A$ and $B$ are unital. The general case follows by adjoining units to the algebras in question and by applying excision in local cyclic cohomology \cite{Pu2}.

\end{proof}

\end{document}